\documentclass[10pt]{amsart}

\usepackage{latexsym,exscale,enumerate,amsfonts,amssymb, ulem, xparse, mathtools}
\usepackage{amsmath,amsthm,amsfonts,amssymb,amscd, stmaryrd,textcomp}
\usepackage{hyperref}
\usepackage{ulem}
\usepackage{bbold}
\usepackage[usenames,dvipsnames]{xcolor}
\colorlet{green}{black!30!green} % <------------------------------this redefines green to not be so bright and harsh!

\usepackage[all]{xy}
\SelectTips{cm}{}
\usepackage{graphicx}

\usepackage{tikz}
\usetikzlibrary{decorations.markings}
\usetikzlibrary{decorations.pathreplacing}
\usetikzlibrary{arrows,shapes,positioning}
\tikzstyle directed=[postaction={decorate,decoration={markings,
    mark=at position #1 with {\arrow{>}}}}]
\tikzstyle rdirected=[postaction={decorate,decoration={markings,
    mark=at position #1 with {\arrow{<}}}}]
\tikzset{anchorbase/.style={baseline={([yshift=-0.5ex]current bounding box.center)}}}

\newcommand{\onenn}[1]{{\mathbf 1}_{#1}}

\newcommand{\onel}{{\mathbf 1}_{\lambda}}
\newcommand{\onell}[1]{{\mathbf 1}_{#1}}

\newcommand\sE{{\cal{E}}}
\newcommand\sF{{\cal{F}}}
\def\cal#1{\mathcal{#1}}%

\def\1{\mathbbm{1}}%

\def\l{\lambda}

\newcommand{\END}{{\rm END}}

\newcommand{\UU}{{\bf U}}

\newcommand{\U}{\dot{{\bf U}}}
\newcommand{\Ucat}{\cal{U}}
\newcommand{\Ucatc}{\check{\cal{U}}}

\newcommand{\UcatD}{\dot{\cal{U}}}

\newcommand{\sln}{\mf{sl}_n}
\newcommand{\slm}{\mf{sl}_m}
\newcommand{\slnn}[1]{\mf{sl}_{#1}}
\newcommand{\gln}{\mf{gl}_n}
\newcommand{\glm}{\mf{gl}_m}
\newcommand{\glnn}[1]{\mf{gl}_{#1}}

\newcommand{\Foam}[2]{#1\cat{Foam}_{#2}}

\newcommand{\Kom}{\operatorname{Kom}}
\newcommand{\Tr}{\operatorname{Tr}}
\newcommand{\vTr}{\operatorname{vTr}}
\newcommand{\hTr}{\operatorname{hTr}}
\newcommand{\grvTr}{\widetilde{\vTr}}
\newcommand{\Rep}{\cat{Rep}}
\newcommand{\grRep}{\cat{grRep}}
\newcommand{\SH}{\operatorname{SH}}
\newcommand{\SHz}{\operatorname{SH}_{t=0}}

\newcommand{\BN}[2]{{_{#1}\mathcal{BN}_{#2}}}
\newcommand{\AFoam}[1][2]{#1\cal{A}\cat{Foam}}
\newcommand{\saKh}{\mathsf{saKh}}
\newcommand{\saKhR}{\mathsf{saKhR}}
\newcommand{\Kh}{\mathsf{Kh}}
\newcommand{\KhR}{\mathsf{KhR}}
\newcommand{\AWeb}[1][n]{#1\cal{A}\cat{Web}}
\newcommand{\refequal}[1]{\xy {\ar@{=}^{#1}
(-1,0)*{};(1,0)*{}};
\endxy}
\newcommand{\cat}[1]{\ensuremath{\mbox{\bfseries {\upshape {#1}}}}}

\newcommand{\Hom}{{\rm Hom}}
\newcommand{\HOM}{{\rm HOM}}
\renewcommand{\to}{\rightarrow}

\newcommand{\id}{{\rm id}}

\newcommand{\End}{{\rm End}}

\def\mf{\mathfrak}
\def\shuffle{\,\raise 1pt\hbox{$\scriptscriptstyle\cup{\mskip
               -4mu}\cup$}\,}

\theoremstyle{definition}
\newtheorem{thm}{Theorem}[section]
\newtheorem{cor}[thm]{Corollary}
\newtheorem{conj}[thm]{Conjecture}
\newtheorem{lem}[thm]{Lemma}
\newtheorem{rem}[thm]{Remark}
\newtheorem{prop}[thm]{Proposition}
\newtheorem{defn}[thm]{Definition}

\numberwithin{equation}{section}

\def\emph#1{{\sl #1\/}}
\def\ie{{\sl i.e. \/}}
\def\eg{{\sl e.g. \/}}

\let\tilde=\widetilde

\let\phi=\varphi
\let\theta=\vartheta
\let\epsilon=\varepsilon

\usepackage{bbm}
\def\C{{\mathbbm C}}
\def\N{{\mathbbm N}}
\def\R{{\mathbbm R}}
\def\Z{{\mathbbm Z}}

\def\cal#1{\mathcal{#1}}%
\def\1{\mathbbm{1}}%
\newcommand{\dotsheet}[2][1]{
\xy
(0,0)*{
\begin{tikzpicture} [fill opacity=0.2,decoration={markings, mark=at position 0.6 with {\arrow{>}}; }, scale=#1]
	\draw [very thick, fill=red] (1,1) -- (-1,2) -- (-1,-1) -- (1,-2) -- cycle;
	\node [opacity=1] at (0,0) {$\bullet^{#2}$};
\end{tikzpicture}};
\endxy}

\newcommand{\sphere}[1][1]{
\xy
(0,0)*{
\begin{tikzpicture} [fill opacity=0.2, scale=#1]
	\path [fill=red] (0,0) circle (1);
	\draw (-1,0) .. controls (-1,-.4) and (1,-.4) .. (1,0);
	\draw[dashed] (-1,0) .. controls (-1,.4) and (1,.4) .. (1,0);
	\draw[very thick] (0,0) circle (1);
\end{tikzpicture}};
\endxy}

\newcommand{\dottedsphere}[2][1]{
\xy
(0,0)*{
\begin{tikzpicture} [fill opacity=0.2, scale=#1]
	\path [fill=red] (0,0) circle (1);
	\draw (-1,0) .. controls (-1,-.4) and (1,-.4) .. (1,0);
	\draw[dashed] (-1,0) .. controls (-1,.4) and (1,.4) .. (1,0);
	\draw[very thick] (0,0) circle (1);
	\node [opacity=1]  at (0,0.6) {$\bullet$};
	\node [opacity=1] at (-0.3,0.6) {$#2$};
\end{tikzpicture}};
\endxy}

\newcommand{\cylinder}[1][1]{
\xy
(0,0)*{
\begin{tikzpicture} [fill opacity=0.2,  decoration={markings, mark=at position 0.6 with {\arrow{>}};  }, scale=#1]
	\path [fill=red] (0,4) ellipse (1 and 0.5);
	\path [fill=red] (0,0) ellipse (1 and 0.5);
	\path[fill=red, opacity=.3] (-1,4) .. controls (-1,3.34) and (1,3.34) .. (1,4) --
		(1,0) .. controls (1,.66) and (-1,.66) .. (-1,0) -- cycle;
	\draw [very thick] (0,4) ellipse (1 and 0.5);
	\draw [very thick] (0,0) ellipse (1 and 0.5);
	\draw[very thick] (1,4) -- (1,0);
	\draw[very thick] (-1,4) -- (-1,0);
\end{tikzpicture}};
\endxy
}

\newcommand{\slthncfour}[1][1]{
\xy
(0,0)*{
\begin{tikzpicture} [fill opacity=0.2,  decoration={markings, mark=at position 0.6 with {\arrow{>}};  }, scale=#1]
	\path [fill=red] (0,4) ellipse (1 and 0.5);
	\path [fill=red, opacity=.3] (1,4) .. controls (1,3.34) and (-1,3.34) .. (-1,4) --
		(-1,4) .. controls (-1,2) and (1,2) .. (1,4);
	\path [fill=red] (0,0) ellipse (1 and 0.5);
	\path [fill=red, opacity=.3] (1,0) .. controls (1,.66) and (-1,.66) .. (-1,0) --
		(-1,0) .. controls (-1,2) and (1,2) .. (1,0);
	\draw [very thick] (0,4) ellipse (1 and 0.5);
	\draw [very thick] (0,0) ellipse (1 and 0.5);
	\draw[very thick] (-1,0) .. controls (-1,2) and (1,2) .. (1,0);
	\draw [very thick] (-1,4) .. controls (-1,2) and (1,2) .. (1,4);
	\node[opacity=1] at (0,1) {$\bullet$};
\end{tikzpicture}};
\endxy
}

\newcommand{\slthncfive}[1][1]{
\xy
(0,0)*{
\begin{tikzpicture} [fill opacity=0.2,  decoration={markings, mark=at position 0.6 with {\arrow{>}};  }, scale=#1]
	\path [fill=red] (0,4) ellipse (1 and 0.5);
	\path [fill=red, opacity=.3] (1,4) .. controls (1,3.34) and (-1,3.34) .. (-1,4) --
		(-1,4) .. controls (-1,2) and (1,2) .. (1,4);
	\path [fill=red] (0,0) ellipse (1 and 0.5);
	\path [fill=red, opacity=.3] (1,0) .. controls (1,.66) and (-1,.66) .. (-1,0) --
		(-1,0) .. controls (-1,2) and (1,2) .. (1,0);
	\draw [very thick] (0,4) ellipse (1 and 0.5);
	\draw [very thick] (0,0) ellipse (1 and 0.5);
	\draw[very thick] (-1,0) .. controls (-1,2) and (1,2) .. (1,0);
	\draw [very thick] (-1,4) .. controls (-1,2) and (1,2) .. (1,4);
	\node[opacity=1] at (0,3) {$\bullet$};
\end{tikzpicture}};
\endxy
}

\newcommand{\capFEfoam}[3][.5]{
\xy
(0,0)*{
\begin{tikzpicture} [scale=#1,fill opacity=0.2]
\path[fill=red] (2.5,2) to (-2.5,2) to (-2.5,-1) to (2.5,-1);
\path[fill=blue] (1.5,-2) to [out=90,in=0] (0,.5) to [out=180,in=90] (-1.5,-2) to
	(-.5,-1) to [out=90,in=180] (0,-.25) to [out=0,in=90] (.5,-1);
\path[fill=red] (2,1) to (-2,1) to (-2,-2) to (2,-2);
\draw[very thick, directed=.5] (2.5,2) to (-2.5,2);
\draw[very thick] (2.5,2) to (2.5,-1);
\draw[very thick] (-2.5,2) to (-2.5,-1);
\draw[very thick, directed=.5] (2.5,-1) to (-2.5,-1);
\draw[very thick, red, directed=.75] (-.5,-1) to [out=90,in=180] (0,-.25)
	to [out=0,in=90] (.5,-1);
\node[red, opacity=1] at (1.75,1.5) {$_{#2}$};
\draw[very thick, directed=.5] (1.5,-2) to (.5,-1);
\draw[very thick, directed=.5] (-.5,-1) to (-1.5,-2);
\draw[very thick, directed=.5] (2,1) to (-2,1);
\draw[very thick] (2,1) to (2,-2);
\draw[very thick] (-2,1) to (-2,-2);
\draw[very thick, directed=.5] (2,-2) to (-2,-2);
\draw[very thick, red, directed=.65] (1.5,-2) to [out=90,in=0] (0,.5)
	to [out=180,in=90] (-1.5,-2);
\node[red, opacity=1] at (1.25,.5) {$_{#3}$};
\end{tikzpicture}};
\endxy
}

\newcommand{\curm}{\UU(\slnn{m}[t])}
\newcommand{\dcurm}{{\U(\slnn{m}[t])}}

\begin{document}
%
% ==============================================================================

\title[Sutured annular Khovanov-Rozansky homology]
{
Sutured annular Khovanov-Rozansky homology
}

\author{Hoel Queffelec}
\address{Mathematical Sciences Institute, Australian National University, Canberra, ACT 0200, Australia}
\email{hoel.queffelec@anu.edu.au}

 \author{David E. V. Rose}
 \address{Department of Mathematics, University of Southern California, Los Angeles, CA 90089, USA}
 \email{davidero@usc.edu}

\begin{abstract}
We introduce an $\sln$ homology theory for knots and links in the thickened annulus.
To do so, we first give a fresh perspective on sutured annular Khovanov homology, 
showing that its definition follows naturally from trace decategorifications of enhanced $\slnn{2}$ foams and categorified quantum $\glm$, via classical skew Howe duality. 
This framework then extends to give our annular $\sln$ link homology theory, which we call sutured annular Khovanov-Rozansky homology.
We show that the $\sln$ sutured annular Khovanov-Rozansky homology of an annular link carries an action of the Lie algebra $\sln$, 
which in the $n=2$ case recovers a result of Grigsby-Licata-Wehrli.
\end{abstract}

\maketitle
%\setcounter{tocdepth}{3}
%\setcounter{tocdepth}{2}
%\tableofcontents

% ###############################################################################
%
\section{Introduction}
%
% ###############################################################################

Following Khovanov's categorification of the Jones polynomial of knots and links in $S^3$ \cite{Kh1,Kh2,Kh6}, 
Asaeda, Przytycki, and Sikora \cite{APS} introduced a homology theory for links in $\Sigma \times [0,1]$ 
categorifying the Kauffman skein module of the surface $\Sigma$. 
In the case that $\Sigma = \cal{A} := S^1 \times [0,1]$, this invariant has been extensively studied by Roberts \cite{Rob} and Grigsby-Wehrli \cite{GW}, 
and dubbed sutured annular Khovanov homology ($\saKh$), due to its relation with sutured Floer homology.
In recent work \cite{GLW}, Grigsby, Licata, and Wehrli have shown that $\saKh$ is not only useful topologically, 
but is also of representation-theoretic interest. Motivated by a conjecture relating this invariant to symplectic geometry, 
they show that the $\saKh$ of a link carries an action of $\slnn{2}$.

After Khovanov's initial work on link homology, he and Rozansky constructed a homology theory for knots and links in $S^3$ categorifying 
the $\sln$ Reshetikhin-Turaev knot polynomials \cite{KhR}. However, there has not been a construction extending Khovanov-Rozansky homology 
to knots and links in other 3-manifolds.

In the present work, we give the first such extension, constructing an $\sln$ homology theory for knots and links in $\cal{A} \times [0,1]$ which categorifies 
the $\sln$ skein module of the annulus. 
To do so, we first give a new construction in the $n=2$ case, showing that $\saKh$ follows from trace decategorifications of the Khovanov-Lauda 
categorified quantum group $\cal{U}_Q(\glm)$ \cite{KL, KL2, KL3} (see also independent work of Rouquier \cite{Rou2}) and the enhanced $\slnn{2}$ foam category introduced 
by Blanchet \cite{Blan} and studied by the authors (joint with Lauda) in \cite{LQR1}. 
The extension to the $\sln$ case is natural from this point of view, and we obtain an $\sln$ homology for colored links in the thickened annulus, which we call 
sutured annular Khovanov-Rozansky homology ($\saKhR$). 
A direct consequence of our construction is that the $\saKhR$ of an annular link carries an action of $\sln$, 
which specializes in the $n=2$ case to recover the aforementioned result of Grigsby-Licata-Wehrli.

We'll now quickly summarize our construction in the $n=2$ case, where the structures are most familiar. 
The definitions, details, and extension to general $n$ are saved for the remainder of the paper.
In \cite{LQR1}, we showed that there is a skew Howe $2$-functor $\cal{U}_Q(\glm) \rightarrow 2\cat{Foam}$, where the latter is the enhanced (Blanchet) 
version of Bar-Natan's quotient of the $2$-category of planar tangles and cobordisms. 
Given a link $\cal{L} \subset S^3$, we can assign to it a complex 
$C(\cal{L})$ in $2\cat{Foam}$, from which its Khovanov homology can be obtained. Moreover, this $2$-functor factors through 
the quotient $\cal{U}_Q(\glm)^{0 \leq 2}$, in which we've killed $\glm$ weights with entries not in the set $\{0,1,2\}$.

We rederive $\saKh$ by applying trace decategorifications to the $2$-functor $\cal{U}_Q(\glm)^{0 \leq 2} \to 2\cat{Foam}$. 
Given a graded 2-category $\cal{C}$, we consider both its horizontal trace $\hTr(\cal{C})$ and a graded version of its 
vertical trace  $\grvTr(\cal{C})$, which includes into the former as a full subcategory.
Since both of these operations are functorial, we obtain the following diagram:
\begin{equation}\label{maindiag}
\begin{gathered}
\xymatrix{
\grvTr(\cal{U}_Q(\glm)^{0 \leq 2}) \ar[rr] \ar@{^{(}->}[d] & & \grvTr(2\cat{Foam}) \ar@{^{(}->}[d] \\
\hTr(\cal{U}_Q(\glm)^{0 \leq 2}) \ar[rr] & & \hTr(2\cat{Foam}).
} 
\end{gathered}
\end{equation}
The category in the bottom right is an enhanced version of the Bar-Natan category of the annulus. 
Using Rickard complexes \cite{CR} in $\cal{U}_Q(\glm)$, 
we assign a complex $C(\cal{L})$ in this category to any link $\cal{L} \subseteq \cal{A}\times [0,1]$. 
This complex lifts\footnote{In fact, we'll show that both of the inclusions in equation \eqref{maindiag} are equivalences of categories.} 
to a complex $\tilde{C}(\cal{L})$ in $\grvTr(2\cat{Foam})$, 
which in turn lifts to a complex $\tilde{D}(\cal{L})$ in $\grvTr(\cal{U}_Q(\glm)^{0 \leq 2})$, the universal annular $\slnn{2}$ invariant of $\cal{L}$. 
Results of Beliakova, Guliyev, Habiro, Lauda, Webster, and \u{Z}ivkovi\'{c} \cite{BHLZ,BGHL,BHLW} show that $\grvTr(\cal{U}_Q(\slm))$ is isomorphic to 
$\dcurm$, the idempotented form of the current algebra of $\slm$.
This pairs with skew Howe duality on $\bigwedge(\C^2 \otimes \C^m)$ to give a functor 
\[
\dcurm \xrightarrow{t=0}  \U(\slm) \xrightarrow{\SH} \Rep(\slnn{2})
\]
which necessarily factors through $\grvTr(\cal{U}_Q(\glm)^{0 \leq 2})$. Applying the induced functor 
\[
\grvTr(\cal{U}_Q(\glm)^{0 \leq 2})\xrightarrow{\SH}\Rep(\slnn{2})
\]
to $\tilde{D}(\cal{L})$ gives a complex whose chain groups are $\slnn{2}$ representations and whose differentials are $\slnn{2}$ intertwiners. 
The homology of this complex is $\saKh(\cal{L})$, which clearly carries an action of $\slnn{2}$.
The extension to general $n$ follows from a similar schematic, with $\Foam{2}{}$ replaced by $\Foam{n}{}$, the enhanced $\sln$ foam 2-category 
constructed by the authors in \cite{QR}.

In Section \ref{background}, we review Bar-Natan's cobordism model for Khovanov homology, Blanchet's enhanced version, and the relation 
to categorified quantum groups. We also review Asaeda-Przytycki-Sikora's annular Khovanov homology, its relation to the Bar-Natan category of 
the annulus, and introduce the Blanchet version of this category. In Section \ref{traces}, we discuss traces of 2-categories, showing how 
the (2-)categories discussed in Section \ref{background} assemble to give the commutative diagram in equation \eqref{maindiag}. 
In Section \ref{sH}, we show how to recover $\saKh$ via skew Howe duality, and also show that our arguments extend \textit{mutatis mutandis} 
to define $\sln$ sutured annular Khovanov-Rozansky homology, which again carries an action of $\sln$. 
Finally, in Section \ref{KhR} we relate the sutured annular Khovanov-Rozansky homology of an annular link to the 
Khovanov-Rozansky homology of the corresponding link in $S^3$, by constructing a spectral sequence from the former to the latter. \\

\noindent \textbf{Acknowledgements:} 
We'd like to thank 
Ivan Cherednik, 
Mikhail Khovanov,
Aaron Lauda, 
Lukas Lewark,
Scott Morrison,
Lev Rozansky, 
Peter Samuelson,
and Daniel Tubbenhauer
for their interest in this work and for helpful discussions.
Special thanks to Eli Grigsby, Tony Licata, and Stephan Wehrli for sharing their results before \cite{GLW} appeared;
their theorem that sutured annular Khovanov homology carries an action of $\slnn{2}$ served as a guiding principle for this work.
 H.Q. was funded by the ARC DP 140103821 and D.R. was partially supported by the John Templeton Foundation and NSF grant DMS-1255334.

% ########################################################################################################
%
\section{Background}\label{background}%%%%%%%%%%%%%%%%%%%%%%%%%%%%%%%%%%%%%%%%%%%%%%%%%%%%%%%%
%
% ########################################################################################################

\subsection{Khovanov homology}

In \cite{BN2}, Bar-Natan interprets Khovanov's homology theory for links $\cal{L} \subseteq S^3$ in terms of a quotient of the 
category of planar tangles and cobordisms. For\footnote{By convention, $0 \in \N$.} $a,b \in \N$, let $\BN{b}{a}$ denote the category whose objects are $\Z$-graded, 
formal direct sums of planar $(b,a)$-tangles in $[0,1] \times \R$, and whose morphisms are matrices of linear sums of degree-zero dotted cobordisms between 
such tangles, modulo isotopy (relative to the boundary) and local relations. We'll view our tangles as mapping from the $a$ endpoints on 
the right to the $b$ endpoints on the left, and cobordisms as mapping from the bottom tangle to the top. In fact, the categories $\{\BN{b}{a}\}_{a,b \in \N}$ 
assemble to give a $2$-category $\displaystyle \cal{BN}$ where $\Hom_{\cal{BN}}(a,b) := \BN{b}{a}$, 
and in which horizontal and vertical composition of 2-morphisms is given by gluing cobordisms in the indicated direction.

The local relations are as follows: 
\begin{equation}\label{BNrel}
\xy
(0,5)*{
\sphere[.4] \;\; = \;\; 0
};
(0,-5)*{
\dottedsphere[.4]{}  \;\; = \;\; 1
};
\endxy \qquad , \qquad
\cylinder[.4] \;\; = \;\; \slthncfour[.4] \;\; + \;\; \slthncfive[.4]
\end{equation}
and the degree of a cobordism $C: q^{d_1}T_1 \to q^{d_2}T_2$ is given by the 
formula\footnote{We use the negative of the grading convention from \cite{BN2}, 
to match with the categorified quantum group convention that a dot has degree $2$.}
\begin{equation}\label{sl2foamdeg}
\deg(C) = -\chi(C) + 2\#D + \frac{\# \partial}{2} - d_2 + d_1
\end{equation}
where $\chi$ is the Euler characteristic of the underlying surface, 
$\# D$ is the number of dots, $\# \partial$ is the number of boundary points in either $T_1$ or $T_2$, and the
powers of the formal variable $q$ denote degree shift.

The categorified skein relations\footnote{Here, and throughout, we've \uwave{underlined} the term of a complex in homological degree zero. 
Additionally, we use the grading shifts for complexes assigned to crossings from \cite{QR} throughout.}:
\begin{equation}\label{BNskein}
\left \llbracket
\xy
(0,0)*{
\begin{tikzpicture} [scale=.5]
\draw[very thick, directed=.99] (0,0) to (1,1);
\draw[very thick, directed=.99] (.4,.6) to (0,1);
\draw[very thick] (1,0) to (.6,.4);
\end{tikzpicture}
}
\endxy
\right \rrbracket
=
\left(
q^{-1} \;
\xy
(0,0)*{
\begin{tikzpicture} [scale=.5]
\draw[very thick] (0,0) to [out=45,in=315] (0,1);
\draw[very thick] (1,0) to [out=135,in=225] (1,1);
\end{tikzpicture}
}
\endxy
\xrightarrow{
\xy
(0,0)*{
\begin{tikzpicture} [scale=.15,fill opacity=0.2]
	%shading
	\path [fill=red] (4.25,-.5) to (4.25,2) to [out=170,in=10] (-.5,2) to (-.5,-.5) to 
		[out=0,in=225] (.75,0) to [out=90,in=180] (1.625,1.25) to [out=0,in=90] 
			(2.5,0) to [out=315,in=180] (4.25,-.5);
	\path [fill=red] (3.75,.5) to (3.75,3) to [out=190,in=350] (-1,3) to (-1,.5) to 
		[out=0,in=135] (.75,0) to [out=90,in=180] (1.625,1.25) to [out=0,in=90] 
			(2.5,0) to [out=45,in=180] (3.75,.5);
	%bottom web
	\draw [very thick] (.75,0) to [out=120,in=0] (-1,.5);
	\draw [very thick] (.75,0) to [out=240,in=0] (-.5,-.5);
	\draw [very thick] (3.75,.5) to [out=180,in=60] (2.5,0);
	\draw [very thick] (4.25,-.5) to [out=180,in=300] (2.5,0);
	%seam
	\draw [very thick] (.75,0) to [out=90,in=180] (1.625,1.25);
	\draw [very thick] (1.625,1.25) to [out=0,in=90] (2.5,0);
	%vertical edges
	\draw [very thick] (3.75,3) to (3.75,.5);
	\draw [very thick] (4.25,2) to (4.25,-.5);
	\draw [very thick] (-1,3) to (-1,.5);
	\draw [very thick] (-.5,2) to (-.5,-.5);
	%top web
	\draw [very thick] (4.25,2) to [out=170,in=10] (-.5,2);
	\draw [very thick] (3.75,3) to [out=190,in=350] (-1,3);	
\end{tikzpicture}
};
\endxy
}
\uwave{
\xy
(0,0)*{
\begin{tikzpicture} [scale=.5]
\draw[very thick] (0,0) to [out=45,in=135] (1,0);
\draw[very thick] (0,1) to [out=315,in=225] (1,1);
\end{tikzpicture}
}
\endxy} \;
\right) 
\;\; , \;\;
\left \llbracket
\xy
(0,0)*{
\begin{tikzpicture} [scale=.5]
\draw[very thick, directed=.99] (1,0) to (0,1);
\draw[very thick, directed=.99] (.6,.6) to (1,1);
\draw[very thick] (0,0) to (.4,.4);
\end{tikzpicture}
}
\endxy
\right \rrbracket
=
\left(
\uwave{
\xy
(0,0)*{
\begin{tikzpicture} [scale=.5]
\draw[very thick] (0,0) to [out=45,in=135] (1,0);
\draw[very thick] (0,1) to [out=315,in=225] (1,1);
\end{tikzpicture}
}
\endxy}
\xrightarrow{
\xy
(0,0)*{
\begin{tikzpicture} [scale=.15,fill opacity=0.2]
	%shading
	\path [fill=red] (4.25,2) to (4.25,-.5) to [out=170,in=10] (-.5,-.5) to (-.5,2) to
		[out=0,in=225] (.75,2.5) to [out=270,in=180] (1.625,1.25) to [out=0,in=270] 
			(2.5,2.5) to [out=315,in=180] (4.25,2);
	\path [fill=red] (3.75,3) to (3.75,.5) to [out=190,in=350] (-1,.5) to (-1,3) to [out=0,in=135]
		(.75,2.5) to [out=270,in=180] (1.625,1.25) to [out=0,in=270] 
			(2.5,2.5) to [out=45,in=180] (3.75,3);
%	\path[fill=blue] (2.5,2.5) to [out=270,in=0] (1.625,1.25) to [out=180,in=270] (.75,2.5);
	%bottom web
	\draw [very thick] (4.25,-.5) to [out=170,in=10] (-.5,-.5);
	\draw [very thick] (3.75,.5) to [out=190,in=350] (-1,.5);
	%seam
	\draw [very thick] (2.5,2.5) to [out=270,in=0] (1.625,1.25);
	\draw [very thick] (1.625,1.25) to [out=180,in=270] (.75,2.5);
	%vertical edges
	\draw [very thick] (3.75,3) to (3.75,.5);
	\draw [very thick] (4.25,2) to (4.25,-.5);
	\draw [very thick] (-1,3) to (-1,.5);
	\draw [very thick] (-.5,2) to (-.5,-.5);
	%top web
%	\draw [very thick,rdirected=.55] (2.5,2.5) to (.75,2.5);
	\draw [very thick] (.75,2.5) to [out=120,in=0] (-1,3);
	\draw [very thick] (.75,2.5) to [out=240,in=0] (-.5,2);
	\draw [very thick] (3.75,3) to [out=180,in=60] (2.5,2.5);
	\draw [very thick] (4.25,2) to [out=180,in=300] (2.5,2.5);		
\end{tikzpicture}
};
\endxy
}
q \;
\xy
(0,0)*{
\begin{tikzpicture} [scale=.5]
\draw[very thick] (0,0) to [out=45,in=315] (0,1);
\draw[very thick] (1,0) to [out=135,in=225] (1,1);
\end{tikzpicture}
}
\endxy \;
\right)
\end{equation}
assign a complex in $\cal{BN}$ to any framed tangle; in particular, they assign a complex $\llbracket \cal{L} \rrbracket$ in 
$\cal{BN}$ to a link $\cal{L} \subseteq S^3$. 
Applying (graded) $\HOM(\varnothing,-)$ to each term of this complex 
and imposing the further relation
\begin{equation}\label{2dots}
\dotsheet[.5]{2} \;\; = \;\; 0
\end{equation}
gives a complex of graded vector spaces, and taking homology gives $\Kh(\cal{L})$, the Khovanov homology of the link.

In \cite{LQR1}, the authors (joint with Lauda) relate $2\cat{Foam}$, the Blanchet foam version of Bar-Natan's cobordism $2$-category, to categorified quantum groups. 
The $2$-category $2\cat{Foam}$, introduced in \cite{Blan}, provides a properly\footnote{Khovanov's original construction of $\slnn{2}$ link homology is functorial under link 
cobordism only up to factors of $\pm1$. See \cite{CMW} for the original fix to this issue, which requires the use of complex coefficients in the ground ring.} 
functorial model for Khovanov homology, and is presented in the spirit of Khovanov's $\slnn{3}$ foams \cite{Kh5} (see also \cite{MSV} and \cite{QR} for the $\sln$ case).
Objects in the $2$-category $2\cat{Foam}$ are sequences of the symbols $1$ and $2$, which we view as labeling boundary points on $\{0\} \times \R$ and 
$\{1\} \times \R$, together with a zero object.  
The $1$-morphisms are graded, formal direct sums of left-directed $\slnn{2}$ webs -- trivalent graphs in $[0,1] \times \R$ generated by the basic webs:
\[
\xy
(0,0)*{
\begin{tikzpicture}[scale=.5]
	\draw [very thick, directed=.55] (1,0) to (-1,0);
	\node at (1.25,0) {\tiny $1$};
	\node at (-1.25,0) {\tiny $1$};
\end{tikzpicture}
};
\endxy
\quad , \quad
\xy
(0,0)*{
\begin{tikzpicture}[scale=.5]
	\draw [double] (1,0) to (-1,0);
	\node at (1.25,0) {\tiny $2$};
	\node at (-1.25,0) {\tiny $2$};
\end{tikzpicture}
};
\endxy
\quad , \quad
\xy
(0,0)*{
\begin{tikzpicture}[scale=.5]
	\draw [double] (2.25,0) to (.75,0);
	\draw [very thick,directed=.55] (.75,0) to [out=120,in=0] (-1,.75);
	\draw [very thick,directed=.55] (.75,0) to [out=240,in=0] (-1,-.75);
	\node at (2.5,0) {\tiny $2$};
	\node at (-1.25,.75) {\tiny $1$};
	\node at (-1.25,-.75) {\tiny $1$};
\end{tikzpicture}
};
\endxy
\quad , \quad
\xy
(0,0)*{
\begin{tikzpicture}[scale=.5]
	\draw [double] (-2.25,0) to (-.75,0);
	\draw [very thick,rdirected=.55] (-.75,0) to [out=60,in=180] (1,.75);
	\draw [very thick,rdirected=.55] (-.75,0) to [out=300,in=180] (1,-.75);
	\node at (-2.5,0) {\tiny $2$};
	\node at (1.25,.75) {\tiny $1$};
	\node at (1.25,-.75) {\tiny $1$};
\end{tikzpicture}
};
\endxy
\quad .
\]
The $2$-morphisms in this category are matrices of linear combinations of degree-zero $\slnn{2}$ foams, singular surfaces with $1$- and $2$-labeled facets 
between such webs 
generated by identity foams $W \times [0,1]$ of the generating webs $W$ above, 
dots on 1-labeled facets, 
and the basic foams:
\begin{equation}\label{foamgens}
\xy
(0,0)*{
\begin{tikzpicture} [scale=.4,fill opacity=0.2]
	%shading
	\path [fill=red] (4.25,-.5) to (4.25,2) to [out=165,in=15] (-.5,2) to (-.5,-.5) to 
		[out=0,in=225] (.75,0) to [out=90,in=180] (1.625,1.25) to [out=0,in=90] 
			(2.5,0) to [out=315,in=180] (4.25,-.5);
	\path [fill=red] (3.75,.5) to (3.75,3) to [out=195,in=345] (-1,3) to (-1,.5) to 
		[out=0,in=135] (.75,0) to [out=90,in=180] (1.625,1.25) to [out=0,in=90] 
			(2.5,0) to [out=45,in=180] (3.75,.5);
	\path[fill=yellow] (.75,0) to [out=90,in=180] (1.625,1.25) to [out=0,in=90] (2.5,0);
	%bottom web
	\draw [double] (2.5,0) to (.75,0);
	\draw [very thick,directed=.55] (.75,0) to [out=135,in=0] (-1,.5);
	\draw [very thick,directed=.55] (.75,0) to [out=225,in=0] (-.5,-.5);
	\draw [very thick,directed=.55] (3.75,.5) to [out=180,in=45] (2.5,0);
	\draw [very thick,directed=.55] (4.25,-.5) to [out=180,in=315] (2.5,0);
	%seam
	\draw [very thick, red, directed=.75] (.75,0) to [out=90,in=180] (1.625,1.25);
	\draw [very thick, red] (1.625,1.25) to [out=0,in=90] (2.5,0);
	%vertical edges
	\draw [very thick] (3.75,3) to (3.75,.5);
	\draw [very thick] (4.25,2) to (4.25,-.5);
	\draw [very thick] (-1,3) to (-1,.5);
	\draw [very thick] (-.5,2) to (-.5,-.5);
	%top web
	\draw [very thick,directed=.55] (4.25,2) to [out=165,in=15] (-.5,2);
	\draw [very thick, directed=.55] (3.75,3) to [out=195,in=345] (-1,3);
	%labels
	\node [Apricot, opacity=1]  at (1.625,.5) {\tiny{$2$}};
	\node[red, opacity=1] at (3.5,2.65) {\tiny{$1$}};
	\node[red, opacity=1] at (4,1.75) {\tiny{$1$}};		
\end{tikzpicture}
};
\endxy
\quad , \quad 
\xy
(0,0)*{
\begin{tikzpicture} [scale=.4,fill opacity=0.2]
	%shading
	\path [fill=red] (4.25,2) to (4.25,-.5) to [out=165,in=15] (-.5,-.5) to (-.5,2) to
		[out=0,in=225] (.75,2.5) to [out=270,in=180] (1.625,1.25) to [out=0,in=270] 
			(2.5,2.5) to [out=315,in=180] (4.25,2);
	\path [fill=red] (3.75,3) to (3.75,.5) to [out=195,in=345] (-1,.5) to (-1,3) to [out=0,in=135]
		(.75,2.5) to [out=270,in=180] (1.625,1.25) to [out=0,in=270] 
			(2.5,2.5) to [out=45,in=180] (3.75,3);
	\path[fill=yellow] (2.5,2.5) to [out=270,in=0] (1.625,1.25) to [out=180,in=270] (.75,2.5);
	%bottom web
	\draw [very thick,directed=.55] (4.25,-.5) to [out=165,in=15] (-.5,-.5);
	\draw [very thick, directed=.55] (3.75,.5) to [out=195,in=345] (-1,.5);
	%seam
	\draw [very thick, red, directed=.75] (2.5,2.5) to [out=270,in=0] (1.625,1.25);
	\draw [very thick, red] (1.625,1.25) to [out=180,in=270] (.75,2.5);
	%vertical edges
	\draw [very thick] (3.75,3) to (3.75,.5);
	\draw [very thick] (4.25,2) to (4.25,-.5);
	\draw [very thick] (-1,3) to (-1,.5);
	\draw [very thick] (-.5,2) to (-.5,-.5);
	%top web
	\draw [double] (2.5,2.5) to (.75,2.5);
	\draw [very thick,directed=.55] (.75,2.5) to [out=135,in=0] (-1,3);
	\draw [very thick,directed=.55] (.75,2.5) to [out=225,in=0] (-.5,2);
	\draw [very thick,directed=.55] (3.75,3) to [out=180,in=45] (2.5,2.5);
	\draw [very thick,directed=.55] (4.25,2) to [out=180,in=315] (2.5,2.5);
	%labels
	\node [Apricot, opacity=1]  at (1.625,2) {\tiny{$2$}};
	\node[red, opacity=1] at (3.5,2.65) {\tiny{$1$}};
	\node[red, opacity=1] at (4,1.75) {\tiny{$1$}};		
\end{tikzpicture}
};
\endxy
\quad , \quad
\xy
(0,0)*{
\begin{tikzpicture} [scale=.4,fill opacity=0.2]
	%back cup 
	\path[fill=red] (-.75,4) to [out=270,in=180] (0,2.5) to [out=0,in=270] (.75,4) .. controls (.5,4.5) and (-.5,4.5) .. (-.75,4);
	%sheet
	\path[fill=yellow] (-.75,4) to [out=270,in=180] (0,2.5) to [out=0,in=270] (.75,4) -- (2,4) -- (2,1) -- (-2,1) -- (-2,4) -- (-.75,4);
	%front cup
	\path[fill=red] (-.75,4) to [out=270,in=180] (0,2.5) to [out=0,in=270] (.75,4) .. controls (.5,3.5) and (-.5,3.5) .. (-.75,4);
	%bottom web
	\draw[double] (2,1) -- (-2,1);
	\path (.75,1) .. controls (.5,.5) and (-.5,.5) .. (-.75,1); %for spacing symmetry
	%seam
	\draw [very thick, red, directed=.7] (-.75,4) to [out=270,in=180] (0,2.5) to [out=0,in=270] (.75,4);
	%vertical edges
	\draw[very thick] (2,4) -- (2,1);
	\draw[very thick] (-2,4) -- (-2,1);
	%top web
	\draw[double] (2,4) -- (.75,4);
	\draw[double] (-.75,4) -- (-2,4);
	\draw[very thick,directed=.55] (.75,4) .. controls (.5,3.5) and (-.5,3.5) .. (-.75,4);
	\draw[very thick,directed=.55] (.75,4) .. controls (.5,4.5) and (-.5,4.5) .. (-.75,4);
	%labels
	\node [Apricot, opacity=1]  at (1.5,3.5) {\tiny{$2$}};
	\node[red, opacity=1] at (-.25,3.375) {\tiny{$1$}};
	\node[red, opacity=1] at (-.25,4.1) {\tiny{$1$}};	
\end{tikzpicture}
};
\endxy
\quad , \quad
\xy
(0,0)*{
\begin{tikzpicture} [scale=.4,fill opacity=0.2]
	%back cup 
	\path[fill=red] (-.75,-4) to [out=90,in=180] (0,-2.5) to [out=0,in=90] (.75,-4) .. controls (.5,-4.5) and (-.5,-4.5) .. (-.75,-4);
	%sheet
	\path[fill=yellow] (-.75,-4) to [out=90,in=180] (0,-2.5) to [out=0,in=90] (.75,-4) -- (2,-4) -- (2,-1) -- (-2,-1) -- (-2,-4) -- (-.75,-4);
	%front cup
	\path[fill=red] (-.75,-4) to [out=90,in=180] (0,-2.5) to [out=0,in=90] (.75,-4) .. controls (.5,-3.5) and (-.5,-3.5) .. (-.75,-4);
	%top web
	\draw[double] (2,-1) -- (-2,-1);
	\path (.75,-1) .. controls (.5,-.5) and (-.5,-.5) .. (-.75,-1); %for spacing symmetry
	%seam
	\draw [very thick, red, directed=.7] (.75,-4) to [out=90,in=0] (0,-2.5) to [out=180,in=90] (-.75,-4);
	%vertical edges
	\draw[very thick] (2,-4) -- (2,-1);
	\draw[very thick] (-2,-4) -- (-2,-1);
	%bottom web
	\draw[double] (2,-4) -- (.75,-4);
	\draw[double] (-.75,-4) -- (-2,-4);
	\draw[very thick,directed=.55] (.75,-4) .. controls (.5,-3.5) and (-.5,-3.5) .. (-.75,-4);
	\draw[very thick,directed=.55] (.75,-4) .. controls (.5,-4.5) and (-.5,-4.5) .. (-.75,-4);
	%labels
	\node [Apricot, opacity=1]  at (1.25,-1.5) {\tiny{$2$}};
	\node[red, opacity=1] at (-.25,-3.25) {\tiny{$1$}};
	\node[red, opacity=1] at (-.25,-4) {\tiny{$1$}};
\end{tikzpicture}
};
\endxy
\end{equation}
modulo isotopy and relations analogous to \eqref{BNrel}. 
Enhanced versions of the skein relations \eqref{BNskein} assign\footnote{Cap and cup tangles with vertical tangencies 
are sent to the appropriate generating webs.} 
a complex to any link:
\begin{equation}\label{eskein}
\left \llbracket
\xy
(0,0)*{
\begin{tikzpicture} [scale=.5]
\draw[very thick, directed=.99] (1,0) to (0,1);
\draw[very thick] (.6,.6) to (1,1);
\draw[very thick,directed=.99] (.4,.4) to (0,0);
\end{tikzpicture}
}
\endxy
\right \rrbracket
=
\left(
q^{-1} \;
\xy
(0,0)*{
\begin{tikzpicture} [scale=.6, rotate=90]
\draw[very thick, directed=.99] (0,0) to [out=45,in=315] (0,1);
\draw[very thick, directed=.99] (1,0) to [out=135,in=225] (1,1);
\end{tikzpicture}
}
\endxy
\xrightarrow{
\xy
(0,0)*{
\begin{tikzpicture} [scale=.2,fill opacity=0.2]
	%shading
	\path [fill=red] (4.25,2) to (4.25,-.5) to [out=170,in=10] (-.5,-.5) to (-.5,2) to
		[out=0,in=225] (.75,2.5) to [out=270,in=180] (1.625,1.25) to [out=0,in=270] 
			(2.5,2.5) to [out=315,in=180] (4.25,2);
	\path [fill=red] (3.75,3) to (3.75,.5) to [out=190,in=350] (-1,.5) to (-1,3) to [out=0,in=135]
		(.75,2.5) to [out=270,in=180] (1.625,1.25) to [out=0,in=270] 
			(2.5,2.5) to [out=45,in=180] (3.75,3);
	\path[fill=yellow] (2.5,2.5) to [out=270,in=0] (1.625,1.25) to [out=180,in=270] (.75,2.5);
	%bottom web
	\draw [very thick] (4.25,-.5) to [out=170,in=10] (-.5,-.5);
	\draw [very thick] (3.75,.5) to [out=190,in=350] (-1,.5);
	%seam
	\draw [very thick, red] (2.5,2.5) to [out=270,in=0] (1.625,1.25) to [out=180,in=270] (.75,2.5);
	%vertical edges
	\draw [very thick] (3.75,3) to (3.75,.5);
	\draw [very thick] (4.25,2) to (4.25,-.5);
	\draw [very thick] (-1,3) to (-1,.5);
	\draw [very thick] (-.5,2) to (-.5,-.5);
	%top web
	\draw [double] (2.5,2.5) to (.75,2.5);
	\draw [very thick] (.75,2.5) to [out=135,in=0] (-1,3);
	\draw [very thick] (.75,2.5) to [out=225,in=0] (-.5,2);
	\draw [very thick] (3.75,3) to [out=180,in=45] (2.5,2.5);
	\draw [very thick] (4.25,2) to [out=180,in=315] (2.5,2.5);		
\end{tikzpicture}
};
\endxy
}
\uwave{
\xy
(0,0)*{
\begin{tikzpicture} [scale=.6, rotate=90]
\draw[very thick, directed=.85] (.5,.75) to (1,1);
\draw[very thick, directed=.85] (.5,.75) to (0,1);
\draw[double] (.5,.75) to (.5,.25);
\draw[very thick, directed=.55] (1,0) to (.5,.25);
\draw[very thick, directed=.55] (0,0) to (.5,.25);
\end{tikzpicture}
}
\endxy} \;
\right) 
\quad , \quad
\left \llbracket
\xy
(0,0)*{
\begin{tikzpicture} [scale=.5]
\draw[very thick, directed=.99] (1,1) to (0,0);
\draw[very thick, directed=.99] (.4,.6) to (0,1);
\draw[very thick] (1,0) to (.6,.4);
\end{tikzpicture}
}
\endxy
\right \rrbracket
=
\left(
\uwave{
\xy
(0,0)*{
\begin{tikzpicture} [scale=.6, rotate=90]
\draw[very thick, directed=.85] (.5,.75) to (1,1);
\draw[very thick, directed=.85] (.5,.75) to (0,1);
\draw[double] (.5,.75) to (.5,.25);
\draw[very thick, directed=.55] (1,0) to (.5,.25);
\draw[very thick, directed=.55] (0,0) to (.5,.25);
\end{tikzpicture}
}
\endxy}
\xrightarrow{
\xy
(0,0)*{
\begin{tikzpicture} [scale=.2,fill opacity=0.2]
	%shading
	\path [fill=red] (4.25,-.5) to (4.25,2) to [out=170,in=10] (-.5,2) to (-.5,-.5) to 
		[out=0,in=225] (.75,0) to [out=90,in=180] (1.625,1.25) to [out=0,in=90] 
			(2.5,0) to [out=315,in=180] (4.25,-.5);
	\path [fill=red] (3.75,.5) to (3.75,3) to [out=190,in=350] (-1,3) to (-1,.5) to 
		[out=0,in=135] (.75,0) to [out=90,in=180] (1.625,1.25) to [out=0,in=90] 
			(2.5,0) to [out=45,in=180] (3.75,.5);
	\path[fill=yellow] (.75,0) to [out=90,in=180] (1.625,1.25) to [out=0,in=90] (2.5,0);
	%bottom web
	\draw [double] (2.5,0) to (.75,0);
	\draw [very thick] (.75,0) to [out=135,in=0] (-1,.5);
	\draw [very thick] (.75,0) to [out=225,in=0] (-.5,-.5);
	\draw [very thick] (3.75,.5) to [out=180,in=45] (2.5,0);
	\draw [very thick] (4.25,-.5) to [out=180,in=315] (2.5,0);
	%seam
	\draw [very thick, red] (.75,0) to [out=90,in=180] (1.625,1.25) to [out=0,in=90] (2.5,0);
	%vertical edges
	\draw [very thick] (3.75,3) to (3.75,.5);
	\draw [very thick] (4.25,2) to (4.25,-.5);
	\draw [very thick] (-1,3) to (-1,.5);
	\draw [very thick] (-.5,2) to (-.5,-.5);
	%top web
	\draw [very thick] (4.25,2) to [out=170,in=10] (-.5,2);
	\draw [very thick] (3.75,3) to [out=190,in=350] (-1,3);	
\end{tikzpicture}
};
\endxy
}
q \;
\xy
(0,0)*{
\begin{tikzpicture} [scale=.6, rotate=90]
\draw[very thick, directed=.99] (0,0) to [out=45,in=315] (0,1);
\draw[very thick, directed=.99] (1,0) to [out=135,in=225] (1,1);
\end{tikzpicture}
}
\endxy \;
\right).
\end{equation}
Passing to a quotient (denoted $\Foam{2}{}^\bullet$ in \cite{QR}) in which we impose 
the relation in equation \eqref{2dots} and applying a suitable representable functor 
gives a complex of finite-dimensional vector spaces, whose homology is Khovanov homology; see \cite{Blan}, \cite[Section 3.1]{LQR1}, or the  
$n=2$ case of \cite{QR} for full details.

In \cite{LQR1} we construct $2$-functors:
\begin{equation}\label{LQRthm}
\cal{U}_Q(\slm) \xrightarrow{\Phi^{m,N}} 2\cat{Foam}
\end{equation}
for each $m\geq 2$ and $N \geq1$, 
where $\cal{U}_Q(\slm)$ denotes the Khovanov-Lauda diagrammatic categorification of quantum $\slm$ \cite{KL,KL2,KL3}, 
in particular the version\footnote{We'll work with the choice of sign data $Q$ from \cite{QR}, \ie $t_{i,i+1}=-1$ and $t_{i+1,i}=1$.} 
described in \cite{CLau}. 
Recall that the 2-category $\cal{U}_Q(\slm)$ has $\slm$ weights $\lambda=(\lambda_1, \ldots, \lambda_{m-1}) \in \Z^{m-1}$ as 
objects, and $1$-morphisms are given by graded, formal direct sums of composites of identity morphisms $\onel$ and 
basic generators $\sE_i \onel \in \Hom(\lambda,\lambda+\alpha_i)$, $\sF_i \onel \in \Hom(\lambda,\lambda-\alpha_i)$,
where $\alpha_i$ is the $i^{\mathrm{th}}$ simple root. The $2$-morphisms in $\cal{U}_Q(\slm)$ are matrices of linear combinations of degree-zero 
dotted string diagrams\footnote{These diagrams should be viewed, both figuratively and literally, as perpendicular to 
planar tangles and webs.} colored by nodes of the $\slm$ Dynkin diagram,
generated by:
\[
\xy
(0,0)*{
\begin{tikzpicture}[scale=.5]
\draw[thick, ->] (0,-1) to (0,1);
\node at (.5,.5) {\tiny$_\lambda$};
\node at (-1,.5) {\tiny$_{\lambda+\alpha_i}$};
\node at (-.25,-1) {\tiny$_i$};
\end{tikzpicture}
};
\endxy
\;\; , \;\;
\xy
(0,0)*{
\begin{tikzpicture}[scale=.5]
\draw[thick, <-] (0,-1) to (0,1);
\node at (.5,.5) {\tiny$_\lambda$};
\node at (-1,.5) {\tiny$_{\lambda-\alpha_i}$};
\node at (-.25,-1) {\tiny$_i$};
\end{tikzpicture}
};
\endxy
\;\; , \;\;
\xy
(0,0)*{
\begin{tikzpicture}[scale=.5]
\draw[thick, ->] (0,-1) to (0,1);
\node at (.5,.5) {\tiny$_\lambda$};
\node at (-1,.5) {\tiny$_{\lambda+\alpha_i}$};
\node at (-.25,-1) {\tiny$_i$};
\node at (0,0) {$\bullet$};
\end{tikzpicture}
};
\endxy
\;\; , \;\;
\xy
(0,0)*{
\begin{tikzpicture}[scale=.5]
\draw[thick, ->] (0,-1) to [out=90,in=270] (1,1);
\draw[thick, ->] (1,-1) to [out=90,in=270] (0,1);
\node at (-.25,-1) {\tiny$_i$};
\node at (1.25,-1) {\tiny$_j$};
\node at (1.25,.25) {\tiny$_\lambda$};
\end{tikzpicture}
};
\endxy
\;\; , \;\;
\xy
(0,0)*{
\begin{tikzpicture}[scale=.5]
\draw[thick, ->] (0,-1) to [out=90,in=180] (.5,0) to [out=0,in=90] (1,-1);
\node at (-.25,-1) {\tiny$_i$};
\node at (1,.5) {\tiny$_\lambda$};
\end{tikzpicture}
};
\endxy
\;\; , \;\;
\xy
(0,0)*{
\begin{tikzpicture}[scale=.5]
\draw[thick, <-] (0,-1) to [out=90,in=180] (.5,0) to [out=0,in=90] (1,-1);
\node at (1.25,-1) {\tiny$_i$};
\node at (1,.5) {\tiny$_\lambda$};
\end{tikzpicture}
};
\endxy
\;\; , \;\;
\xy
(0,0)*{
\begin{tikzpicture}[scale=.5]
\draw[thick, ->] (0,1) to [out=270,in=180] (.5,0) to [out=0,in=270] (1,1);
\node at (-.25,1) {\tiny$_i$};
\node at (1,-.5) {\tiny$_\lambda$};
\end{tikzpicture}
};
\endxy
\;\; , \;\;
\xy
(0,0)*{
\begin{tikzpicture}[scale=.5]
\draw[thick, <-] (0,1) to [out=270,in=180] (.5,0) to [out=0,in=270] (1,1);
\node at (1.25,1) {\tiny$_i$};
\node at (1,-.5) {\tiny$_\lambda$};
\end{tikzpicture}
};
\endxy
\]
modulo local relations. These diagrams should be read as mapping from the sequence of boundary points at the 
bottom, to those at the top, where we translate a sequence of boundary points to a composition of basic 1-morphisms by identifying 
an upward $i$-labeled strand with $\sE_i \onel$ and a downward $i$-labeled strand with $\sF_i \onel$
and using the labeling of the planar regions to indicate the domain and codomain of the 1-morphism, \eg
\[
\xy
(0,0)*{
\begin{tikzpicture}[scale=.5]
\draw[thick, <-] (0,1) to [out=270,in=180] (1,0) to [out=0,in=270] (2,1);
\node at (2.25,1) {\tiny$_k$};
\draw[thick, ->] (0,-1) to [out=90,in=270] (1,1);
\node at (-.25,-1) {\tiny$_i$};
\draw[thick, ->] (1,-1) to [out=90,in=270] (-1,1);
\node at (1.25,-1) {\tiny$_j$};
\node at (2,-.5) {\tiny$\lambda$};
\end{tikzpicture}
};
\endxy
\]
is a $2$-morphism in $\Hom(\sE_i \sE_j \onel, q^d \sE_j \sE_k \sE_i \sF_k \onel)$. 
Here, and throughout, we use the (nonstandard) convention that powers of $q$ denote 
degree shifts of $1$-morphisms in categorified quantum groups; 
the shift in the codomain above depends on the adjacency of the nodes in the Dynkin diagram labeling the strings involved in the crossings.
See \cite{CLau} for full details about $\cal{U}_Q(\slm)$, where the grading shift $q^d$ is instead denoted by $\langle d \rangle$. 
The main result of \cite{KL3} is that $\UcatD_Q(\slm)$, 
the 2-category obtained from $\cal{U}_Q(\slm)$ by taking the Karoubi completion in each $\Hom$-category, 
categorifies $\U_q(\slm)$, the Lusztig idempotented form of quantum $\slm$.

\begin{rem}\label{rem:Ucheck}
We actually show in \cite{LQR1,QR} that the domain of the 2-functors in equation \eqref{LQRthm} (as well as those in equation \eqref{QRthm} below) 
can be extended to the full 2-subcategory $\Ucatc_Q(\slm)$ of $\UcatD_Q(\slm)$ generated by the divided power 1-morphisms 
$\cal{E}_{i}^{(a)} \onel$ and $\cal{F}_{i}^{(a)} \onel$ (see \cite{Lau1} and \cite{KLMS} for details on divided powers in the $\slnn{2}$ case). 
Since the vertical trace used later in this paper cannot distinguish between these 2-categories, we'll mostly work with the non-idempotent complete version for simplicity. 
However, in a handful of instances, our results can only be accurately stated in terms of $\Ucatc_Q(\slm)$; 
the reader unfamiliar with this version of the categorified quantum group can see \cite{LQR1} and \cite{QR} for full details.
\end{rem}

The 2-functor $\Phi^{m,N}$ in equation \eqref{LQRthm} 
is a combinatorial version of the categorical geometric skew Howe duality of Cautis-Kamnitzer-Licata \cite{CKL} and
was motivated by (and categorifies) the diagrammatic skew Howe duality of Cautis-Kamnitzer-Morrison \cite{CKM}.
It is given by first sending an $\slm$ weight $\lambda = (\lambda_1,\ldots,\lambda_{m-1})$ to the sequence 
$[ a_1,\ldots,a_m ] \in \Z^m$ so that $\sum a_i = N$ and $a_i - a_{i+1} = \lambda_i$ when a solution exists for $a_i \in \{0,1,2\}$
and then by deleting the zero entries in this sequence to give an object in $\Foam{2}{}$;
if no such sequence $[ a_1,\ldots,a_m ]$ exists then the $\slm$ weight is mapped to the zero object.
The following schematic specifies $\Phi^{m,N}$ on generating $1$-morphisms:
\[
\onel \mapsto
\xy
(0,0)*{
\begin{tikzpicture} [scale=.6]
%% draw webs
\draw [very thick, directed=.55] (3,0) -- (.5,0);
\draw [very thick, directed=.55] (3,1) -- (.5,1);
\draw [very thick, directed=.55] (3,1.5) -- (.5,1.5);
%% draw labels
\node at (3.6,0) {\tiny$a_m$};
\node at (3.6,1) {\tiny$a_2$};
\node at (3.6,1.5) {\tiny$a_1$};
\node at (1.75,.65) {\tiny$\cdot$};
\node at (1.75,.5) {\tiny$\cdot$};
\node at (1.75,.35) {\tiny$\cdot$};
\end{tikzpicture}};
\endxy
\;\; , \;\;
\sE_i \onel \mapsto
\xy
(0,0)*{
\begin{tikzpicture} [scale=.6]
%% draw webs
\draw [very thick, directed=.55] (3,0) -- (2,0);
\draw[very thick, directed=.55] (2,0) -- (0,0);
\draw [very thick, directed=.55] (3,1) -- (1,1);
\draw[very thick, directed=.55] (1,1) -- (0,1);
\draw [very thick, directed=.55] (2,0) -- (1,1);
%% draw labels
\node at (3.6,0) {\tiny$a_{i+1}$};
\node at (3.6,1) {\tiny$a_i$};
\node at (-1,0) {\tiny$a_{i+1} {-} 1$};
\node at (-1,1) {\tiny$a_i {+} 1$};
\node at (2,.5) {\tiny$1$};
\end{tikzpicture}};
\endxy
\;\; , \;\;
\sF_i \onel \mapsto
\xy
(0,0)*{
\begin{tikzpicture} [scale=.6]
%% draw webs
\draw [very thick, directed=.55] (3,0) -- (1,0);
\draw[very thick, directed=.55] (1,0) -- (0,0);
\draw [very thick, directed=.55] (3,1) -- (2,1);
\draw[very thick, directed=.55] (2,1) -- (0,1);
\draw [very thick, directed=.55] (2,1) -- (1,0);
%% draw labels
\node at (3.6,0) {\tiny$a_{i+1}$};
\node at (3.6,1) {\tiny$a_i$};
\node at (-1,0) {\tiny$a_{i+1} {+} 1$};
\node at (-1,1) {\tiny$a_i {-} 1$};
\node at (2,.5) {\tiny$1$};
\end{tikzpicture}};
\endxy
\]
where $0$-labeled edges are understood to be deleted. The $2$-functor is given on generating $2$-morphisms 
by sending the identity $2$-morphism of $\onel$ to parallel vertical sheets (of the appropriate type) and on 
the generating $2$-morphisms by viewing them as cross-sections of facets attached between the vertical sheets
(and deleting $0$-labeled facets), \eg
\[
\Phi_n \left(
\;
\xy
(0,0)*{
\begin{tikzpicture}[scale=.5]
\draw[thick, <-] (0,-1) to [out=90,in=180] (.5,0) to [out=0,in=90] (1,-1);
\node at (1.25,-1) {\tiny$_i$};
\node at (1,.5) {\tiny$_\lambda$};
\end{tikzpicture}
};
\endxy
\right) = \quad \capFEfoam{a_i}{a_{i+1}} \qquad .
\]
For more details, including a complete description of the 2-functor on 2-morphisms, see \cite{LQR1} and \cite{QR}.

We can describe the family of $2$-functors $\{\Phi^{m,N}\}_{N \geq 1}$ more succinctly using the categorified quantum group $\cal{U}_Q(\glm)$ from \cite{MSV2}.
This 2-category admits a description almost identical to $\cal{U}_Q(\slm)$, but objects are instead given by $\glm$ weights $\mathbf{a} = [a_1,\ldots,a_m] \in \Z^m$ 
and the (non-identity) generating $1$-morphisms are $\sE_i \onenn{\mathbf{a}} \in \Hom(\mathbf{a},\mathbf{a}+\epsilon_i)$ and 
$\sF_i \onenn{\mathbf{a}} \in \Hom(\mathbf{a},\mathbf{a}-\epsilon_i)$ for $\epsilon_i = [0,\ldots,1,-1,\ldots,0]$.
This then gives a $2$-functor $\cal{U}_Q(\glm) \xrightarrow{\Phi} 2\cat{Foam}$ defined on objects by simply forgetting the $0$-entries of a $\glm$ weight 
when each entry of the weight lies in $\{0,1,2\}$, and sending the weight to the zero object otherwise. 
This $2$-functor is given on $1$- and $2$-morphisms analogously to the $\slm$ case.

Note that $\Phi$ factors through the ``$2$-bounded" quotient of $\cal{U}_Q(\glm)$ 
where we've killed all weights with entries not lying in $\{0,1,2\}$; we'll denote this quotient by $\cal{U}_Q(\glm)^{0\leq2}$. 
We'll also later consider the more general $n$-bounded quotient $\cal{U}_Q(\glm)^{0\leq n}$ for $n \geq 2$, where we kill weights with entries not in the set $\{0,\dots,n\}$.
 We denote the limiting $2$-category, where we only quotient by $\glm$ weights containing negative integers, by $\cal{U}_Q(\glm)^{ 0\leq}$.

\begin{rem}\label{glmVslm}
Since the $\Hom$-categories in $\cal{U}_Q(\slm)$ and $\cal{U}_Q(\glm)$ are only non-trivial between weights which differ by the respective root lattices
(generated by the $\alpha_i$'s and $\epsilon_i$'s, respectively), we have that both 2-categories are given as the disjoint unions of 2-subcategories 
corresponding to root lattice cosets in the weight lattice. Using this, we have an identification
\[
\cal{U}_Q(\glm) \cong \coprod_{\Z} \cal{U}_Q(\slm)
\]
where a $\glm$ weight $[a_1,\ldots,a_m]$ on the left-hand side corresponds to the $\slm$ weight $(\lambda_1,\ldots,\lambda_{m-1})$ 
with $\lambda_i = a_i - a_{i+1}$ in the $\left\lfloor \frac{1}{m} \sum a_i \right\rfloor^{\mathrm{th}}$ copy on the right-hand side. 
We'll use this identification, and the corresponding one at the decategorified level, to import constructions from the ``$\slm$-side" to the 
``$\glm$-side," \eg we immediately obtain $2$-categories $\Ucatc_Q(\glm)$ and $\UcatD_Q(\glm)$. 
Similarly, the $2$-functor $\Phi$ is given by $\Phi^{m,N}$ acting on the (root lattice coset 2-subcategory of the) 
$\left\lfloor \frac{N}{m} \right\rfloor^{\mathrm{th}}$ summand on the right-hand side.
\end{rem}

\subsection{Annular cobordism categories and saKh}%####################################################################################################

There is a natural extension of the Bar-Natan category $\BN{0}{0}$ to the annulus $\cal{A} = S^1 \times [0,1]$, which we'll denote by $\cal{ABN}$. 
Objects in $\cal{ABN}$ are graded direct sums of disjoint unions of simple closed curves in $\cal{A}$ and morphisms are 
matrices of linear combinations of degree-zero dotted cobordisms embedded in $\cal{A} \times [0,1]$ between such objects, modulo isotopy and the relations in equation \eqref{BNrel}. 
Note that, unlike in the case of $\BN{0}{0}$, not all objects in $\cal{ABN}$ are isomorphic to direct sums of shifts of the empty curve.

The skein formulae \eqref{BNskein} assign a complex in $\cal{ABN}$ to any link $\cal{L} \subseteq \cal{A} \times [0,1]$; however, there is no longer a canonical choice of 
representable functor giving a complex of vector spaces of which we can take homology. 
Rather, the sutured annular Khovanov homology $\saKh(\cal{L})$ of Asaeda-Przytycki-Sikora can be interpreted as 
taking the homology of the resulting complex after applying a functor from $\cal{ABN}$ to the category of 
bi-graded vector spaces (with the grading denoted by powers of $q$ and $w$).
After imposing the additional relation \eqref{2dots}, 
this functor is given as in the non-annular case on the subcategory consisting of 
curves and cobordisms which don't interact with the annular core, and placing these vector spaces in zero $w$-degree. 
This functor is determined on all of $\cal{ABN}$ by the requirement that it is monoidal, sends the ``essential circle"
$
\xy
(0,0)*{
\begin{tikzpicture}[scale=.3]
\draw[gray] (0,0) circle (2);
\draw[very thick] (0,0) circle (1);
\draw[gray] (0,0) circle (.25);
\end{tikzpicture}
};
\endxy
$
to the graded vector space $V_1:=w\C \oplus w^{-1}\C = \mathrm{span}\{v_{+},v_{-} \}$, 
and maps the following cobordisms as indicated\footnote{A choice should be made to identify each circle with a particular tensor factor for the second cobordism.
Additionally, these maps differ slightly from those appearing previously in the literature \cite{APS,Rob,GW}, 
but are equal to them up to a change of basis and shifts in grading conventions.}:
\begin{equation}\label{safunctor}
\begin{aligned}
\xy
(0,0)*{
\begin{tikzpicture}[scale=.4, fill opacity=.2]
%shading
\path[fill=red] (1,4) arc (360:180: 1 and .5) to (-1,0) arc (180:360:1 and .5);
\path[fill=red] (1,4) arc (0:180: 1 and .5) to (-1,0) arc (180:0:1 and .5);
%bottom
\draw[gray] (0,0) ellipse (2 and 1);
\draw[very thick] (0,0) ellipse (1 and .5);
\draw[gray] (0,0) circle (.25 and .125);
%\middle
\draw[gray] (.25,0) to (.25,4);
\draw[gray] (-.25,0) to (-.25,4);
\draw[very thick] (1,0) to (1,4);
\draw[very thick] (-1,0) to (-1,4);
\draw[gray] (2,0) to (2,4);
\draw[gray] (-2,0) to (-2,4);
\node[opacity=1] at (.625,2) {$\bullet$};
%top
\draw[gray] (0,4) ellipse (2 and 1);
\draw[very thick] (0,4) ellipse (1 and .5);
\draw[gray] (0,4) circle (.25 and .125);
\end{tikzpicture}
};
\endxy
\;\; \mapsto 0
\qquad , \qquad
\xy
(0,0)*{
\begin{tikzpicture}[scale=.4, fill opacity=.2]
%shading
\path[fill=red] (1.5,0) to [out=90,in=315] (1.125,2) arc (360:180: 1.125 and .5625) to [out=225,in=90] (-1.5,0) arc (180:360:1.5 and .75);
\path[fill=red] (1.5,0) to [out=90,in=315] (1.125,2) arc (0:180: 1.125 and .5625) to [out=225,in=90] (-1.5,0) arc (180:0:1.5 and .75);
\path[fill=red] (.75,0) to [out=90,in=225] (1.125,2) arc (360:180: 1.125 and .5625) to [out=315,in=90] (-.75,0) arc (180:360:.75 and .375);
\path[fill=red] (.75,0) to [out=90,in=225] (1.125,2) arc (0:180: 1.125 and .5625) to [out=315,in=90] (-.75,0) arc (180:0:.75 and .375);
%bottom
\draw[gray] (0,0) ellipse (2 and 1);
\draw[very thick] (0,0) ellipse (.75 and .375);
\draw[very thick] (0,0) ellipse (1.5 and .75);
\draw[gray] (0,0) circle (.25 and .125);
%\middle
\draw[thick] (1.125,2) arc (0:180:1.125 and .5625);
\draw[gray] (.25,0) to (.25,4);
\draw[gray] (-.25,0) to (-.25,4);
\draw[very thick] (1.5,0) to [out=90,in=315] (1.125,2) to [out=225,in=90] (.75,0);
\draw[very thick] (-1.5,0) to [out=90,in=225] (-1.125,2) to [out=315,in=90] (-.75,0);
\draw[thick] (1.125,2) arc (360:180:1.125 and .5625);
\draw[gray] (2,0) to (2,4);
\draw[gray] (-2,0) to (-2,4);
%top
\draw[gray] (0,4) circle (2 and 1);
\draw[gray] (0,4) circle (.25 and .125);
\end{tikzpicture}
};
\endxy
\mapsto
\left \{
\begin{array}{c}
V_1 \otimes V_1 \longrightarrow  \C \\
v_{\pm} \otimes v_{\pm} \mapsto 0 \\
v_{\pm} \otimes v_{\mp} \mapsto \pm 1
\end{array}
\right \} \\
\xy
(0,0)*{
\begin{tikzpicture}[scale=.4, fill opacity=.2]
%shading
\path[fill=red] (1.5,4) to [out=270,in=45] (1.125,2) arc (360:180: 1.125 and .5625) to [out=135,in=90] (-1.5,4) arc (180:360:1.5 and .75);
\path[fill=red] (1.5,4) to [out=270,in=45] (1.125,2) arc (0:180: 1.125 and .5625) to [out=135,in=270] (-1.5,4) arc (180:0:1.5 and .75);
\path[fill=red] (.75,4) to [out=270,in=135] (1.125,2) arc (360:180: 1.125 and .5625) to [out=45,in=270] (-.75,4) arc (180:360:.75 and .375);
\path[fill=red] (.75,4) to [out=270,in=135] (1.125,2) arc (0:180: 1.125 and .5625) to [out=45,in=270] (-.75,4) arc (180:0:.75 and .375);
%bottom
\draw[gray] (0,0) circle (2 and 1);
\draw[gray] (0,0) circle (.25 and .125);
%top middle
\draw[thick] (1.125,2) arc (0:180:1.125 and .5625);
\draw[gray] (.25,0) to (.25,4);
\draw[gray] (-.25,0) to (-.25,4);
\draw[very thick] (1.5,4) to [out=270,in=45] (1.125,2) to [out=135,in=270] (.75,4);
\draw[very thick] (-1.5,4) to [out=270,in=135] (-1.125,2) to [out=45,in=270] (-.75,4);
\draw[thick] (1.125,2) arc (360:180:1.125 and .5625);
\draw[gray] (2,0) to (2,4);
\draw[gray] (-2,0) to (-2,4);
%top
\draw[gray] (0,4) ellipse (2 and 1);
\draw[very thick] (0,4) ellipse (.75 and .375);
\draw[very thick] (0,4) ellipse (1.5 and .75);
\draw[gray] (0,4) circle (.25 and .125);
\end{tikzpicture}
};
\endxy 
\mapsto
\left \{
\begin{array}{c}
\C \longrightarrow V_1 \otimes V_1  \\
1 \mapsto v_{+} \otimes v_{-} - v_{-} \otimes v_{+}
\end{array}
\right \} .
\end{aligned}
\end{equation}
In \cite{GLW}, Grigsby-Licata-Wehrli prove that this actually describes a monoidal functor $\cal{ABN} \to \cat{gr}\Rep(\slnn{2})$, where $\cat{gr}\Rep(\slnn{2})$ is 
a version of the category of finite-dimensional $\slnn{2}$ representations, which is also equipped with a formal $q$-degree.
This functor maps the empty curve to the trivial representation $V_0$ and hence sends the ``trivial circle"
\[
\xy
(0,0)*{
\begin{tikzpicture}[scale=.3]
\draw[gray] (0,0) circle (2);
\draw[very thick] (-1,0) circle (.5);
\draw[gray] (0,0) circle (.25);
\end{tikzpicture}
};
\endxy
\mapsto
qV_0 \oplus q^{-1}V_0
\]
and sends an essential circle to the fundamental $\slnn{2}$ representation $V_1$.

Again, in order to make contact with the theory of categorified quantum groups (and their trace decategorifications), we must introduce an enhanced version 
of $\cal{ABN}$. Let $\AFoam$ be the category in which objects are graded, direct sums of annular closures of left-directed $\slnn{2}$ webs, 
\ie the trivalent graphs obtained by gluing the matching endpoints of left-directed $\slnn{2}$ webs together, around the core of $\cal{A}$, \eg
\[
\xy
(0,0)*{
\begin{tikzpicture}[scale=.5]
\draw[double] (3,0) to (2.25,0);
\draw[very thick, directed=.55] (2.25,0) to [out=225,in=315] (.75,0);
\draw[very thick, directed=.55] (2.25,0) to (1.75,.5);
\draw[double] (1.75,.5) to (1.25,.5);
\draw[very thick, directed=.55] (1.25,.5) to (.75,0);
\draw[double] (.75,0) to (0,0);
\draw[very thick, directed=.55] (3,1) to [out=180,in=45] (1.75,.5);
\draw [very thick,->] (1.25,.5) to [out=135,in=0] (0,1);
\end{tikzpicture}
};
\endxy
\in 1\mathrm{Mor}(2\cat{Foam})
\qquad
\leadsto
\qquad
\xy
(0,0)*{
\begin{tikzpicture}[scale=.5]
\draw[gray] (0,0) circle (2);
\draw[gray] (0,0) circle (.125);
\draw[very thick, directed=.55] (.75,-1.25) to [out=225,in=315] (-.75,-1.25);
\draw[very thick, directed=.65] (.75,-1.25) to (.25,-.75);
\draw[double] (.25,-.75) to (-.25,-.75);
\draw[very thick, directed=.65] (-.25,-.75) to (-.75,-1.25);
\draw[very thick, directed=.55] (-.75,0) arc (180:0:.75);
\draw[double] (-1.5,0) arc (180:0:1.5);
\draw[very thick] (-.25,-.75) to [out=150,in=270] (-.75,0);
\draw[very thick] (.75,0) to [out=270,in=30] (.25,-.75);
\draw[double] (-.75,-1.25) to [out=180,in=270] (-1.5,0);
\draw[double] (1.5,0) to [out=270,in=0] (.75,-1.25);
\end{tikzpicture}
};
\endxy
\in
\mathrm{Ob}(\AFoam) .
\]

Morphisms in $\AFoam$ are given by matrices of linear combinations of degree-zero $\slnn{2}$ foams between such webs, embedded in $\cal{A} \times [0,1]$, \eg
\begin{equation}\label{AFoamex}
\xy
(0,0)*{
\begin{tikzpicture}[scale=.4, fill opacity=.2]
%shading
\path[fill=red] (1.5,0) to [out=90,in=315] (1.125,2) arc (360:180: 1.125 and .5625) to [out=225,in=90] (-1.5,0) arc (180:360:1.5 and .75);
\path[fill=red] (1.5,0) to [out=90,in=315] (1.125,2) arc (0:180: 1.125 and .5625) to [out=225,in=90] (-1.5,0) arc (180:0:1.5 and .75);
\path[fill=red] (.75,0) to [out=90,in=225] (1.125,2) arc (360:180: 1.125 and .5625) to [out=315,in=90] (-.75,0) arc (180:360:.75 and .375);
\path[fill=red] (.75,0) to [out=90,in=225] (1.125,2) arc (0:180: 1.125 and .5625) to [out=315,in=90] (-.75,0) arc (180:0:.75 and .375);
\path[fill=yellow] (1.125,4) arc (360:180: 1.125 and .5625) to (-1.125,2) arc (180:360:1.125 and .5625);
\path[fill=yellow] (1.125,4) arc (0:180: 1.125 and .5625) to (-1.125,2) arc (180:0:1.125 and .5625);
%bottom
\draw[gray] (0,0) ellipse (2 and 1);
\draw[very thick] (0,0) ellipse (.75 and .375);
\draw[very thick] (0,0) ellipse (1.5 and .75);
\draw[gray] (0,0) circle (.25 and .125);
%bottom middle
\draw[red, very thick] (1.125,2) arc (0:180:1.125 and .5625);
\draw[gray] (.25,0) to (.25,4);
\draw[gray] (-.25,0) to (-.25,4);
\draw[very thick] (1.5,0) to [out=90,in=315] (1.125,2) to [out=225,in=90] (.75,0);
\draw[very thick] (-1.5,0) to [out=90,in=225] (-1.125,2) to [out=315,in=90] (-.75,0);
\draw[red, very thick] (1.125,2) arc (360:180:1.125 and .5625);
\draw[gray] (2,0) to (2,4);
\draw[gray] (-2,0) to (-2,4);
% top middle
\draw (1.125,2) to (1.125,4);
\draw (-1.125,2) to (-1.125,4);
%top
\draw[gray] (0,4) circle (2 and 1);
\draw[gray] (0,4) circle (.25 and .125);
\draw[double] (0,4) circle (1.125 and .5625);
\end{tikzpicture}
};
\endxy
\qquad , \qquad
\xy
(0,0)*{
\begin{tikzpicture}[scale=.4, fill opacity=.2]
%shading
\path[fill=red] (1.5,4) to [out=270,in=45] (1.125,2) arc (360:180: 1.125 and .5625) to [out=135,in=90] (-1.5,4) arc (180:360:1.5 and .75);
\path[fill=red] (1.5,4) to [out=270,in=45] (1.125,2) arc (0:180: 1.125 and .5625) to [out=135,in=270] (-1.5,4) arc (180:0:1.5 and .75);
\path[fill=red] (.75,4) to [out=270,in=135] (1.125,2) arc (360:180: 1.125 and .5625) to [out=45,in=270] (-.75,4) arc (180:360:.75 and .375);
\path[fill=red] (.75,4) to [out=270,in=135] (1.125,2) arc (0:180: 1.125 and .5625) to [out=45,in=270] (-.75,4) arc (180:0:.75 and .375);
\path[fill=yellow] (1.125,0) arc (360:180: 1.125 and .5625) to (-1.125,2) arc (180:360:1.125 and .5625);
\path[fill=yellow] (1.125,0) arc (0:180: 1.125 and .5625) to (-1.125,2) arc (180:0:1.125 and .5625);
%bottom
\draw[gray] (0,0) circle (2 and 1);
\draw[gray] (0,0) circle (.25 and .125);
\draw[double] (0,0) circle (1.125 and .5625);
% top middle
\draw (1.125,2) to (1.125,0);
\draw (-1.125,2) to (-1.125,0);
%top middle
\draw[red, very thick] (1.125,2) arc (0:180:1.125 and .5625);
\draw[gray] (.25,0) to (.25,4);
\draw[gray] (-.25,0) to (-.25,4);
\draw[very thick] (1.5,4) to [out=270,in=45] (1.125,2) to [out=135,in=270] (.75,4);
\draw[very thick] (-1.5,4) to [out=270,in=135] (-1.125,2) to [out=45,in=270] (-.75,4);
\draw[red, very thick] (1.125,2) arc (360:180:1.125 and .5625);
\draw[gray] (2,0) to (2,4);
\draw[gray] (-2,0) to (-2,4);
%top
\draw[gray] (0,4) ellipse (2 and 1);
\draw[very thick] (0,4) ellipse (.75 and .375);
\draw[very thick] (0,4) ellipse (1.5 and .75);
\draw[gray] (0,4) circle (.25 and .125);
\end{tikzpicture}
};
\endxy
\qquad , \qquad \text{or} \qquad
\xy
(0,0)*{
\begin{tikzpicture}[scale=.4, fill opacity=.2]
%shading
\path[fill=yellow] (1.5,4) arc (0:180:1.5 and .75) to (-1.5,0) arc (180:0:1.5 and .75);
\path[fill=yellow] (1.5,4) to [out=270,in=0] (.75,3.5) to [out=150,in=30] (-.75,3.5) to [out=180,in=270] (-1.5,4) to
	(-1.5,0) arc (180:360:1.5 and .75);
\path[fill=red] (-.75,3.5) to [out=270,in=180] (0,2) to [out=0,in=270] (.75,3.5) to [out=150,in=30] (-.75,3.5);
\path[fill=red] (-.75,3.5) to [out=270,in=180] (0,2) to [out=0,in=270] (.75,3.5) to [out=210,in=330] (-.75,3.5);
%bottom
\draw[gray] (0,0) circle (2 and 1);
\draw[gray] (0,0) circle (.25 and .125);
\draw[double] (0,0) circle (1.5 and .75);
% middle
\draw[gray] (.25,0) to (.25,4);
\draw[gray] (-.25,0) to (-.25,4);
\draw[gray] (2,0) to (2,4);
\draw[gray] (-2,0) to (-2,4);
\draw[very thick, red] (-.75,3.5) to [out=270,in=180] (0,2) to [out=0,in=270] (.75,3.5);
\draw (1.5,0) to (1.5,4);
\draw (-1.5,0) to (-1.5,4);
%top
\draw[gray] (0,4) ellipse (2 and 1);
\draw[gray] (0,4) circle (.25 and .125);
\draw[very thick] (.75,3.5) to [out=150,in=30] (-.75,3.5);
\draw[very thick] (.75,3.5) to [out=210,in=330] (-.75,3.5);
\draw[double] (-1.5,4) arc (180:0:1.5 and .75);
\draw[double] (1.5,4) to [out=270,in=0] (.75,3.5);
\draw[double] (-.75,3.5) to [out=180,in=270] (-1.5,4);
\end{tikzpicture}
};
\endxy
\end{equation}
again modulo isotopy and local relations. In equation \eqref{AFoamex}, the first two foams should be viewed as shorthand for the annular 
foams given by the composition of the annular closures of generating foams in \eqref{foamgens}.
Note that $\AFoam$ has a natural monoidal structure, given by gluing the outermost circle of one annulus to the 
innermost circle of the other.

The skein relations \eqref{eskein} assign a complex in $\AFoam$ to any link $\cal{L} \subset \cal{A} \times [0,1]$, and an analog of the above functor gives $\saKh(\cal{L})$. 
In Section \ref{sH}, we'll show how this functor arises naturally by considering trace decategorifications of $\cal{U}_Q(\glm)$ and skew Howe duality.

% ########################################################################################################
%
\section{Trace decategorification}\label{traces}%%%%%%%%%%%%%%%%%%%%%%%%%%%%%%%%%%%%%%%%%%%%%%%%%%%%%%%%
%
% ########################################################################################################

In this section, we'll discuss trace decategorifications of linear 2-categories, in particular the vertical and horizontal traces of categorified quantum groups and 
foam 2-categories. See \cite{BHLZ,BGHL,BHLW} for a detailed study of the vertical trace decategorification of $\cal{U}_Q(\slm)$.

\subsection{Vertical trace}%####################################################################################################

Recall from \cite{BHLZ} that the trace of a small linear category $\mathbf{C}$ is the abelian group given by 
\[
\Tr(\cat{C}) := \bigoplus_{c \in \mathrm{Ob}(\mathbf{C})} \Hom_{\mathbf{C}}(c,c) \bigg/ \mathrm{span}\{fg - gf\}
\]
where the span is taken over all pairs $f \in \Hom_{\mathbf{C}}(c_1,c_2)$ and $g \in \Hom_{\mathbf{C}}(c_2,c_1)$. 
This directly extends to give a notion of trace for linear $2$-categories; recall that if $\cal{C}$ is a linear $2$-category and 
$c_1, c_2$ are objects in $\cal{C}$ then $\Hom_{\cal{C}}(c_1,c_2)$ is itself a linear category.
\begin{defn} The \emph{vertical trace} of a linear $2$-category $\cal{C}$ is the linear category $\vTr(\cal{C})$ whose objects 
are the same as those in $\cal{C}$ and with morphisms determined by
\[
\Hom_{\vTr(\cal{C})}(c_1,c_2) := \Tr \left( \Hom_{\cal{C}}(c_1,c_2) \right).
\]
\end{defn}
Given a $2$-morphism $\mathsf{D}$ in $\cal{C}$ whose domain and codomain $1$-morphisms agree, we'll denote the corresponding 
morphism in $\vTr(\cal{C})$ by $\vTr(\mathsf{D})$.
If the 2-category $\cal{C}$ is presented graphically via string diagrams, such as $\cal{U}_Q(\slm)$, the vertical trace admits a topological interpretation. 
Morphisms in $\vTr(\cal{C})$ are described by 2-morphisms in $\cal{C}$ with equal domain and codomain 1-morphisms, closed around a 
cylinder\footnote{In \cite{BGHL}, this cylinder is flattened to an annulus, but we'll not do this in order to avoid confusion with 
annular $\slnn{2}$ webs, and to highlight the relation with the horizontal trace.}, \eg 
\begin{equation}\label{vTrex}
\vTr\left(
\xy
(0,0)*{
\begin{tikzpicture}[scale=.5]
\draw[thick, ->] (0,-1) to [out=90,in=270] (1,1);
\draw[thick, ->] (1,-1) to [out=90,in=270] (0,1);
%\node at (-.25,-1) {\tiny$_i$};
%\node at (1.25,-1) {\tiny$_j$};
\node at (1.25,-.25) {\tiny$_\lambda$};
\node at (-.3,-.25) {\tiny$_{\lambda{+}4}$};
\node at (.25,.25) {$\bullet$};
\end{tikzpicture}
};
\endxy
\right) = \;
\xy
(0,0)*{
\begin{tikzpicture}[scale=.3]
%left
\draw[gray] (-5,0) circle (1 and 2);
%middle
\draw[gray] (-5,2) to (4,2);
\draw[gray] (-5,-2) to (4,-2);
\draw[thick, dashed] (-1,2) arc (90:270:1 and 2); 
\draw[thick, dashed] (1,2) arc (90:270:1 and 2);
\draw[thick, directed=.85] (1,-2) to [out=0,in=0] (-1,2);
\draw[thick, directed=.85] (-1,-2) to [out=0,in=0] (1,2);
%right
\draw[gray] (4,-2) arc (-90:90:1 and 2);
\draw[gray, dashed] (4,-2) arc (270:90:1 and 2);
%labels
\node at (2.5,1) {\tiny$\lambda$};
\node at (-3,1) {\tiny$\lambda{+}4$};
\node at (.625,1) {$\bullet$};
\end{tikzpicture}
};
\endxy
\end{equation}
in $\vTr(\cal{U}_Q(\slnn{2}))$. In this description of $\vTr(\cal{U}_Q(\slm))$, we must also specify a grading shift\footnote{The category $\vTr(\cal{U}_Q(\slm))$ is not graded
\emph{sensu stricto}, since there is no grading (shift) on objects.} $q^d$ on each cylindrical diagram, corresponding 
to the shift in degree of the (co)domain 1-morphism of the corresponding $2$-endomorphism in $\cal{U}_Q(\slm)$; if no such grading shift appears, as in 
equation \eqref{vTrex}, the diagram is understood to be in degree zero.
The relations on diagrams are extended to allow isotopy around the cylinder; composition is given by gluing two such cylinders together along a boundary.
In \cite{BHLZ,BGHL,BHLW}, the authors show that $\U_q(\slm) \cong \vTr(\UcatD_Q(\slm)) \cong \vTr(\cal{U}_Q(\slm))$, with the isomorphism determined by
\begin{equation}\label{grvTr}
q^k E_i 1_\l \mapsto
q^k \;
\xy
(0,0)*{
\begin{tikzpicture}[scale=.3]
%left
\draw[gray] (-4,0) circle (1 and 2);
%middle
\draw[gray] (-4,2) to (3,2);
\draw[gray] (-4,-2) to (3,-2);
\draw[thick, directed=.75] (0,-2) arc (-90:90:1 and 2);
\draw[dashed, thick] (0,-2) arc (270:90:1 and 2);
%right
\draw[gray] (3,-2) arc (-90:90:1 and 2);
\draw[gray, dashed] (3,-2) arc (270:90:1 and 2);
%labels
\node at (2.75,.75) {$_\lambda$};
\node at (.25,-1.25) {\tiny$i$};
\end{tikzpicture}
};
\endxy \;
\quad , \quad
q^k F_j 1_\l \mapsto 
q^k \;
\xy
(0,0)*{
\begin{tikzpicture}[scale=.3]
%left
\draw[gray] (-4,0) circle (1 and 2);
%middle
\draw[gray] (-4,2) to (3,2);
\draw[gray] (-4,-2) to (3,-2);
\draw[thick, rdirected=.35] (0,-2) arc (-90:90:1 and 2);
\draw[dashed, thick] (0,-2) arc (270:90:1 and 2);
%right
\draw[gray] (3,-2) arc (-90:90:1 and 2);
\draw[gray, dashed] (3,-2) arc (270:90:1 and 2);
%labels
\node at (2.75,.75) {$_\lambda$};
\node at (.25,1.25) {\tiny$j$};
\end{tikzpicture}
};
\endxy \;\; .
\end{equation}

A more interesting vertical trace results by considering a non-graded variant of $\cal{U}_Q(\slm)$. Let $\cal{U}^*_Q(\slm)$ be the version of $\cal{U}_Q(\slm)$ in which 
all of the different grading shifts of each 1-morphism are identified or, equivalently, where there are no grading shifts of 1-morphisms and 2-morphisms are not required 
to be degree-zero. In \cite{BHLZ,BHLW}, the vertical trace of this 2-category is related to the current algebra $\curm$, the universal enveloping algebra of 
$\slm[t]:= \slm \otimes \C[t]$.
Let
\[a_{ij} = 
\begin{cases}
2 & \text{if }i=j \\
-1 & \text{if }i=j\pm1 \\
0 & \text{else}
\end{cases} 
\] 
be the Cartan matrix of $\slm$.

\begin{defn} 
The \emph{current algebra} $\curm$ of $\slm$ is the $\C$-algebra generated by $E_{i,r}$, $F_{i,r}$ and $H_{i,r}$, for $1\leq i \leq m-1$ and $r\in \N$, subject to the following relations:
\begin{enumerate}
\item[\textbf{1.}]
$[H_{i,r},H_{j,s}]=0 \;\; , \;\; [H_{i,r},E_{j,s}] = a_{ij} E_{j,r+s} \;\; , \;\; [H_{i,r},F_{j,s}] = -a_{ij} F_{j,r+s}$ \label{ca1} \\
\item[\textbf{2.}]
$[E_{i,r+1},E_{j,s}]=[E_{i,r},E_{j,s+1}] \;\; , \;\; [F_{i,r+1},F_{j,s}]=[F_{i,r},F_{j,s+1}] \;\; , \;\; [E_{i,r},F_{j,s}]=\delta_{i,j}H_{i,r+s}$  \\
\item[\textbf{3.}]
for $|i-j| \neq1$: \;\; $E_{i,r}E_{j,s}=E_{j,s}E_{i,r} \;\; , \;\; F_{i,k}F_{j,l}=F_{j,l}F_{i,k}$ \\
\item[\textbf{4.}]
for $|i-j| = 1$: \;\; $[E_{i,r} , [E_{i,s} , E_{j,t} ] ] = 0 = [F_{i,r} , [F_{i,s} , F_{j,t} ] ]$.
\end{enumerate}
\end{defn}

\begin{rem}
Our relation $\mathbf{1}$ is the same as relations $\mathbf{C1}$, $\mathbf{C2}$, and $\mathbf{C3}$ from \cite{BHLW}, and 
our relation $\mathbf{2}$ is their $\mathbf{C4}$ and $\mathbf{C5}$. 
Our relation $\mathbf{3}$ is their relation $\mathbf{C6a}$ when $|i-j| > 1$ 
(and follows from our relation $\mathbf{2}$ when $i=j$ and when working over a field), and our relation 
$\mathbf{4}$ is their relation $\mathbf{C6b}$.
\end{rem}

The \emph{Lusztig idempotented form} $\dcurm$ of the current algebra is the additive category with objects $\slm$ weights $\lambda \in \Z^{m-1}$ 
and with morphisms generated by $H_{i,r}1_\l \in \End(\l)$, $E_{i,r} 1_\l \in \Hom(\l, \l+\alpha_i)$, and $F_{i,r} 1_\l \in \Hom(\l, \l-\alpha_i)$ 
subject to the above relations and the condition that $H_{i,0}1_\l = \l_i 1_\l$. As usual, we can also consider $\U(\glm[t])$ by passing to $\glm$ weights.

The main results of \cite{BHLZ,BHLW} give the following:
\begin{thm}\label{thm:Cur} 
There is an isomorphism $\dcurm \cong \vTr(\cal{U}^*_Q(\slm))$ of linear categories determined by
\[
E_{i,r}1_\l \mapsto 
\vTr\left(
\xy
(0,0)*{
\begin{tikzpicture}[scale=.5]
\draw[thick, ->] (0,-1) to (0,1);
\node at (-.25,-.75) {\tiny$i$};
\node at (1,.5) {$_\lambda$};
\node at (.15,0) {$\bullet^r$};
\end{tikzpicture}
};
\endxy
\right) = \;
\xy
(0,0)*{
\begin{tikzpicture}[scale=.3]
%left
\draw[gray] (-4,0) circle (1 and 2);
%middle
\draw[gray] (-4,2) to (3,2);
\draw[gray] (-4,-2) to (3,-2);
\draw[thick, directed=.75] (0,-2) arc (-90:90:1 and 2);
\draw[dashed, thick] (0,-2) arc (270:90:1 and 2);
%right
\draw[gray] (3,-2) arc (-90:90:1 and 2);
\draw[gray, dashed] (3,-2) arc (270:90:1 and 2);
%labels
\node at (2.75,.75) {$_\lambda$};
\node at (1.25,-.25) {$\bullet^r$};
\node at (.25,-1.25) {\tiny$i$};
\end{tikzpicture}
};
\endxy
\quad , \quad
F_{j,s}1_\l \mapsto 
\vTr\left(
\xy
(0,0)*{
\begin{tikzpicture}[scale=.5]
\draw[thick, <-] (0,-1) to (0,1);
\node at (-.25,.75) {\tiny$j$};
\node at (1,.5) {$_\lambda$};
\node at (.15,0) {$\bullet^s$};
\end{tikzpicture}
};
\endxy
\right) = \;
\xy
(0,0)*{
\begin{tikzpicture}[scale=.3]
%left
\draw[gray] (-4,0) circle (1 and 2);
%middle
\draw[gray] (-4,2) to (3,2);
\draw[gray] (-4,-2) to (3,-2);
\draw[thick, rdirected=.35] (0,-2) arc (-90:90:1 and 2);
\draw[dashed, thick] (0,-2) arc (270:90:1 and 2);
%right
\draw[gray] (3,-2) arc (-90:90:1 and 2);
\draw[gray, dashed] (3,-2) arc (270:90:1 and 2);
%labels
\node at (2.75,.75) {$_\lambda$};
\node at (1.25,.25) {$\bullet^s$};
\node at (.25,1.25) {\tiny$j$};
\end{tikzpicture}
};
\endxy
\]
Under this assignment, $H_{i,0} 1_\l \mapsto \lambda_i 1_\l$ and 
\begin{align*}
H_{i,r} \mapsto
\sum_{a+b=r}
(a+1)
\vTr \left(
\xy
(0,0)*{
\begin{tikzpicture}[scale=.75]
	\draw[thick, ->] (.5,0) arc (0:-360:0.5);
	\node at (0,-.5) {$\bullet$};
	\node at (-.25,-.75) {\tiny$\lambda_i {-} 1 {+} a$};
\end{tikzpicture}
};
\endxy
\;\;
\xy
(0,0)*{
\begin{tikzpicture}[scale=.75]
	\draw[thick, <-] (.5,0) arc (0:-360:0.5);
	\node at (0,-.5) {$\bullet$};
	\node at (.25,-.75) {\tiny$-\lambda_i {-} 1 {+} b$};
\end{tikzpicture}
};
\endxy
\right)
=&
- \sum_{a+b=r}
(b+1)
\vTr \left(
\xy
(0,0)*{
\begin{tikzpicture}[scale=.75]
	\draw[thick, ->] (.5,0) arc (0:-360:0.5);
	\node at (0,-.5) {$\bullet$};
	\node at (-.25,-.75) {\tiny$\lambda_i {-} 1 {+} a$};
\end{tikzpicture}
};
\endxy
\;\;
\xy
(0,0)*{
\begin{tikzpicture}[scale=.75]
	\draw[thick, <-] (.5,0) arc (0:-360:0.5);
	\node at (0,-.5) {$\bullet$};
	\node at (.25,-.75) {\tiny$-\lambda_i {-} 1 {+} b$};
\end{tikzpicture}
};
\endxy
\right) \\
=&
- \sum_{a+b=r}
a \;
\vTr \left(
\xy
(0,0)*{
\begin{tikzpicture}[scale=.75]
	\draw[thick, ->] (.5,0) arc (0:-360:0.5);
	\node at (0,-.5) {$\bullet$};
	\node at (-.25,-.75) {\tiny$\lambda_i {-} 1 {+} a$};
\end{tikzpicture}
};
\endxy
\;\;
\xy
(0,0)*{
\begin{tikzpicture}[scale=.75]
	\draw[thick, <-] (.5,0) arc (0:-360:0.5);
	\node at (0,-.5) {$\bullet$};
	\node at (.25,-.75) {\tiny$-\lambda_i {-} 1 {+} b$};
\end{tikzpicture}
};
\endxy
\right)
\end{align*}
for $r>0$ (here the bubbles are all $i$ colored).
\end{thm}

We can also consider the vertical trace of the $\slnn{2}$ foam $2$-category, which admits a topological description as well. We'll again work 
both with $\Foam{2}{}$ and a version $2\cat{Foam}^*$ in which we've forgotten the grading on 1-morphisms, and allow 2-morphisms of 
any degree. Morphisms in $\vTr(\Foam{2}{})$ or $\vTr(\Foam{2}{}^*)$ are given by closing foams with the same domain and codomain 1-morphisms 
around a horizontal axis, \eg
\[
\vTr\left(
\xy
(0,0)*{
\begin{tikzpicture} [fill opacity=0.2, scale=.35]
	% draw the bottom web
	\draw[double] (2,0) -- (.75,0);
	\draw[double] (-.75,0) -- (-2,0);
	\draw[very thick, rdirected=.55] (-.75,0) .. controls (-.5,-.5) and (.5,-.5) .. (.75,0);
	\draw[very thick, rdirected=.55] (-.75,0) .. controls (-.5,.5) and (.5,.5) .. (.75,0);
	% draw and color the vertical sheets
	\draw[very thick] (-2,0) -- (-2,4);
	\draw[very thick] (2,0) -- (2,4);
	\path [fill=yellow] (-2,4) -- (-.75,4) -- (-.75,0) -- (-2,0) -- cycle;
	\path [fill=yellow] (2,4) -- (.75,4) -- (.75,0) -- (2,0) -- cycle;
	\path [fill=yellow] (-.75,4) .. controls (-.75,2) and (.75,2) .. (.75,4) --
			(.75, 0) .. controls (.75,2) and (-.75,2) .. (-.75,0);
	% blue caps and cups
	\path [fill=red] (-.75,4) .. controls (-.5,4.5) and (.5,4.5) .. (.75,4) --
			(.75,4) .. controls (.75,2) and (-.75,2) .. (-.75,4);
	\path [fill=red] (-.75,4) .. controls (-.5,3.5) and (.5,3.5) .. (.75,4) --
			(.75,4) .. controls (.75,2) and (-.75,2) .. (-.75,4);
	\path [fill=red] (-.75, 0) .. controls (-.75,2) and (.75,2) .. (.75,0) --
			(.75,0) .. controls (.5,-.5) and (-.5,-.5) .. (-.75,0);
	\path [fill=red] (-.75, 0) .. controls (-.75,2) and (.75,2) .. (.75,0) --
			(.75,0) .. controls (.5,.5) and (-.5,.5) .. (-.75,0);
	% the dot
	\node [opacity=1] at (0,1) {\footnotesize$\bullet$};
	%draw the seams
	\draw [very thick, red] (.75,4) .. controls (.75,2) and (-.75,2) .. (-.75,4);
	% to reverse use (-.75,4) .. controls (-.75,2) and (.75,2) .. (.75,4); above
	\draw [very thick, red] (-.75, 0) .. controls (-.75,2) and (.75,2) .. (.75,0);
	% to reverse use (.75, 0) .. controls (.75,2) and (-.75,2) .. (-.75,0); above
	% draw the top web
	\draw[double] (2,4) -- (.75,4);
	\draw[double] (-.75,4) -- (-2,4);
	\draw[very thick, rdirected=.55] (-.75,4) .. controls (-.5,3.5) and (.5,3.5) .. (.75,4);
	\draw[very thick, rdirected=.55] (-.75,4) .. controls (-.5,4.5) and (.5,4.5) .. (.75,4);
\end{tikzpicture}};
\endxy
\right)
= \;
\xy
(0,0)*{
\begin{tikzpicture}[fill opacity=.2, scale=.4]
%shading
\path[fill=yellow] (2,-1.25) arc (270:90:.625 and 1.25) to (-3,1.25) arc (90:270:.625 and 1.25);
\path[fill=yellow] (.25,0) to (2.625,0) arc (0:90:.625 and 1.25) to (-3,1.25) arc (90:-90:.625 and 1.25) to (2,-1.25) arc (270:360:.625 and 1.25) to (.25,0) arc (360:0:.75 and 1);
\path[fill=red] (-.5,0) circle (.75 and 1);
\path[fill=red] (-.5,0) circle (.75 and 1);
%right
\draw[gray] (2,-2) arc (-90:90:1 and 2);
\draw[gray, dashed] (2,-2) arc (270:90:1 and 2);
\draw[double] (2,0) circle (.625 and 1.25);
\draw[gray] (2,-.25) arc (-90:90:0.125 and .25);
\draw[gray, dashed] (2,-.25) arc (270:90:0.125 and .25);
%middle
\draw[gray] (-3,2) to (2,2);
\draw[gray] (-3,-2) to (2,-2);
\draw[gray] (-3,.25) to (2,.25);
\draw[gray] (-3,-.25) to (2,-.25);
\draw (-3,1.25) to (2,1.25);
\draw (-3,-1.25) to (2,-1.25);
\draw[very thick, red] (-.5,0) circle (.75 and 1);
\draw[dashed, red] (-1.25,0) to [out=30,in=150] (.25,0);
\draw[dashed, red] (.25,0) to [out=210,in=330] (-1.25,0);
\node[opacity=1] at (-.5,.5) {$\bullet$};
%left
\draw[gray] (-3,0) circle (1 and 2);
\draw[double] (-3,0) circle (.625 and 1.25);
\draw[gray] (-3,0) circle (.125 and .25);
\end{tikzpicture}
};
\endxy \quad .
\]
Note that objects in $\vTr(\Foam{2}{})$ or $\vTr(\Foam{2}{}^*)$ giving the domain and codomain of such a morphism, 
which are sequences of $1$'s and $2$'s, 
can be read off (radially) from the strand labelings of the circles on the right and left boundaries. As before, such a closure of a 
degree-zero foam only specifies the morphism in $\vTr(\Foam{2}{})$ up to grading shift, since we can simultaneously shift the 
domain and codomain webs in $\Foam{2}{}$ to obtain distinct 2-morphisms, hence different morphisms in $\vTr(\Foam{2}{})$.

Functoriality of the vertical trace, together with equation \eqref{LQRthm} and Theorem \ref{thm:Cur}, immediately gives the following result.

\begin{prop}\label{prop:vTrFunctor}
For each $m,N$, there exists a functor $\dcurm \xrightarrow{\varphi^{m,N}} \vTr(\Foam{2}{}^*)$ determined by $\lambda \mapsto [a_1, \ldots, a_m]$
for $\lambda_i = a_i - a_{i+1}$, $\sum a_i = N$ and
\[
E_{i,r}1_\l \mapsto 
\xy
(0,0)*{
\begin{tikzpicture}[fill opacity=.2, scale=.5]
%shading
\path[fill=red] (3,-.75) arc (-90:90:.375 and .75) to (-4,.75) arc (90:-90:.375 and .75);
\path[fill=red] (3,-.75) arc (270:90:.375 and .75) to (-4,.75) arc (90:270:.375 and .75);
\path[fill=blue] (0,-1.5) arc (-90:90:.75 and 1.5) to (-1,.75) arc (90:-90:.375 and .75);
\path[fill=blue] (0,-1.5) arc (270:90:.75 and 1.5) to (-1,.75) arc (90:270:.375 and .75);
\path[fill=red] (3,-1.5) arc (-90:90:.75 and 1.5) to (-4,1.5) arc (90:-90:.75 and 1.5);
\path[fill=red] (3,-1.5) arc (270:90:.75 and 1.5) to (-4,1.5) arc (90:270:.75 and 1.5);
%right
\draw[gray] (3,-2) arc (-90:90:1 and 2);
\draw[gray, dashed] (3,-2) arc (270:90:1 and 2);
\draw[gray] (3,-.25) arc (-90:90:0.125 and .25);
\draw[gray, dashed] (3,-.25) arc (270:90:0.125 and .25);
\draw[very thick] (3,0) circle (.75 and 1.5);
\draw[very thick] (3,0) circle (.375 and .75);
%middle
\draw[gray] (-4,2) to (3,2);
\draw[gray] (-4,-2) to (3,-2);
\draw[gray] (-4,.25) to (3,.25);
\draw[gray] (-4,-.25) to (3,-.25);
\draw[very thick, red] (0,-1.5) arc (270:90:.75 and 1.5);
\draw[very thick,red] (-1,-.75) arc (270:90:.375 and .75);
\draw (-4,1.5) to (3,1.5);
\draw (-4,-1.5) to (3,-1.5);
\draw (-4,.75) to (3,.75);
\draw (-4,-.75) to (3,-.75);
\draw[very thick, red] (0,-1.5) arc (-90:90:.75 and 1.5);
\draw[very thick,red] (-1,-.75) arc (-90:90:.375 and .75);
%left
\draw[gray] (-4,0) circle (1 and 2);
\draw[gray] (-4,0) circle (.125 and .25);
\draw[very thick] (-4,0) circle (.75 and 1.5);
\draw[very thick] (-4,0) circle (.375 and .75);
%labels
\node[opacity=1] at (0,1.125) {$\bullet^r$};
\node[red, opacity=1] at (3,1) {\tiny$_{a_{i+1}}$};
\node[red, opacity=1] at (3,.375) {\tiny$_{a_i}$};
\end{tikzpicture}
};
\endxy
\quad , \quad 
F_{j,s} 1_\l \mapsto
\xy
(0,0)*{
\begin{tikzpicture}[fill opacity=.2, scale=.5]
%shading
\path[fill=red] (3,-.75) arc (-90:90:.375 and .75) to (-4,.75) arc (90:-90:.375 and .75);
\path[fill=red] (3,-.75) arc (270:90:.375 and .75) to (-4,.75) arc (90:270:.375 and .75);
\path[fill=blue] (0,-.75) arc (-90:90:.375 and .75) to (-1,1.5) arc (90:-90:.75 and 1.5);
\path[fill=blue] (0,-.75) arc (270:90:.375 and .75) to (-1,1.5) arc (90:270:.75 and 1.5);
\path[fill=red] (3,-1.5) arc (-90:90:.75 and 1.5) to (-4,1.5) arc (90:-90:.75 and 1.5);
\path[fill=red] (3,-1.5) arc (270:90:.75 and 1.5) to (-4,1.5) arc (90:270:.75 and 1.5);
%right
\draw[gray] (3,-2) arc (-90:90:1 and 2);
\draw[gray, dashed] (3,-2) arc (270:90:1 and 2);
\draw[gray] (3,-.25) arc (-90:90:0.125 and .25);
\draw[gray, dashed] (3,-.25) arc (270:90:0.125 and .25);
\draw[very thick] (3,0) circle (.75 and 1.5);
\draw[very thick] (3,0) circle (.375 and .75);
%middle
\draw[gray] (-4,2) to (3,2);
\draw[gray] (-4,-2) to (3,-2);
\draw[gray] (-4,.25) to (3,.25);
\draw[gray] (-4,-.25) to (3,-.25);
\draw[very thick, red] (-1,-1.5) arc (270:90:.75 and 1.5);
\draw[very thick,red] (0,-.75) arc (270:90:.375 and .75);
\draw (-4,1.5) to (3,1.5);
\draw (-4,-1.5) to (3,-1.5);
\draw (-4,.75) to (3,.75);
\draw (-4,-.75) to (3,-.75);
\draw[very thick, red] (-1,-1.5) arc (-90:90:.75 and 1.5);
\draw[very thick,red] (0,-.75) arc (-90:90:.375 and .75);
%left
\draw[gray] (-4,0) circle (1 and 2);
\draw[gray] (-4,0) circle (.125 and .25);
\draw[very thick] (-4,0) circle (.75 and 1.5);
\draw[very thick] (-4,0) circle (.375 and .75);
%labels
\node[opacity=1] at (-1,1.125) {$\bullet^s$};
\node[red, opacity=1] at (3,1) {\tiny$_{a_{i+1}}$};
\node[red, opacity=1] at (3,.375) {\tiny$_{a_i}$};
\end{tikzpicture}
};
\endxy
\]
where we have omitted the remaining parallel cylindrical foam components in the images and, 
as usual, $0$-labeled facets are understood to be deleted.
\end{prop}

In Section \ref{sH}, we'll see that the functor giving $\saKh$ is induced by the one in Proposition \ref{prop:vTrFunctor} 
and skew Howe duality.

%%%%%%%%%%%%%%%%%%%%%%%%%%%%%%%%%%%%%%%%%%%%
\subsection{Horizontal trace}
%%%%%%%%%%%%%%%%%%%%%%%%%%%%%%%%%%%%%%%%%%%%

A second notion of trace for a $2$-category $\cal{C}$ is introduced in \cite{BHLZ}, the \emph{horizontal trace} $\hTr(\cal{C})$; 
the definition, which is somewhat complicated in full generality, 
simplifies for $2$-categories in which every 1-morphism admits a biadjoint. 
Since $\cal{U}_Q(\slm)$ and $\Foam{2}{}$ are such 2-categories, the following two propositions are immediate.
\begin{prop}
The horizontal trace $\hTr(\cal{U}_Q(\slm))$ of $\cal{U}_Q(\slm)$ is equivalent to the category whose objects are direct sums of 1-morphisms in 
$\cal{U}_Q(\slm)$ with equal domain and codomain objects, and whose morphisms are matrices of linear combinations of degree-zero 
$\cal{U}_Q(\slm)$ string diagrams between such 1-morphisms, drawn on a cylinder.
\end{prop}

For example, the following picture
\begin{equation}\label{hTrex}
\xy
(0,0)*{
\begin{tikzpicture}[scale=.3]
%bottom
\draw[gray] (2,0) arc (360:180:2 and 1);
\draw[gray, dashed] (2,0) arc (0:180:2 and 1);
%middle
\draw[gray] (2,0) to (2,7);
\draw[gray] (-2,0) to (-2,7);
\draw[thick, dashed] (2,1) to [out=90,in=270] (-2,3);
\draw[thick, dashed] (2,2) to [out=90,in=270] (-2,4);
\draw[thick, directed=.55] (.5,-.95) to [out=90,in=270] (2,1);
\draw[thick,directed=.55] (-2,3) to [out=90,in=270] (.5,6.05);
\draw[thick,rdirected=.55] (-.5,-.95) to [out=90,in=270] (2,2);
\draw[thick,rdirected=.55] (-2,4) to [out=90,in=270] (-.5,6.05);
%top
\draw[gray] (0,7) circle (2 and 1);
%labels
\node at (1,5) {\tiny$\lambda$};
\end{tikzpicture}
};
\endxy
\end{equation}
denotes a morphism $\sF \sE \onel \to \sF \sE \onel$ in $\hTr(\cal{U}_Q(\slnn{2}))$. 
In general, we can recover the (co)domain from such a ``cylindrical string diagram'' from the 
boundary points, which are required to be on the front of the top and bottom circles.
We have a similar result in the foam setting:

\begin{prop}
The horizontal trace $\hTr(\Foam{2}{})$ of $\Foam{2}{}$ is equivalent to $\AFoam$.
\end{prop}

\begin{proof}
The equivalence $\hTr(\Foam{2}{}) \to \AFoam$ is given on objects \eg by sending 
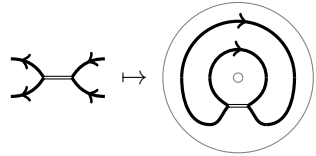
\[
\xy
(0,0)*{
\begin{tikzpicture}[scale=.5]
\draw[double] (.375,0) to (-.375,0);
\draw[very thick, directed=.45] (1.25,.5) to [out=180,in=60] (.375,0);
\draw[very thick, directed=.45] (1.25,-.5) to [out=180,in=300] (.375,0);
\draw[very thick, directed=.75] (-.375,0) to [out=120,in=0] (-1.25,.5);
\draw[very thick, directed=.75] (-.375,0) to [out=240,in=0] (-1.25,-.5);
\end{tikzpicture}
};
\endxy
\; \mapsto \;
\xy
(0,0)*{
\begin{tikzpicture}[scale=.5]
\draw[gray] (0,0) circle (2);
\draw[gray] (0,0) circle (.125);
\draw[double] (.25,-.75) to (-.25,-.75);
\draw[very thick] (.75,-1.25) to [out=180,in=315] (.25,-.75);
\draw[very thick] (-.25,-.75) to [out=225,in=0] (-.75,-1.25);
\draw[very thick, directed=.55] (-.75,0) arc (180:0:.75);
\draw[very thick, directed=.55] (-1.5,0) arc (180:0:1.5);
\draw[very thick] (-.25,-.75) to [out=150,in=270] (-.75,0);
\draw[very thick] (.75,0) to [out=270,in=30] (.25,-.75);
\draw[very thick] (-.75,-1.25) to [out=180,in=270] (-1.5,0);
\draw[very thick] (1.5,0) to [out=270,in=0] (.75,-1.25);
\end{tikzpicture}
};
\endxy
\]
where the first, non-annular web is viewed as an object of $\hTr(\Foam{2}{})$. 
The equivalence is similarly given on 2-morphisms by ``filling in'' an enhanced foam embedded in $\cal{A} \times[0,1]$ 
(which possibly wraps around the annulus) with vertical sheets.
\end{proof}

Both $\hTr(\cal{U}_Q(\slm))$ and $\hTr(\Foam{2}{})$ inherit gradings, and the latter coincides with 
the grading on $\AFoam$. Given a $2$-morphism $\mathsf{D}$ in either $\cal{U}_Q(\slm)$ or $\Foam{2}{}$ with 
equal domain and codomain objects, we'll denote by $\hTr(\mathsf{D})$ the obvious horizontal closure in 
$\hTr(\cal{U}_Q(\slm))$ or $\hTr(\Foam{2}{})$, \eg 
\[
\hTr \left(
\;
\xy
(0,0)*{
\begin{tikzpicture}[scale=.5]
\draw[thick, <-] (0,-1) to [out=90,in=180] (.5,0) to [out=0,in=90] (1,-1);
\node at (1.25,-1) {\tiny$_i$};
\node at (1,.5) {\tiny$_\lambda$};
\end{tikzpicture}
};
\endxy
\right) =
\xy
(0,0)*{
\begin{tikzpicture}[scale=.3]
%bottom
\draw[gray] (2,0) arc (360:180:2 and 1);
\draw[gray, dashed] (2,0) arc (0:180:2 and 1);
%middle
\draw[gray] (2,0) to (2,5);
\draw[gray] (-2,0) to (-2,5);
\draw[thick, directed=.5] (1,-.9) to [out=90,in=0] (0,2) to [out=180,in=90] (-1,-.9);
%top
\draw[gray] (0,5) circle (2 and 1);
%labels
\node at (1,3) {\tiny$\lambda$};
\node at (1.25,-.25) {\tiny$_i$};
\end{tikzpicture}
};
\endxy
\quad .
\]
Note, however, that there exist morphisms in $\hTr{(\cal{C})}$ which are not the horizontal closure 
of a $2$-morphism in $\cal{C}$, for example the 2-morphism in equation \eqref{hTrex}.
Observe also that functoriality of $\hTr$ implies that for each $m,N$ there exists a functor
\begin{equation}\label{hTrFunctor}
\hTr(\cal{U}_Q(\slm)) \xrightarrow{\phi^{m,N}} \AFoam
\end{equation}
which is given by the cylindrical/annular extension of equation \eqref{LQRthm}.

Next, recall from \cite{BHLZ} that there is a nice relationship between the vertical and horizontal traces of a $2$-category.
More specifically, given a $2$-category $\cal{C}$, there is a fully faithful functor
\begin{equation}\label{vTrinhTr}
\vTr(\cal{C}) \xhookrightarrow{\iota} \hTr(\cal{C})
\end{equation}
which is given on objects by sending $c \in \mathrm{Ob}(\cal{C})$ to the identity 1-morphism $\onell{c}$ of that object.
For the $2$-categories considered here, this functor is given on morphisms by simply turning the cylinder (or $\cal{A} \times [0,1]$) sideways, \eg 
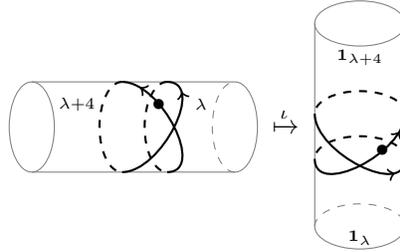
\[
\xy
(0,0)*{
\begin{tikzpicture}[scale=.3]
%left
\draw[gray] (-5,0) circle (1 and 2);
%middle
\draw[gray] (-5,2) to (4,2);
\draw[gray] (-5,-2) to (4,-2);
\draw[thick, dashed] (-1,2) arc (90:270:1 and 2); 
\draw[thick, dashed] (1,2) arc (90:270:1 and 2);
\draw[thick, directed=.85] (1,-2) to [out=0,in=0] (-1,2);
\draw[thick, directed=.85] (-1,-2) to [out=0,in=0] (1,2);
%right
\draw[gray] (4,-2) arc (-90:90:1 and 2);
\draw[gray, dashed] (4,-2) arc (270:90:1 and 2);
%labels
\node at (2.5,1) {\tiny$\lambda$};
\node at (-3,1) {\tiny$\lambda{+}4$};
\node at (.625,1) {$\bullet$};
\end{tikzpicture}
};
\endxy
\; \xmapsto{\iota} \;
\xy
(0,0)*{
\begin{tikzpicture}[scale=.3,rotate=270]
%left
\draw[gray] (-5,0) circle (1 and 2);
%middle
\draw[gray] (-5,2) to (4,2);
\draw[gray] (-5,-2) to (4,-2);
\draw[thick, dashed] (-1,2) arc (90:270:1 and 2); 
\draw[thick, dashed] (1,2) arc (90:270:1 and 2);
\draw[thick, directed=.85] (1,-2) to [out=0,in=0] (-1,2);
\draw[thick, directed=.85] (-1,-2) to [out=0,in=0] (1,2);
%right
\draw[gray] (4,-2) arc (-90:90:1 and 2);
\draw[gray, dashed] (4,-2) arc (270:90:1 and 2);
%labels
\node at (4.5,0) {\tiny$\onel$};
\node at (-3.5,0) {\tiny$\onell{\lambda{+}4}$};
\node at (.625,1) {$\bullet$};
\end{tikzpicture}
};
\endxy
\]
and, in the case of $\cal{U}_Q(\slm)$ or $\Foam{2}{}$, shifting the gradings on the domain and codomain in $\hTr$ 
by the degree of the morphism in $\vTr$.

In our setting, we can actually extend the inclusion \eqref{vTrinhTr}, 
using the relationship between $\cal{U}_Q(\slm)$ and $\Foam{2}{}$ and their non-graded forms. 
Let $\cal{D}$ denote either $\cal{U}_Q(\slm)$ or $\Foam{2}{}$ and note that the morphisms in $\vTr(\cal{D}^*)$ inherit 
gradings via the grading on the corresponding 2-morphism in $\cal{D}$, independent of the choice of representative 
(in $\vTr(\cal{U}_Q^*(\slm))$, this is simply the grading on the current algebra). 
Following the discussion in \cite[Section 6]{BN2}, we can consider a graded version of $\vTr(\cal{D}^*)$ by introducing a formal 
grading shift $q^k$ on objects, imposing the condition that a morphism $q^{k_1} d_1 \xrightarrow{\mathsf{D}} q^{k_2} d_2$ has 
degree\footnote{Our somewhat backwards looking convention is chosen to best match the two existing conventions that 
a dot in $\cal{U}_Q(\mf{g})$ has degree two, and that the Khovanov differential ``raises'' the power of $q$.} $\deg(\mathsf{D})+k_1-k_2$, 
and restricting to degree-zero morphisms.

Denoting the additive closure of this graded version of the vertical trace of $\cal{D}^*$ by $\widetilde{\vTr}(\cal{D})$, we have the following enhancement of 
equation \eqref{vTrinhTr}.

\begin{prop}\label{prop:hTrReduction}
There is a degree-preserving, fully faithful functor $\grvTr(\cal{D}) \xhookrightarrow{\iota} \hTr(\cal{D})$.
\end{prop}
\begin{proof}
The functor is again given on objects by sending $q^k d \in \mathrm{Ob}(\grvTr(\cal{D}))$ to $q^k \onell{d} \in \mathrm{Ob}(\hTr(\cal{D}))$ and 
on morphisms by turning the cylinder (in the case that $\cal{D} = \cal{U}_Q(\slm)$) or $\cal{A} \times [0,1]$ (when $\cal{D} = \Foam{2}{}$) 
sideways. This functor preserves degree, as the degrees of morphisms in both $\grvTr(\cal{D})$ and $\hTr(\cal{D})$ are computed locally.
\end{proof}

This proposition, together with functoriality of the traces and the fact that 
$\cal{U}_Q(\glm) \cong \coprod_{\Z} \cal{U}_Q(\slm)$, now gives the diagram
\[
\xymatrix{
\grvTr(\cal{U}_Q(\glm)) \ar[rr] \ar@{^{(}->}[d] & & \grvTr(2\cat{Foam}) \ar@{^{(}->}[d] \\
\hTr(\cal{U}_Q(\glm)) \ar[rr] & & \hTr(2\cat{Foam}).
} 
\]
Since the horizontal morphisms clearly factor through the traces of $\cal{U}_Q(\glm)^{0 \leq 2}$, 
this gives the commutative diagram in equation \eqref{maindiag}.

% ########################################################################################################
%
\section{Skew Howe duality and $\saKh$}\label{sH} %%%%%%%%%%%%%%%%%%%%%%%%%%%%%%%%%%%%%%%%%%%%%%%%%%%%%%%%%%%%%%%%%%%%%%%
%
% ########################################################################################################

We begin this section by quickly recalling skew Howe duality, \emph{\`{a} la} Cautis, Kamnitzer, Licata, and Morrison  \cite{CKL,CKM}. 
We then show how the diagram in equation \eqref{maindiag} pairs with this to recover sutured annular Khovanov homology.

\subsection{A review of skew Howe duality}

Consider the vector space $\bigwedge^N(\C^n \otimes \C^m)$, which carries commuting actions of $\sln$ and $\glm$. 
The $\glm$ weight space decomposition 
\begin{equation}\label{eqskewHowe}
{\textstyle \bigwedge^N}(\C^n \otimes \C^m) \cong \bigoplus_{N=\sum a_i} \wedge^{a_1} \C^n \otimes \cdots \otimes \wedge^{a_m} \C^n
\end{equation}
gives a functor $\U(\glm) \to \Rep(\sln)$. 
Here $\U(\glm)$ is the idempotented enveloping algebra of $\glm$, $\Rep(\sln)$ is the category of finite-dimensional $\sln$-modules, 
and the functor sends a $\glm$ weight $\mathbf{a} = [a_1,\ldots,a_m]$ to the $\sln$ representation 
$\wedge^{a_1} \C^n \otimes \cdots \otimes \wedge^{a_m} \C^n$, and the elements 
$E_i 1_{\mathbf{a}}$ and $F_i 1_{\mathbf{a}}$ to the morphisms of $\sln$ representations determined by the $\glm$ action 
in equation \eqref{eqskewHowe}. 
As in Remark \ref{glmVslm}, we can view $\U(\glm)$ as $\coprod_{\Z} \U(\slm)$, 
so for each $N \geq 0$ we obtain functors $\U(\slm) \to \Rep(\sln)$, which send an $\sln$ weight 
$\lambda = (\lambda_1,\ldots,\lambda_{m-1})$ to $\wedge^{a_1} \C^n \otimes \cdots \otimes \wedge^{a_m} \C^n$ 
where $\lambda_i = a_i - a_{i+1}$ and $N=\sum a_i$ (and to the zero representation otherwise). 
These (or more precisely their quantum versions) are the functors categorified by equation \eqref{LQRthm}.

Our next result will be used to construct our annular link invariant.
Before stating it, we first introduce the category $\grRep(\slnn{2})$, 
a version of the category of finite-dimensional $\slnn{2}$ representations to which we've introduced a formal $q$-grading. 
Its objects are direct sums of $\slnn{2}$ representations which are formally shifted in $q$-degree, 
\eg they take the form $\bigoplus_{i=1}^{l} q^{k_i} V_i$, where each $V_i$ is an $\slnn{2}$ representation.
Morphisms are given by $q$-degree-zero maps of $\slnn{2}$ representations, \ie we have
\[
\Hom
(q^{k_1}V,q^{k_2}W) = 
\begin{cases}
\Hom
(V,W) & \text{if } k_1 = k_2 \\
0 & \text{else}
\end{cases}
\]
hence $\grRep(\slnn{2}) \cong \bigoplus_\Z \Rep(\slnn{2})$. 
Note that each finite-dimensional $\slnn{2}$-module is itself graded via its weight space decomposition (this is precisely the $w$ grading appearing earlier),
so the vector spaces in $\grRep(\slnn{2})$ are bi-graded.

\begin{lem}\label{lem:sl2Rep}
The skew Howe duality functor induces a functor $\grvTr(\cal{U}_Q(\glm)) \xrightarrow{\SHz} \grRep(\slnn{2})$ which factors through $\grvTr(\Foam{2}{})$.
The induced functor $\grvTr(\Foam{2}{}) \xrightarrow{\SH} \grRep(\slnn{2})$ is monoidal, maps essential 1-labeled circles to the vector representation $V_1$, 
and maps essential $2$-labeled circles to the trivial representation.
\end{lem}

\begin{proof}
The functor $\grvTr(\cal{U}_Q(\glm)) \xrightarrow{\SHz} \grRep(\slnn{2})$ is given by ``setting $t=0$'' and then applying skew Howe duality. 
Explicitly, an object in $\grvTr(\cal{U}_Q(\glm))$ is a direct sum of objects of the form $q^k [a_1,\ldots,a_m]$, 
and we send each of these objects to the $q$-graded $\slnn{2}$ representation $q^k (\wedge^{a_1} \C^2 \otimes \cdots \otimes \wedge^{a_m} \C^2)$ 
if every $a_i \in \{0,1,2\}$ and to the zero representation otherwise, and then extend to direct sums. 
Morphisms in $\grvTr(\cal{U}_Q(\glm))$ are given by $q$-degree shifts of elements in $\dcurm$, 
and the functor is given by sending $E_{i,r}1_\l , F_{i,r}1_\l, H_{i,r}1_\l \mapsto 0$ if $r>0$, 
and using skew Howe duality if $r=0$. Here, we use the fact that $\U(\slm)$ is the $t$-degree-zero piece of $\dcurm$.

In \cite{QR}, we showed that $\Foam{2}{}$ is the quotient of $\displaystyle \lim_{\substack{\longrightarrow\\m}}\Ucatc_Q(\glm)$ by the ideal generated by 
(the identity 2-morphisms of identity 1-morphisms of) weights with entries not in $\{0, 1, 2\}$, hence the same is true after taking traces. The result then follows since 
$\grvTr(\Ucatc_Q(\glm)) \cong \grvTr(\cal{U}_Q(\glm))$ and
$\grvTr(\cal{U}_Q(\glm)) \xrightarrow{\SHz} \grRep(\slnn{2})$ kills such weights.
\end{proof}

\subsection{The annular $\slnn{2}$ invariant}%%%%%%%%%%%%%%%%%%%%%%%%%%%%%%%%%%%%%%%%%%%%%%%%%%%%%%%%%%%%%%%%%%%%%%%%%

We'll now use the functor in Lemma \ref{lem:sl2Rep} to reconstruct sutured annular Khovanov homology.
Recall from \cite{LQR1} that equation \eqref{eskein} assigns a complex of enhanced webs and foams to any framed tangle. 
Since any framed, annular link can be presented as the annular closure of such a tangle, 
this same equation assigns a complex $C(\cal{L})$ in $\AFoam \cong \hTr(\Foam{2}{})$ to any such link $\cal{L}$. 
Our first result shows that this complex is in fact isomorphic to a complex in $\grvTr(\Foam{2}{})$.

\begin{prop}
The inclusion $\grvTr(\Foam{2}{}) \hookrightarrow \hTr(\Foam{2}{})$ is an equivalence of categories.
\end{prop}

\begin{proof}
Since $\grvTr(\Foam{2}{})$ embeds as a full subcategory of $\hTr(\Foam{2}{})$, 
it suffices to show that every object in the latter is isomorphic to one in the former. 
In other words, it suffices to show that every annular enhanced $\slnn{2}$ web is isomorphic 
to a direct sum of degree shifts of tensor products of essential 1- and 2-labeled circles. 
This then follows using the local web isomorphisms
\begin{gather}
\label{eq:sl2webiso}
\xy
(0,0)*{
\begin{tikzpicture}[scale=.5]
\draw[double] (3,0) to (2.25,0);
\draw[very thick, directed=.55] (2.25,0) to [out=225,in=315] (.75,0);
\draw[very thick, directed=.55] (2.25,0) to (1.75,.5);
\draw[double] (1.75,.5) to (1.25,.5);
\draw[very thick, directed=.55] (1.25,.5) to (.75,0);
\draw[double] (.75,0) to (0,0);
\draw[very thick, directed=.55] (3,1) to [out=180,in=45] (1.75,.5);
\draw [very thick,->] (1.25,.5) to [out=135,in=0] (0,1);
\end{tikzpicture}
};
\endxy
\cong 
\xy
(0,0)*{
\begin{tikzpicture}[scale=.5]
\draw[double] (2.5,0) to (0,0);
\draw[very thick, directed=.55] (2.5,1) to (0,1);
\end{tikzpicture}
};
\endxy
\;\; , \;\;
\xy
(0,0)*{
\begin{tikzpicture}[scale=.5]
\draw[double] (3,0) to (2.25,0);
\draw[very thick, directed=.55] (2.25,0) to [out=135,in=45] (.75,0);
\draw[very thick, directed=.55] (2.25,0) to (1.75,-.5);
\draw[double] (1.75,-.5) to (1.25,-.5);
\draw[very thick, directed=.55] (1.25,-.5) to (.75,0);
\draw[double] (.75,0) to (0,0);
\draw[very thick, directed=.55] (3,-1) to [out=180,in=315] (1.75,-.5);
\draw [very thick,->] (1.25,-.5) to [out=225,in=0] (0,-1);
\end{tikzpicture}
};
\endxy
\cong 
\xy
(0,0)*{
\begin{tikzpicture}[scale=.5]
\draw[double] (2.5,1) to (0,1);
\draw[very thick, directed=.55] (2.5,0) to (0,0);
\end{tikzpicture}
};
\endxy \\
\nonumber
\xy
(0,0)*{
\begin{tikzpicture}[scale=.5]
\draw[double] (3,0) to (2.25,0);
\draw[very thick, directed=.55] (2.25,0) to [out=225,in=315] (.75,0);
\draw[very thick, directed=.55] (2.25,0) to [out=135,in=45] (.75,0);
\draw[double] (.75,0) to (0,0);
\end{tikzpicture}
};
\endxy
\cong
q
\xy
(0,0)*{
\begin{tikzpicture}[scale=.5]
\draw[double] (2,0) to (0,0);
\end{tikzpicture}
};
\endxy
\oplus
q^{-1}
\xy
(0,0)*{
\begin{tikzpicture}[scale=.5]
\draw[double] (2,0) to (0,0);
\end{tikzpicture}
};
\endxy
\end{gather}
which are given \eg as the $n=2$ case of \cite[Remarks 3.6 and 3.12]{QR}.
It follows, either from a Euler characteristics argument or as in Lemma \ref{lem:annEvalAlgo} below, 
that these moves are sufficient to decompose any annular web into nested circles.
\end{proof}

Noting that the isomorphisms in equation \eqref{eq:sl2webiso} can be lifted to $\cal{U}_Q(\glm)^{0\leq2}$, 
we have actually shown the following.

\begin{cor}
The inclusion $\grvTr(\cal{U}_Q(\glm)^{0\leq2}) \hookrightarrow \hTr(\cal{U}_Q(\glm)^{0\leq2})$ is an equivalence of categories.
\end{cor}

Given a framed annular link $\cal{L}$, we can now pass from $C(\cal{L})$ to an isomorphic complex $\tilde{C}(\cal{L})$ which lies in $\grvTr(\Foam{2}{})$. 
Applying the functor $\grvTr(\Foam{2}{}) \xrightarrow{\SH} \grRep(\slnn{2})$ from Lemma \ref{lem:sl2Rep},
we obtain a complex $\SH(\tilde{C}(\cal{L}))$ in $\grRep(\slnn{2})$.

\begin{prop}
The homology of the complex $\SH(\tilde{C}(\cal{L}))$ is isomorphic to $\saKh(\cal{L})$.
\end{prop}
\begin{proof}

It is easy to check that the chain groups in the complex $\SH(\tilde{C}(\cal{L}))$ are isomorphic to those in the complex 
used to compute $\saKh(\cal{L})$ (up to shifts coming from differing grading conventions), 
so it remains to check that the differentials in $\SH(\tilde{C}(\cal{L}))$ give a complex isomorphic to the one used for $\saKh(\cal{L})$. 
By a standard argument (see \eg \cite{ORS}), it suffices to check that the maps on the edges of the ``cube of resolutions'' 
assigned to a link in both theories agree up to a sign. 
All of the maps in the $\saKh$ cube of resolutions are given by saddle cobordisms between circles, 
so we analyze the images of the analogous foams, 
and compare them to the maps assigned to the saddle cobordisms determined by equation \eqref{safunctor}.

We first consider the differential corresponding to an ``enhanced saddle'' in $\AFoam$ between two essential circles and one enhanced inessential circle:
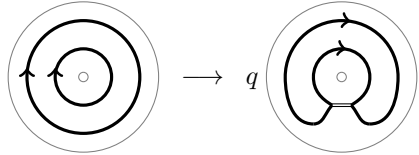
\[
\xy
(0,0)*{
\begin{tikzpicture}[scale=.5]
\draw[gray] (0,0) circle (2);
\draw[gray] (0,0) circle (.125);
\draw[very thick, rdirected=.5] (0,0) circle (.75);
\draw[very thick, rdirected=.5] (0,0) circle (1.5);
\end{tikzpicture}
};
\endxy
\;\;
\longrightarrow
\;\;
q \;
\xy
(0,0)*{
\begin{tikzpicture}[scale=.5]
\draw[gray] (0,0) circle (2);
\draw[gray] (0,0) circle (.125);
\draw[double] (.25,-.75) to (-.25,-.75);
\draw[very thick] (.75,-1.25) to [out=180,in=315] (.25,-.75);
\draw[very thick] (-.25,-.75) to [out=225,in=0] (-.75,-1.25);
\draw[very thick, directed=.55] (-.75,0) arc (180:0:.75);
\draw[very thick, directed=.55] (-1.5,0) arc (180:0:1.5);
\draw[very thick] (-.25,-.75) to [out=150,in=270] (-.75,0);
\draw[very thick] (.75,0) to [out=270,in=30] (.25,-.75);
\draw[very thick] (-.75,-1.25) to [out=180,in=270] (-1.5,0);
\draw[very thick] (1.5,0) to [out=270,in=0] (.75,-1.25);
\end{tikzpicture}
};
\endxy
\]
which is the image under $\hTr(\Phi)$ of the morphism
\[
\onenn{[1,1]}
\xrightarrow{
\hTr\left(\;
\xy
(0,0)*{
\begin{tikzpicture}[scale=.5]
\draw[thick, ->] (0,1) to [out=-90,in=180] (.5,0) to [out=0,in=-90] (1,1);
\end{tikzpicture}
};
\endxy
\; \right)}
q
\cal{F}\cal{E}\onenn{[1,1]}
\]
in $\hTr(\Ucat_Q(\glnn{2}))$. 
Although the latter term is indecomposable in $\Ucat_Q(\glnn{2})^{0 \leq 2}$, the following map is an isomorphism in $\hTr(\Ucat_Q(\glnn{2})^{0 \leq 2})$:
\begin{equation}\label{eq:FEannularIso1}
\begin{gathered}
\cal{F}\cal{E}\onenn{[1,1]}
\xrightarrow{
\;
\begin{pmatrix} \;
\xy
(0,0)*{
\begin{tikzpicture}[scale=.3]
%bottom
\draw[gray] (-2,-4) arc (-180:0:2 and 1);
\draw[gray, dashed] (-2,-4) arc (180:0:2 and 1);
%top
\draw[gray] (0,2) circle (2 and 1);
%middle
\draw[gray] (2,2) to (2,-4);
\draw[gray] (-2,2) to (-2,-4);
\draw[thick] (1,-4.9) to [out=90,in=-90] (2,-1);
\draw[thick, dashed] (2,-1) arc (0:180:2 and 1); 
\draw[thick,<-] (-1,-4.9) to [out=90,in=-90] (-2,-1);
% %labels
% \node at (4.5,0) {\tiny$\lambda$};
% \node at (-3.5,0) {\tiny$\lambda+4$};
% \node at (.625,1) {$\bullet$};
\end{tikzpicture}
};
\endxy
\;
& , \quad
\xy
(0,0)*{
\begin{tikzpicture}[scale=.3]
%bottom
\draw[gray] (-2,-4) arc (-180:0:2 and 1);
\draw[gray, dashed] (-2,-4) arc (180:0:2 and 1);
%top
\draw[gray] (0,2) circle (2 and 1);
%middle
\draw[gray] (2,2) to (2,-4);
\draw[gray] (-2,2) to (-2,-4);
\draw[thick] (1,-4.9) to [out=90,in=-90] (2,-1);
\draw[thick, dashed] (2,-1) arc (0:180:2 and 1); 
\draw[thick,<-] (-1,-4.9) to [out=90,in=-90] (-2,-1);
\node at (-1.5,-3) {$\bullet$};
% %labels
% \node at (4.5,0) {\tiny$\lambda$};
% \node at (-3.5,0) {\tiny$\lambda+4$};
% \node at (.625,1) {$\bullet$};
\end{tikzpicture}
};
\endxy
\;
\end{pmatrix}
}
q^{-1} \onenn{[2,0]} \oplus q \onenn{[2,0]}
\end{gathered}
\end{equation}
with inverse given by
\begin{equation}\label{eq:FEannularIso2}
\begin{gathered}
q^{-1} \onenn{[2,0]} \oplus q \; \onenn{[2,0]}
\xrightarrow{
\begin{pmatrix} \;
\xy
(0,0)*{
\begin{tikzpicture}[scale=.3]
%bottom
\draw[gray] (-2,-1) arc (-180:0:2 and 1);
\draw[gray, dashed] (-2,-1) arc (180:0:2 and 1);
%top
\draw[gray] (0,5) circle (2 and 1);
%middle
\draw[gray] (2,5) to (2,-1);
\draw[gray] (-2,5) to (-2,-1);
\draw[thick,<-] (1,4.2) to [out=-90,in=90] (2,1.5);
\draw[thick, dashed] (2,1.5) arc (0:-180:2 and 1); 
\draw[thick] (-1,4.2) to [out=-90,in=90] (-2,1.5);
\node at (1.5,2.75) {$\bullet$};
% %labels
% \node at (4.5,0) {\tiny$\lambda$};
% \node at (-3.5,0) {\tiny$\lambda+4$};
% \node at (.625,1) {$\bullet$};
\end{tikzpicture}
};
\endxy
\;+\;
\xy
(0,0)*{
\begin{tikzpicture}[scale=.3]
%bottom
\draw[gray] (-2,-1) arc (-180:0:2 and 1);
\draw[gray, dashed] (-2,-1) arc (180:0:2 and 1);
%top
\draw[gray] (0,5) circle (2 and 1);
%middle
\draw[gray] (2,5) to (2,-1);
\draw[gray] (-2,5) to (-2,-1);
\draw[thick,<-] (1,4.2) to [out=-90,in=90] (2,1.5);
\draw[thick, dashed] (2,1.5) arc (0:-180:2 and 1); 
\draw[thick] (-1,4.2) to [out=-90,in=90] (-2,1.5);
\draw [thick,->] (.75,0) arc (0:360:.75);
\node at (0,-.75) {$\bullet$};
\node at (0,-1.25) {\tiny $-2$};
% %labels
% \node at (4.5,0) {\tiny$\lambda$};
% \node at (-3.5,0) {\tiny$\lambda+4$};
% \node at (.625,1) {$\bullet$};
\end{tikzpicture}
};
\endxy
\; & , \quad
\:
\xy
(0,0)*{
\begin{tikzpicture}[scale=.3]
%bottom
\draw[gray] (-2,-1) arc (-180:0:2 and 1);
\draw[gray, dashed] (-2,-1) arc (180:0:2 and 1);
%top
\draw[gray] (0,5) circle (2 and 1);
%middle
\draw[gray] (2,5) to (2,-1);
\draw[gray] (-2,5) to (-2,-1);
\draw[thick,<-] (1,4.2) to [out=-90,in=90] (2,1.5);
\draw[thick, dashed] (2,1.5) arc (0:-180:2 and 1); 
\draw[thick] (-1,4.2) to [out=-90,in=90] (-2,1.5);
% %labels
% \node at (4.5,0) {\tiny$\lambda$};
% \node at (-3.5,0) {\tiny$\lambda+4$};
% \node at (.625,1) {$\bullet$};
\end{tikzpicture}
};
\endxy
\; \end{pmatrix}
}
\cal{F}\cal{E}\onenn{[1,1]}.
\end{gathered}
\end{equation}
Hence, the complex in the vertical trace corresponding to the enhanced saddle is:
\[
[1,1]
\xrightarrow{
\begin{pmatrix}
\xy
(0,0)*{
\begin{tikzpicture}[scale=.3]
%right
\draw[gray] (3,-2) arc (-90:90:1 and 2);
\draw[gray, dashed] (3,-2) arc (270:90:1 and 2);
%left
\draw[gray] (-5,0) circle (1 and 2);
%middle
\draw[gray] (-5,2) to (3,2);
\draw[gray] (-5,-2) to (3,-2);
\draw[thick, dashed] (-1,2) arc (90:270:1 and 2); 
\draw[thick, rdirected=.85] (-1,2) to [out=0,in=0] (-1,-2);
%labels
\node at (2,1) {\tiny$[1,1]$};
\node at (-3,1) {\tiny$[2,0]$};
%\node at (.625,1) {$\bullet$};
\end{tikzpicture}
};
\endxy
& , \quad
\xy
(0,0)*{
\begin{tikzpicture}[scale=.3]
%right
\draw[gray] (3,-2) arc (-90:90:1 and 2);
\draw[gray, dashed] (3,-2) arc (270:90:1 and 2);
%left
\draw[gray] (-5,0) circle (1 and 2);
%middle
\draw[gray] (-5,2) to (3,2);
\draw[gray] (-5,-2) to (3,-2);
\draw[thick, dashed] (-1,2) arc (90:270:1 and 2); 
\draw[thick, rdirected=.85] (-1,2) to [out=0,in=0] (-1,-2);
%labels
\node at (2,1) {\tiny$[1,1]$};
\node at (-3,1) {\tiny$[2,0]$};
\node at (0,1) {$\bullet$};
\end{tikzpicture}
};
\endxy
\end{pmatrix}
}
[2,0] \oplus q^2
[2,0] .
\]
Under skew Howe duality, this is mapped to:
\[
V_1 \otimes V_1 \xrightarrow{
\begin{pmatrix}
\begin{Bmatrix}
   v_-\otimes v_-\mapsto 0\\
   v_+\otimes v_-\mapsto v_0\\
   v_-\otimes v_+\mapsto -v_0\\
   v_+\otimes v_+\mapsto 0
\end{Bmatrix}
& ,\quad 
0
\end{pmatrix}}
V_0\bigoplus q^2 V_0
\]
which matches equation \eqref{safunctor}.

Similarly,
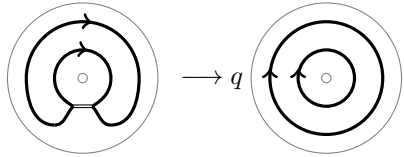
\[
\xy
(0,0)*{
\begin{tikzpicture}[scale=.5]
\draw[gray] (0,0) circle (2);
\draw[gray] (0,0) circle (.125);
\draw[double] (.25,-.75) to (-.25,-.75);
\draw[very thick] (.75,-1.25) to [out=180,in=315] (.25,-.75);
\draw[very thick] (-.25,-.75) to [out=225,in=0] (-.75,-1.25);
\draw[very thick, directed=.55] (-.75,0) arc (180:0:.75);
\draw[very thick, directed=.55] (-1.5,0) arc (180:0:1.5);
\draw[very thick] (-.25,-.75) to [out=150,in=270] (-.75,0);
\draw[very thick] (.75,0) to [out=270,in=30] (.25,-.75);
\draw[very thick] (-.75,-1.25) to [out=180,in=270] (-1.5,0);
\draw[very thick] (1.5,0) to [out=270,in=0] (.75,-1.25);
\end{tikzpicture}
};
\endxy
\;\; \longrightarrow
q \;
\xy
(0,0)*{
\begin{tikzpicture}[scale=.5]
\draw[gray] (0,0) circle (2);
\draw[gray] (0,0) circle (.125);
\draw[very thick, rdirected=.5] (0,0) circle (.75);
\draw[very thick, rdirected=.5] (0,0) circle (1.5);
\end{tikzpicture}
};
\endxy
\]
lifts to the complex 
\[
q^{-1}[2,0] \oplus q[2,0]
\xrightarrow{
\begin{pmatrix}
\xy
(0,0)*{
\begin{tikzpicture}[scale=.3]
%right
\draw[gray] (3,-2) arc (-90:90:1 and 2);
\draw[gray, dashed] (3,-2) arc (270:90:1 and 2);
%left
\draw[gray] (-5,0) circle (1 and 2);
%middle
\draw[gray] (-5,2) to (3,2);
\draw[gray] (-5,-2) to (3,-2);
\draw[thick, dashed] (-1,2) arc (90:270:1 and 2); 
\draw[thick, directed=.85] (-1,2) to [out=0,in=0] (-1,-2);
%labels
\node at (2,1) {\tiny$[2,0]$};
\node at (-3,1) {\tiny$[1,1]$};
\node at (0,1) {$\bullet$};
\end{tikzpicture}
};
\endxy
\;+\;
\xy
(0,0)*{
\begin{tikzpicture}[scale=.3]
%right
\draw[gray] (3,-2) arc (-90:90:1 and 2);
\draw[gray, dashed] (3,-2) arc (270:90:1 and 2);
%left
\draw[gray] (-5,0) circle (1 and 2);
%middle
\draw[gray] (-5,2) to (3,2);
\draw[gray] (-5,-2) to (3,-2);
\draw[thick, dashed] (-1,2) arc (90:270:1 and 2); 
\draw[thick, directed=.85] (-1,2) to [out=0,in=0] (-1,-2);
\draw [thick,directed=.75] (2,-1) arc (0:360:.75);
\node at (2,-1) {$\bullet$};
\node at (2.75,-.5) {\tiny $-2$};
%labels
\node at (2,1) {\tiny$[2,0]$};
\node at (-3,1) {\tiny$[1,1]$};
%\node at (.625,1) {$\bullet$};
\end{tikzpicture}
};
\endxy
& , \quad
\xy
(0,0)*{
\begin{tikzpicture}[scale=.3]
%right
\draw[gray] (3,-2) arc (-90:90:1 and 2);
\draw[gray, dashed] (3,-2) arc (270:90:1 and 2);
%left
\draw[gray] (-5,0) circle (1 and 2);
%middle
\draw[gray] (-5,2) to (3,2);
\draw[gray] (-5,-2) to (3,-2);
\draw[thick, dashed] (-1,2) arc (90:270:1 and 2); 
\draw[thick, directed=.85] (-1,2) to [out=0,in=0] (-1,-2);
%labels
\node at (2,1) {\tiny$[2,0]$};
\node at (-3,1) {\tiny$[1,1]$};
%\node at (.625,1) {$\bullet$};
\end{tikzpicture}
};
\endxy
\end{pmatrix}
}
\;
q[1,1]
\]
which under skew Howe duality is sent to:
\[
q^{-1}V_0\oplus qV_0
\xrightarrow{
\begin{pmatrix}
0 & , \quad
\{v_0\mapsto v_+\otimes v_--v_-\otimes v_+\}
\end{pmatrix}}
qV_1 \otimes V_1.
\]

We now turn to the analog in the foam setting of an essential circle merging with an 
inessential one under a saddle cobordism:
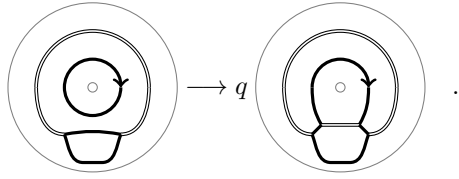
\[
\xy
(0,0)*{
\begin{tikzpicture}[scale=.5]
\draw[gray] (0,0) circle (2.25);
\draw[gray] (0,0) circle (.125);
\draw[very thick, ->] (.75,0) arc (360:0:.75);
\draw[double] (-1.5,0) arc (180:0:1.5);
\draw[double] (-.75,-1.25) to [out=180,in=270] (-1.5,0);
\draw[double] (1.5,0) to [out=270,in=0] (.75,-1.25);
\draw[very thick] (.75,-1.25) to [out=170,in=10] (-.75,-1.25);
\draw[very thick] (.75,-1.25) to [out=250,in=0] (.25,-2) to (-.25,-2) to [out=180,in=290] (-.75,-1.25);
\end{tikzpicture}
};
\endxy
\longrightarrow
q\;
\xy
(0,0)*{
\begin{tikzpicture}[scale=.5]
\draw[gray] (0,0) circle (2.25);
\draw[gray] (0,0) circle (.125);
\draw[very thick] (.75,0) to [out=270,in=60] (.5,-1);
\draw[very thick] (-.5,-1) to [out=120,in=270] (-.75,0);
\draw[very thick, ->] (-.75,0) arc (180:0:.75);
\draw[double] (-1.5,0) arc (180:0:1.5);
\draw[double] (-.75,-1.25) to [out=180,in=270] (-1.5,0);
\draw[double] (1.5,0) to [out=270,in=0] (.75,-1.25);
\draw[double] (.5,-1) to (-.5,-1);
\draw[very thick] (.75,-1.25) to (.5,-1);
\draw[very thick] (-.5,-1) to (-.75,-1.25);
\draw[very thick] (.75,-1.25) to [out=250,in=0] (.25,-2) to (-.25,-2) to [out=180,in=290] (-.75,-1.25);
\end{tikzpicture}
};
\endxy \quad .
\]
This lifts to the following complex in $\hTr( \cal{U}_Q(\glnn{3})^{0 \leq 2})$:
\[
\cal{E}_2 \cal{F}_2\onenn{[1,2,0]}
\xrightarrow{
\xy
(0,0)*{
\begin{tikzpicture}[scale=.3]
%bottom
\draw[gray] (-2,-2) arc (-180:0:2 and 1);
\draw[gray, dashed] (-2,-2) arc (180:0:2 and 1);
%top
\draw[gray] (0,3) circle (2 and 1);
%middle
\draw[gray] (2,3) to (2,-2);
\draw[gray] (-2,3) to (-2,-2);
\draw[thick, green,<-] (1,-2.8) to [out=90,in=-90] (1.6,2.4);
\draw[thick, green,->] (-1,-2.8) to [out=90,in=-90] (-1.6,2.4);
\draw [thick, red,->] (1,2.2) to [out=-90,in=0] (0,0) to [out=180,in=-90] (-1,2.2);
% %labels
% \node at (4.5,0) {\tiny$\lambda$};
% \node at (-3.5,0) {\tiny$\lambda+4$};
% \node at (.625,1) {$\bullet$};
\end{tikzpicture}
};
\endxy
}
q\;
\cal{E}_2 \cal{E}_1 \cal{F}_1 \cal{F}_2\onenn{[1,2,0]}
\]
where we color our string diagrams according to the chromatic ordering 
$\xy
(0,0)*{
\begin{tikzpicture} [scale=1]
\draw[thick] (-.5,0) to (0,0);
\node[red] at (-.5,0) {\Large$\bullet$};
\node[black!30!green] at (0,0) {\Large$\bullet$};
\node[red] at (-.5,.25) {\tiny$1$};
\node[black!30!green] at (0,.25) {\tiny$2$};
\node[red] at (-.5,-.25) {};
\node[black!30!green] at (0,-.25) {};
\end{tikzpicture}
}
\endxy$ 
of the nodes in the $\slnn{3}$ Dynkin diagram. We next note that we have the isomorphism:
\[
q^{-1} \onell{[1,2,0]} \oplus q \onell{[1,2,0]}
\xrightarrow{
\begin{pmatrix}
\xy
(0,0)*{
\begin{tikzpicture}[scale=.3]
%bottom
\draw[gray] (-2,-2) arc (-180:0:2 and 1);
\draw[gray, dashed] (-2,-2) arc (180:0:2 and 1);
%top
\draw[gray] (0,3) circle (2 and 1);
%middle
\draw[gray] (2,3) to (2,-2);
\draw[gray] (-2,3) to (-2,-2);
\draw [thick, green,->] (1,2.2) to [out=-90,in=0] (0,0) to [out=180,in=-90] (-1,2.2);
% %labels
% \node at (4.5,0) {\tiny$\lambda$};
% \node at (-3.5,0) {\tiny$\lambda+4$};
\node at (0,0) {$\bullet$};
\end{tikzpicture}
};
\endxy
+
\xy
(0,0)*{
\begin{tikzpicture}[scale=.3]
%bottom
\draw[gray] (-2,-2) arc (-180:0:2 and 1);
\draw[gray, dashed] (-2,-2) arc (180:0:2 and 1);
%top
\draw[gray] (0,3) circle (2 and 1);
%middle
\draw[gray] (2,3) to (2,-2);
\draw[gray] (-2,3) to (-2,-2);
\draw [thick, green,->] (1,2.2) to [out=-90,in=0] (0,0) to [out=180,in=-90] (-1,2.2);
\draw [thick, green,->] (.75,-1.25) arc (0:360:.75);
\node at (0,-2) {$\bullet$};
\node at (.2,-2.5) {\tiny ${-}2$};
% %labels
% \node at (4.5,0) {\tiny$\lambda$};
% \node at (-3.5,0) {\tiny$\lambda+4$};
%\node at (0,0) {$\bullet$};
\end{tikzpicture}
};
\endxy
& , \quad
\xy
(0,0)*{
\begin{tikzpicture}[scale=.3]
%bottom
\draw[gray] (-2,-2) arc (-180:0:2 and 1);
\draw[gray, dashed] (-2,-2) arc (180:0:2 and 1);
%top
\draw[gray] (0,3) circle (2 and 1);
%middle
\draw[gray] (2,3) to (2,-2);
\draw[gray] (-2,3) to (-2,-2);
\draw [thick, green,->] (1,2.2) to [out=-90,in=0] (0,0) to [out=180,in=-90] (-1,2.2);
% %labels
\end{tikzpicture}
};
\endxy
\end{pmatrix}
}
\quad
\cal{E}_2\cal{F}_2 \onell{[1,2,0]}
\]
in $\hTr(\cal{U}_Q(\glnn{3})^{0 \leq 2})$ with inverse:
\[
\cal{E}_2\cal{F}_2 \onell{[1,2,0]}
\quad
\xrightarrow{
\begin{pmatrix}
\xy
(0,0)*{
\begin{tikzpicture}[scale=.3]
%bottom
\draw[gray] (-2,-3) arc (-180:0:2 and 1);
\draw[gray, dashed] (-2,-3) arc (180:0:2 and 1);
%top
\draw[gray] (0,1) circle (2 and 1);
%middle
\draw[gray] (2,1) to (2,-3);
\draw[gray] (-2,1) to (-2,-3);
\draw [thick, green,<-] (1,-3.9) to [out=90,in=0] (0,-1.5) to [out=180,in=90] (-1,-3.9);
% %labels
% \node at (4.5,0) {\tiny$\lambda$};
% \node at (-3.5,0) {\tiny$\lambda+4$};
%\node at (0,0) {$\bullet$};
\end{tikzpicture}
};
\endxy
& ,  \quad
\xy
(0,0)*{
\begin{tikzpicture}[scale=.3]
%bottom
\draw[gray] (-2,-3) arc (-180:0:2 and 1);
\draw[gray, dashed] (-2,-3) arc (180:0:2 and 1);
%top
\draw[gray] (0,1) circle (2 and 1);
%middle
\draw[gray] (2,1) to (2,-3);
\draw[gray] (-2,1) to (-2,-3);
\draw [thick, green,<-] (1,-3.9) to [out=90,in=0] (0,-1.5) to [out=180,in=90] (-1,-3.9);
% %labels
% \node at (4.5,0) {\tiny$\lambda$};
% \node at (-3.5,0) {\tiny$\lambda+4$};
\node at (0,-1.5) {$\bullet$};
\end{tikzpicture}
};
\endxy
\end{pmatrix}
}
q^{-1} \onell{[1,2,0]} \oplus q \onell{[1,2,0]}
\]
and the isomorphism
\[
\cal{E}_2 \cal{E}_1\cal{F}_1 \cal{F}_2 \onell{[1,2,0]}
\xrightarrow{
\xy
(0,0)*{
\begin{tikzpicture}[scale=.3]
%top
\draw[gray] (0,1) circle (2 and 1);
%bottom
\draw[gray] (-2,-3) arc (-180:0:2 and 1);
\draw[gray, dashed] (-2,-3) arc (180:0:2 and 1);
%middle
\draw[gray] (2,1) to (2,-3);
\draw[gray] (-2,1) to (-2,-3);
\draw [thick, red,<-] (1,-3.9) to [out=90,in=0] (0,-1.5) to [out=180,in=90] (-1,-3.9);
\draw [thick, green,<-] (1.5,-3.7) to [out=90,in=0] (0,-.8) to [out=180,in=90] (-1.5,-3.7);
% %labels
\end{tikzpicture}
};
\endxy
}
\onell{[1,2,0]}
\]
with inverse:
\[
\onell{[1,2,0]}
\xrightarrow{
\xy
(0,0)*{
\begin{tikzpicture}[scale=.3]
%bottom
\draw[gray] (-2,-1) arc (-180:0:2 and 1);
\draw[gray, dashed] (-2,-1) arc (180:0:2 and 1);
%top
\draw[gray] (0,3) circle (2 and 1);
%middle
\draw[gray] (2,3) to (2,-1);
\draw[gray] (-2,3) to (-2,-1);
\draw [thick, red,->] (1,2.15) to [out=-90,in=0] (0,.5) to [out=180,in=-90] (-1,2.15);
\draw [thick, green,->] (1.5,2.4) to [out=-90,in=0] (0,-.3) to [out=180,in=-90] (-1.5,2.4);
% %labels
\end{tikzpicture}
};
\endxy
}
\cal{E}_2 \cal{E}_1\cal{F}_1\cal{F}_2 \onell{[1,2,0]}.
\]
Using the following relations (where $l > 0$ and which also hold with the colors switched):
\begin{align*}
\xy
(0,0)*{
\begin{tikzpicture}[scale=.75]
\draw [thick,red,->] (-1.25,0) arc (180:-180:.5);
\draw [thick,green,->] (0,-.75) -- (0,.75);
\node at (-.32,-.25) {$\bullet$};
\node at (-.15,-.5) {\tiny $k$};
\end{tikzpicture}
};
\endxy
\; = \;
\xy
(0,0)*{
\begin{tikzpicture}[scale=.75]
\draw [thick,red,->] (.25,0) arc (180:-180:.5);
\draw [thick,green,->] (0,-.75) -- (0,.75);
\node at (1.18,-.25) {$\bullet$};
\node at (1.5,-.5) {\tiny $k+1$};
\end{tikzpicture}
};
\endxy
\; - \;
\xy
(0,0)*{
\begin{tikzpicture}[scale=.75]
\draw [thick,red,->] (.25,0) arc (180:-180:.5);
\draw [thick,green,->] (0,-.75) -- (0,.75);
\node at (1.18,-.25) {$\bullet$};
\node at (1.3,-.5) {\tiny $k$};
\node at (0,0) {$\bullet$};
\end{tikzpicture}
};
\endxy
\quad&,\quad
\xy
(0,0)*{
\begin{tikzpicture}[scale=.75]
\draw [thick,red,->] (.25,0) arc (180:-180:.5);
\draw [thick,green,->] (0,-.75) -- (0,.75);
\node at (1.18,-.25) {$\bullet$};
\node at (1.3,-.5) {\tiny $l$};
%\node at (0,0) {$\bullet$};
\end{tikzpicture}
};
\endxy
\; = \;
\sum_{r+s=l-1}
\xy
(0,0)*{
\begin{tikzpicture}[scale=.75]
\draw [thick,red,->] (-1.25,0) arc (180:-180:.5);
\draw [thick,green,->] (0,-.75) -- (0,.75);
\node at (-.32,-.25) {$\bullet$};
\node at (-.15,-.5) {\tiny $s$};
\node at (0,0) {$\bullet$};
\node at (.25,0) {\tiny $r$};
\end{tikzpicture}
};
\endxy
%\label{BubbleSlide1}
\\
\xy
(0,0)*{
\begin{tikzpicture}[scale=.75]
\draw [thick,green,<-] (.25,0) arc (180:-180:.5);
\draw [thick,red,->] (0,-.75) -- (0,.75);
\node at (1.18,-.25) {$\bullet$};
\node at (1.3,-.5) {\tiny $k$};
%\node at (0,0) {$\bullet$};
\end{tikzpicture}
};
\endxy
\; = \;
\xy
(0,0)*{
\begin{tikzpicture}[scale=.75]
\draw [thick,green,<-] (-1.25,0) arc (180:-180:.5);
\draw [thick,red,->] (0,-.75) -- (0,.75);
\node at (-.32,-.25) {$\bullet$};
\node at (-.36,-.6) {\tiny $k$+$1$};
\node at (-.38,.7) {};
%\node at (0,0) {$\bullet$};
%\node at (.25,0) {\tiny $r$};
\end{tikzpicture}
};
\endxy
\; - \;
\xy
(0,0)*{
\begin{tikzpicture}[scale=.75]
\draw [thick,green,<-] (-1.25,0) arc (180:-180:.5);
\draw [thick,red,->] (0,-.75) -- (0,.75);
\node at (-.32,-.25) {$\bullet$};
\node at (-.3,-.5) {\tiny $k$};
\node at (0,0) {$\bullet$};
%\node at (.25,0) {\tiny $r$};
\end{tikzpicture}
};
\endxy
\quad&,\quad
\xy
(0,0)*{
\begin{tikzpicture}[scale=.75]
\draw [thick,green,<-] (-1.25,0) arc (180:-180:.5);
\draw [thick,red,->] (0,-.75) -- (0,.75);
\node at (-.32,-.25) {$\bullet$};
\node at (-.15,-.5) {\tiny $l$};
\end{tikzpicture}
};
\endxy
\; = \;
\sum_{r+s=l-1}
\xy
(0,0)*{
\begin{tikzpicture}[scale=.75]
\draw [thick,green,<-] (.25,0) arc (180:-180:.5);
\draw [thick,red,->] (0,-.75) -- (0,.75);
\node at (1.18,-.25) {$\bullet$};
\node at (1.3,-.5) {\tiny $s$};
\node at (0,0) {$\bullet$};
\node at (-.25,0) {\tiny $r$};
\end{tikzpicture}
};
\endxy
%\label{BubbleSlide2}
\end{align*}
we lift the complex to the following one in the vertical trace:
\[
q^{-1} \onell{[1,2,0]} \oplus q \onell{[1,2,0]}
\xrightarrow{
\begin{pmatrix}
\xy
(0,0)*{
\begin{tikzpicture}[scale=.3]
%bottom
\draw[gray] (-2,-2) arc (-180:0:2 and 1);
\draw[gray, dashed] (-2,-2) arc (180:0:2 and 1);
%top
\draw[gray] (0,3) circle (2 and 1);
%middle
\draw[gray] (2,3) to (2,-2);
\draw[gray] (-2,3) to (-2,-2);
\draw [thick, green,->] (1.5,0) arc (0:-360:1.5);
\draw [thick, red,->] (.75,0) arc (0:-360:.75);
\node at (-1.5,0) {$\bullet$};
% %labels
% \node at (4.5,0) {\tiny$\lambda$};
% \node at (-3.5,0) {\tiny$\lambda+4$};
\end{tikzpicture}
};
\endxy
+
\xy
(0,0)*{
\begin{tikzpicture}[scale=.3]
%bottom
\draw[gray] (-2,-2) arc (-180:0:2 and 1);
\draw[gray, dashed] (-2,-2) arc (180:0:2 and 1);
%top
\draw[gray] (0,3) circle (2 and 1);
%middle
\draw[gray] (2,3) to (2,-2);
\draw[gray] (-2,3) to (-2,-2);
\draw [thick, green,->] (.75,.5) arc (0:-360:.75);
\node at (0,-.25) {$\bullet$};
\node at (.5,-.55) {\tiny $2$};
% %labels
% \node at (4.5,0) {\tiny$\lambda$};
% \node at (-3.5,0) {\tiny$\lambda+4$};
\end{tikzpicture}
};
\endxy
& , \quad 
\id
\end{pmatrix}
}
q \onell{[1,2,0]}.
\]
Under skew Howe duality, this is sent to:
\[
q^{-1}V_1 \oplus qV_1
\xrightarrow{\left(
0, \id
\right)}
qV_1
\]
as desired. 
Similarly,
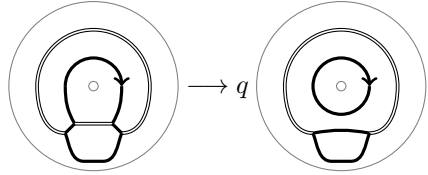
\[
\xy
(0,0)*{
\begin{tikzpicture}[scale=.5]
\draw[gray] (0,0) circle (2.25);
\draw[gray] (0,0) circle (.125);
\draw[very thick] (.75,0) to [out=270,in=60] (.5,-1);
\draw[very thick] (-.5,-1) to [out=120,in=270] (-.75,0);
\draw[very thick, ->] (-.75,0) arc (180:0:.75);
\draw[double] (-1.5,0) arc (180:0:1.5);
\draw[double] (-.75,-1.25) to [out=180,in=270] (-1.5,0);
\draw[double] (1.5,0) to [out=270,in=0] (.75,-1.25);
\draw[double] (.5,-1) to (-.5,-1);
\draw[very thick] (.75,-1.25) to (.5,-1);
\draw[very thick] (-.5,-1) to (-.75,-1.25);
\draw[very thick] (.75,-1.25) to [out=250,in=0] (.25,-2) to (-.25,-2) to [out=180,in=290] (-.75,-1.25);
\end{tikzpicture}
};
\endxy
\longrightarrow
q\;
\xy
(0,0)*{
\begin{tikzpicture}[scale=.5]
\draw[gray] (0,0) circle (2.25);
\draw[gray] (0,0) circle (.125);
\draw[very thick, ->] (.75,0) arc (360:0:.75);
\draw[double] (-1.5,0) arc (180:0:1.5);
\draw[double] (-.75,-1.25) to [out=180,in=270] (-1.5,0);
\draw[double] (1.5,0) to [out=270,in=0] (.75,-1.25);
\draw[very thick] (.75,-1.25) to [out=170,in=10] (-.75,-1.25);
\draw[very thick] (.75,-1.25) to [out=250,in=0] (.25,-2) to (-.25,-2) to [out=180,in=290] (-.75,-1.25);
\end{tikzpicture}
};
\endxy
\]
lifts to the complex 
\[
\onell{[1,2,0]}
\xrightarrow{
\begin{pmatrix}
\id 
&, \quad
\xy
(0,0)*{
\begin{tikzpicture}[scale=.3]
%bottom
\draw[gray] (-2,-2) arc (-180:0:2 and 1);
\draw[gray, dashed] (-2,-2) arc (180:0:2 and 1);
%top
\draw[gray] (0,3) circle (2 and 1);
%middle
\draw[gray] (2,3) to (2,-2);
\draw[gray] (-2,3) to (-2,-2);
\draw [thick, green,->] (1.5,0) arc (0:-360:1.5);
\draw [thick, red,->] (.75,0) arc (0:-360:.75);
\node at (-1.5,0) {$\bullet$};
% %labels
% \node at (4.5,0) {\tiny$\lambda$};
% \node at (-3.5,0) {\tiny$\lambda+4$};
\end{tikzpicture}
};
\endxy
\end{pmatrix}
}
\onell{[1,2,0]} \oplus q^2 \onell{[1,2,0]}
\]
in $\hTr(\cal{U}_Q(\glnn{3}^{0 \leq 2}))$ which maps to:
\[
V_1
\xrightarrow{\left(
\id,
0
\right)}
V_1 \bigoplus q^2 V_1.
\]
The other arrangement, where the trivial circle merging lies inside the essential circle, yields a similar result.

Finally, we consider the inessential, ``non-annular'' situation. The complex
\[
\xy
(0,0)*{
\begin{tikzpicture}[scale=.5]
\draw[gray] (0,0) circle (2.25);
\draw[gray] (0,0) circle (.125);
\draw[double] (-1.5,0) arc (180:0:1.5);
\draw[double] (-1,-1) to [out=180,in=270] (-1.5,0);
\draw[double] (1.5,0) to [out=270,in=0] (1,-1);
\draw[double] (.25,-1) to (-.25,-1);
\draw[very thick] (1,-1) to [out=170,in=10] (.25,-1);
\draw[very thick] (-.25,-1) to [out=170,in=10] (-1,-1);
\draw[very thick] (1,-1) to [out=250,in=0] (.75,-1.75) to (.5,-1.75) to [out=180,in=290] (.25,-1);
\draw[very thick] (-.25,-1) to [out=250,in=0] (-.5,-1.75) to (-.75,-1.75) to [out=180,in=290] (-1,-1);
\end{tikzpicture}
};
\endxy
\longrightarrow
q \;
\xy
(0,0)*{
\begin{tikzpicture}[scale=.5]
\draw[gray] (0,0) circle (2.25);
\draw[gray] (0,0) circle (.125);
\draw[double] (-1.5,0) arc (180:0:1.5);
\draw[double] (-1,-1) to [out=180,in=270] (-1.5,0);
\draw[double] (1.5,0) to [out=270,in=0] (1,-1);
\draw[very thick] (1,-1) to [out=170,in=10] (-1,-1);
\draw[very thick] (1,-1) to [out=250,in=0] (.5,-1.75) to (-.5,-1.75) to [out=180,in=290] (-1,-1);
\end{tikzpicture}
};
\endxy
\]
lifts to:
\[
\cal{E}\cal{F}\cal{E}\cal{F}\onell{[2,0]}
\xrightarrow{
\xy
(0,0)*{
\begin{tikzpicture}[scale=.3]
%bottom
\draw[gray] (-2,-3) arc (-180:0:2 and 1);
\draw[gray, dashed] (-2,-3) arc (180:0:2 and 1);
%top
\draw[gray] (0,2) circle (2 and 1);
%middle
\draw[gray] (2,2) to (2,-3);
\draw[gray] (-2,2) to (-2,-3);
\draw [thick,->] (-1.5,-3.7) to [out=90,in=-90] (-1,1.15);
\draw [thick,<-] (1.5,-3.7) to [out=90,in=-90] (1,1.15);
\draw [thick,->] (1,-3.9) to [out=90,in=0] (0,-1) to [out=180,in=90] (-1,-3.9);
% %labels
\end{tikzpicture}
};
\endxy}
\quad
q \cal{E} \cal{F} \onell{[2,0]}
\]
in $\hTr(\U_Q(\glnn{2}))$. 
This is isomorphic to the complex
\[
q^{-2}[2,0] \oplus [2,0] \oplus [2,0] \oplus q^2 [2,0]
\xrightarrow{
\begin{pmatrix}
3\;
\xy
(0,0)*{
\begin{tikzpicture}[scale=.75]
%% String diagram
\draw [thick,<-] (.75,0) arc (-180:180:.25);
\node at (1.25,0) {$\bullet$};
\node at (1.45,-.15) {\tiny $_2$};
\end{tikzpicture}
};
\endxy
& \id
& \id
& 0
\\
& & & 
\\
\xy
(0,0)*{
\begin{tikzpicture}[scale=.75]
%% String diagram
\draw [thick,<-] (.75,0) arc (-180:180:.25);
\node at (1.25,0) {$\bullet$};
\node at (1.45,-.15) {\tiny $_3$};
\end{tikzpicture}
};
\endxy
+
3 \;
\xy
(0,2.5)*{
\begin{tikzpicture}[scale=.75]
%% String diagram
\draw [thick,<-] (.75,0) arc (-180:180:.25);
\node at (1.25,0) {$\bullet$};
\node at (1.45,-.15) {\tiny $_2$};
\end{tikzpicture}
};
(0,-2.5)*{
\begin{tikzpicture}[scale=.75]
%% String diagram
\draw [thick,<-] (.75,0) arc (-180:180:.25);
\node at (1.25,0) {$\bullet$};
\node at (1.45,-.15) {\tiny $_2$};
\end{tikzpicture}
};
\endxy
& 
2 \;
\xy
(0,0)*{
\begin{tikzpicture}[scale=.75]
%% String diagram
\draw [thick,<-] (.75,0) arc (-180:180:.25);
\node at (1.25,0) {$\bullet$};
\node at (1.45,-.15) {\tiny $_2$};
\end{tikzpicture}
};
\endxy
&
2\;
\xy
(0,0)*{
\begin{tikzpicture}[scale=.75]
%% String diagram
\draw [thick,<-] (.75,0) arc (-180:180:.25);
\node at (1.25,0) {$\bullet$};
\node at (1.45,-.15) {\tiny $_2$};
\end{tikzpicture}
};
\endxy
& \id
\end{pmatrix}
}
[2,0] \oplus q^2 [2,0]
\]
in $\grvTr(\cal{U}_Q(\glnn{2})^{0 \leq 2})$, where the (linear combinations of) bubbles are understood as the corresponding elements in the vertical trace.
This gives under skew Howe duality:
\[
q^{-2}V_0 \oplus V_0 \oplus V_0 \oplus q^2 V_0 \xrightarrow{
\begin{pmatrix}
0 & \id & \id & 0 \\
0 & 0 & 0 & \id
\end{pmatrix}
}
V_0\bigoplus q^2V_0
\]
which agrees with the ``non-annular Khovanov differential'':
\[
\begin{Bmatrix}
1\otimes 1 \mapsto 1 \\
1 \otimes x \mapsto x \\
x\otimes 1 \mapsto x \\
x\otimes x \mapsto 0
\end{Bmatrix}.
\]
A similar computation confirms the case when an enhanced inessential circle splits into two.
\end{proof}

The following result, originally due to Grigsby-Licata-Wehrli \cite{GLW}, now follows easily from our construction of $\saKh$.

\begin{cor}
For any annular link $\cal{L}$, $\saKh(\cal{L})$ carries an action of $\slnn{2}$.
\end{cor}
\begin{proof}
This follows immediately since the differential in $\SH(\tilde{C}(\cal{L}))$ consists of intertwiners between $\slnn{2}$-modules.
\end{proof}

\section{Sutured annular Khovanov-Rozansky homology}\label{saKhR}%##########################################################################

We now show that the specialization to $\slnn{2}$ link homology in Section \ref{sH} only served to simplify one technical result, 
and hence generalize this to construct our annular $\sln$ link homology theory.

Recall that Khovanov extended his categorification of the Jones polynomial to a link homology theory categorifying the $\slnn{3}$ Reshetikhin-Turaev link invariant
using the theory of $\slnn{3}$ foams \cite{Kh5}.  
Khovanov and Rozansky subsequently addressed the case of general $\sln$ Reshetikhin-Turaev invariants, 
using the homotopy category of matrix factorizations to construct $\sln$ 
link homology \cite{KhR}.
This $\sln$ homology theory was then interpreted using foams by Mackaay, Sto{\v{s}}i{\'c} and Vaz \cite{MSV}; 
however, that construction was not completely elementary/combinatorial in nature.
Recently, the authors introduced the 2-category $\Foam{n}{}$ of enhanced $\sln$ foams  to give an entirely elementary foam-based description 
of (colored) Khovanov-Rozansky homology using categorified skew Howe duality \cite{QR}.

Our foam-based presentation of Khovanov-Rozansky $\sln$ link homology was motivated by a desired relation between $\Foam{n}{}$ 
and the categorified quantum group $\Ucat_Q(\glm)$. 
The crucial point is that equation \eqref{LQRthm} remains valid for general $n$, and we obtain a $2$-functor:
\begin{equation}\label{QRthm}
\Ucat_Q(\glm)\xrightarrow{\Phi_n} n\cat{Foam}.
\end{equation}
This functor again factors through the $n$-bounded quotient $\Ucat_Q(\glm)^{0\leq n}$.

We can now emulate our construction of the $\slnn{2}$ annular link invariant, with $\AFoam$ replaced by $\hTr(\Foam{n}{})$ to construct an annular $\sln$ link homology, 
which we call \emph{sutured annular Khovanov-Rozansky homology}. 
We'll assume some familiarity with the 2-category $\Foam{n}{}$ for the duration; see \cite{QR} for full details.

Recall that the above construction of $\saKh$ consisted of the $n=2$ specialization of the following steps.
\begin{enumerate}
\item \label{step1} Assign to any framed annular link a complex in $\hTr(\Foam{n}{})$ which is invariant up to homotopy under Reidemeister moves. 
\item  \label{step2} Lift this complex to an isomorphic complex in $\grvTr(\Foam{n}{})$.
\item \label{step3} Apply a functor $\grvTr(\Foam{n}{}) \to \grRep(\sln)$ to obtain a complex of formally $q$-graded $\sln$-modules, then take homology. 
\end{enumerate}
For general $n$, the first and third steps follow exactly as they did in the $\slnn{2}$ case. 
Indeed, for step \eqref{step1} we can use the $\sln$ skein relations 
\begin{equation}\label{eq:slnSkein}
\left \llbracket
\xy
(0,0)*{
\begin{tikzpicture} [scale=.5]
\draw[very thick, directed=.99] (1,0) to (0,1);
\draw[very thick] (.6,.6) to (1,1);
\draw[very thick,directed=.99] (.4,.4) to (0,0);
\end{tikzpicture}
}
\endxy
\right \rrbracket_n
=
\left(
q^{-1} \;
\xy
(0,0)*{
\begin{tikzpicture} [scale=.7, rotate=90]
\draw[very thick, directed=.99] (0,0) to [out=45,in=315] (0,1);
\draw[very thick, directed=.99] (1,0) to [out=135,in=225] (1,1);
\end{tikzpicture}
}
\endxy
\xrightarrow{
\xy
(0,0)*{
\begin{tikzpicture} [scale=.2,fill opacity=0.2]
	%shading
	\path [fill=red] (4.25,2) to (4.25,-.5) to [out=170,in=10] (-.5,-.5) to (-.5,2) to
		[out=0,in=225] (.75,2.5) to [out=270,in=180] (1.625,1.25) to [out=0,in=270] 
			(2.5,2.5) to [out=315,in=180] (4.25,2);
	\path [fill=red] (3.75,3) to (3.75,.5) to [out=190,in=350] (-1,.5) to (-1,3) to [out=0,in=135]
		(.75,2.5) to [out=270,in=180] (1.625,1.25) to [out=0,in=270] 
			(2.5,2.5) to [out=45,in=180] (3.75,3);
	\path[fill=blue] (2.5,2.5) to [out=270,in=0] (1.625,1.25) to [out=180,in=270] (.75,2.5);
	%bottom web
	\draw [very thick] (4.25,-.5) to [out=170,in=10] (-.5,-.5);
	\draw [very thick] (3.75,.5) to [out=190,in=350] (-1,.5);
	%seam
	\draw [very thick, red] (2.5,2.5) to [out=270,in=0] (1.625,1.25) to [out=180,in=270] (.75,2.5);
	%vertical edges
	\draw [very thick] (3.75,3) to (3.75,.5);
	\draw [very thick] (4.25,2) to (4.25,-.5);
	\draw [very thick] (-1,3) to (-1,.5);
	\draw [very thick] (-.5,2) to (-.5,-.5);
	%top web
	\draw [very thick] (2.5,2.5) to (.75,2.5);
	\draw [very thick] (.75,2.5) to [out=135,in=0] (-1,3);
	\draw [very thick] (.75,2.5) to [out=225,in=0] (-.5,2);
	\draw [very thick] (3.75,3) to [out=180,in=45] (2.5,2.5);
	\draw [very thick] (4.25,2) to [out=180,in=315] (2.5,2.5);		
\end{tikzpicture}
};
\endxy
}
\;
\uwave{
\xy
(0,0)*{
\begin{tikzpicture} [scale=.7, rotate=90]
\draw[very thick, directed=.85] (.5,.75) to (1,1);
\draw[very thick, directed=.85] (.5,.75) to (0,1);
\draw[very thick, directed=.65] (.5,.25) to (.5,.75);
\draw[very thick, directed=.55] (1,0) to (.5,.25);
\draw[very thick, directed=.55] (0,0) to (.5,.25);
\node at (.75,.5) {\tiny $2$};
\end{tikzpicture}
}
\endxy} \;
\right) 
\quad , \quad
\left \llbracket
\xy
(0,0)*{
\begin{tikzpicture} [scale=.5]
\draw[very thick, directed=.99] (1,1) to (0,0);
\draw[very thick, directed=.99] (.4,.6) to (0,1);
\draw[very thick] (1,0) to (.6,.4);
\end{tikzpicture}
}
\endxy
\right \rrbracket_n
=
\left(
\;
\uwave{
\xy
(0,0)*{
\begin{tikzpicture} [scale=.7, rotate=90]
\draw[very thick, directed=.85] (.5,.75) to (1,1);
\draw[very thick, directed=.85] (.5,.75) to (0,1);
\draw[very thick, directed=.65] (.5,.25) to (.5,.75);
\draw[very thick, directed=.55] (1,0) to (.5,.25);
\draw[very thick, directed=.55] (0,0) to (.5,.25);
\node at (.75,.5) {\tiny $2$};
\end{tikzpicture}
}
\endxy}
\xrightarrow{
\xy
(0,0)*{
\begin{tikzpicture} [scale=.2,fill opacity=0.2]
	%shading
	\path [fill=red] (4.25,-.5) to (4.25,2) to [out=170,in=10] (-.5,2) to (-.5,-.5) to 
		[out=0,in=225] (.75,0) to [out=90,in=180] (1.625,1.25) to [out=0,in=90] 
			(2.5,0) to [out=315,in=180] (4.25,-.5);
	\path [fill=red] (3.75,.5) to (3.75,3) to [out=190,in=350] (-1,3) to (-1,.5) to 
		[out=0,in=135] (.75,0) to [out=90,in=180] (1.625,1.25) to [out=0,in=90] 
			(2.5,0) to [out=45,in=180] (3.75,.5);
	\path[fill=blue] (.75,0) to [out=90,in=180] (1.625,1.25) to [out=0,in=90] (2.5,0);
	%bottom web
	\draw [very thick] (2.5,0) to (.75,0);
	\draw [very thick] (.75,0) to [out=135,in=0] (-1,.5);
	\draw [very thick] (.75,0) to [out=225,in=0] (-.5,-.5);
	\draw [very thick] (3.75,.5) to [out=180,in=45] (2.5,0);
	\draw [very thick] (4.25,-.5) to [out=180,in=315] (2.5,0);
	%seam
	\draw [very thick, red] (.75,0) to [out=90,in=180] (1.625,1.25) to [out=0,in=90] (2.5,0);
	%vertical edges
	\draw [very thick] (3.75,3) to (3.75,.5);
	\draw [very thick] (4.25,2) to (4.25,-.5);
	\draw [very thick] (-1,3) to (-1,.5);
	\draw [very thick] (-.5,2) to (-.5,-.5);
	%top web
	\draw [very thick] (4.25,2) to [out=170,in=10] (-.5,2);
	\draw [very thick] (3.75,3) to [out=190,in=350] (-1,3);	
\end{tikzpicture}
};
\endxy
}
q \;
\xy
(0,0)*{
\begin{tikzpicture} [scale=.7, rotate=90]
\draw[very thick, directed=.99] (0,0) to [out=45,in=315] (0,1);
\draw[very thick, directed=.99] (1,0) to [out=135,in=225] (1,1);
\end{tikzpicture}
}
\endxy \;
\right)
\end{equation}
and the $a=1$ case of the formulae 
\begin{equation}\label{eq:CapsAndCups}
\begin{aligned}
\left \llbracket
\xy
(0,0)*{\begin{tikzpicture} [scale=.5]
\draw[very thick, directed=.99] (0,1) to [out=180,in=90] (-1.25,.5) to [out=270,in=180] (0,0);
\node at (-1.5,1) {\small$a$};
\end{tikzpicture}};
\endxy 
\;\; \right \rrbracket_n
\quad = \quad
\xy
(0,0)*{\begin{tikzpicture} [scale=.5]
 \draw [double] (-1.25,.5) to (-2.25,.5);
 \draw [very thick, directed=.55] (0,0) to [out=180,in=300] (-1.25,.5);
\draw [very thick, directed=.55] (0,1) to [out=180,in=60] (-1.25,.5);
 \node at (.875,0) {\small$n{-}a$};
 \node at (.5,1) {\small$a$};
 \node at (-2.75,.5) {\small$n$}; 
\end{tikzpicture}};
\endxy
\quad &, \quad
\left \llbracket \;\;
\xy
(0,0)*{\begin{tikzpicture} [scale=.5]
\draw[very thick, directed=.99] (0,0) to [out=0,in=270] (1.25,.5) to [out=90,in=0] (0,1);
\node at (1.5,1) {\small$a$};
\end{tikzpicture}};
\endxy
\right \rrbracket_n
\quad = \quad
\xy
(0,0)*{\begin{tikzpicture} [scale=.5]
 \draw [double] (2.25,.5) -- (1.25,.5);
 \draw [very thick, directed=.55] (1.25,.5) to [out=240,in=0] (0,0);
\draw [very thick, directed=.55] (1.25,.5) to [out=120,in=0] (0,1);
 \node at (-.875,0) {\small$n{-}a$};
 \node at (-.5,1) {\small$a$};
 \node at (2.75,.5) {\small$n$}; 
\end{tikzpicture}};
\endxy \\
\left \llbracket
\xy
(0,0)*{\begin{tikzpicture} [scale=.5]
\draw[very thick, rdirected=.05] (0,1) to [out=180,in=90] (-1.25,.5) to [out=270,in=180] (0,0);
\node at (-1.5,1) {\small$a$};
\end{tikzpicture}};
\endxy
\;\; \right \rrbracket_n
\quad = \quad
\xy
(0,0)*{\begin{tikzpicture} [scale=.5]
 \draw [double] (-1.25,.5) to (-2.25,.5);
 \draw [very thick, directed=.55] (0,0) to [out=180,in=300] (-1.25,.5);
\draw [very thick, directed=.55] (0,1) to [out=180,in=60] (-1.25,.5);
 \node at (.875,1) {\small$n{-}a$};
 \node at (.5,0) {\small$a$};
 \node at (-2.75,.5) {\small$n$}; 
\end{tikzpicture}};
\endxy
\quad &, \quad
\left \llbracket \;\;
\xy
(0,0)*{\begin{tikzpicture} [scale=.5]
\draw[very thick, rdirected=.05] (0,0) to [out=0,in=270] (1.25,.5) to [out=90,in=0] (0,1);
\node at (1.5,1) {\small$a$};
\end{tikzpicture}};
\endxy
\right \rrbracket_n
\quad = \quad
\xy
(0,0)*{\begin{tikzpicture} [scale=.5]
 \draw [double] (2.25,.5) -- (1.25,.5);
 \draw [very thick, directed=.55] (1.25,.5) to [out=240,in=0] (0,0);
\draw [very thick, directed=.55] (1.25,.5) to [out=120,in=0] (0,1);
 \node at (-.875,1) {\small$n{-}a$};
 \node at (-.5,0) {\small$a$};
 \node at (2.75,.5) {\small$n$}; 
\end{tikzpicture}};
\endxy 
\end{aligned}
\end{equation}
to assign a complex $C_n(\cal{L})$ in $\hTr(\Foam{n}{})$ to any framed annular link 
(recall that edges of webs in $\Foam{n}{}$ carry labels in the set $\{0,1,\ldots,n\}$).
More generally, we can use Rickard complexes \cite{CR}, 
in particular the shifted Rickard complexes\footnote{These complexes live in $\Ucatc_Q(\glm)$, 
so this formally requires using trace decategorifications of the extended 2-functor $\Ucatc_Q(\glm)\xrightarrow{\Phi_n} n\cat{Foam}$. 
As we mentioned in Remark \ref{rem:Ucheck}, $\grvTr(\Ucatc_Q(\glm)) \cong \grvTr(\Ucat_Q(\glm))$, so the consideration of $\Ucatc_Q(\glm)$
is only technical in nature.} 
from \cite[equations (4.1) and (4.2)]{QR}, 
and the general form of equation \eqref{eq:CapsAndCups} to assign a complex $C_n(\cal{L})$ in $\hTr(\Foam{n}{})$ 
to any framed annular link whose components are colored by fundamental\footnote{Even more generally, we can use Cautis-Rozansky 
categorified projectors (see \cite{Roz,Rose,Cautis}) to assign a complex to an annular link with components colored by arbitrary irreducible 
representations of $\sln$.} representations of $\sln$, 
\eg for $b\geq a$
\[
\left \llbracket
\xy
(0,0)*{
\begin{tikzpicture} [scale=.5]
\draw[very thick, directed=.99] (1,0) to (0,1);
\draw[very thick] (.6,.6) to (1,1);
\draw[very thick,directed=.99] (.4,.4) to (0,0);
\node at (1.25,1) {\tiny$a$};
\node at (1.25,0) {\tiny$b$};
%\node at (-.25,1) {\tiny$b$};
%\node at (-.25,0) {\tiny$a$};
\end{tikzpicture}
}
\endxy
\right \rrbracket_n
=
\left(
q^{-a} \;
\xy
(0,0)*{
\begin{tikzpicture}[scale=.4]
\draw[very thick, directed=.65] (1.5,0) to (.5,0);
\draw[very thick, directed=.65] (.5,0) to [out=240,in=0] (-1.5,0);
\draw[very thick, directed=.65] (.5,0) to [out=120,in=300] (-.5,1);
\draw[very thick, directed=.65] (1.5,1) to [out=180,in=60] (-.5,1);
\draw[very thick, directed=.65] (-.5,1) to (-1.5,1);
\node at (1.75,1) {\tiny$a$};
\node at (1.75,0) {\tiny$b$};
\node at (-1.75,1) {\tiny$b$};
\node at (-1.75,0) {\tiny$a$};
\end{tikzpicture}
};
\endxy
\longrightarrow
q^{-a+1} \;
\xy
(0,0)*{
\begin{tikzpicture}[scale=.4]
\draw[very thick, directed=.65] (2,1) to (1.25,1);
\draw[very thick, directed=.65] (2,0) to [out=180,in=300] (.5,0);
\draw[very thick, directed=.65] (1.25,1) to [out=240,in=60] (.5,0);
\draw[very thick, directed=.65] (.5,0) to (-.5,0);
\draw[very thick, directed=.65] (1.25,1) to [out=120,in=60] (-1.25,1);
\draw[very thick, directed=.65] (-.5,0) to [out=120,in=300] (-1.25,1);
\draw[very thick, directed=.65] (-.5,0) to [out=240,in=0] (-2,0);
\draw[very thick, directed=.65] (-1.25,1) to (-2,1);
\node at (1.25,.5) {\tiny$1$};
\node at (2.25,1) {\tiny$a$};
\node at (2.25,0) {\tiny$b$};
\node at (-2.25,1) {\tiny$b$};
\node at (-2.25,0) {\tiny$a$};
\end{tikzpicture}
};
\endxy
\longrightarrow \cdots \longrightarrow
\uwave{
\xy
(0,0)*{
\begin{tikzpicture}[scale=.4]
\draw[very thick, directed=.65] (2,1) to (1.25,1);
\draw[very thick, directed=.65] (2,0) to [out=180,in=300] (.5,0);
\draw[very thick, directed=.65] (1.25,1) to [out=180,in=60] (.5,0);
\draw[very thick, directed=.65] (.5,0) to (-.5,0);
%\draw[very thick, directed=.65] (1.25,1) to [out=120,in=60] (-1.25,1);
\draw[very thick, directed=.65] (-.5,0) to [out=120,in=0] (-1.25,1);
\draw[very thick, directed=.65] (-.5,0) to [out=240,in=0] (-2,0);
\draw[very thick, directed=.65] (-1.25,1) to (-2,1);
\node at (2.25,1) {\tiny$a$};
\node at (2.25,0) {\tiny$b$};
\node at (-2.25,1) {\tiny$b$};
\node at (-2.25,0) {\tiny$a$};
\end{tikzpicture}
};
\endxy
}
\right).
\]
As in the $\slnn{2}$ case, the complex $C_n(\cal{L})$ is invariant under Reidemeister moves, 
up to homotopy equivalence (and up to tensoring with essential $n$-labeled circles, which won't affect the resulting homology).
To see that step (3) carries over to the $\sln$ case, simply repeat the arguments in Lemma \ref{lem:sl2Rep} with $2$ replaced by $n$.

Step (2), however, is more difficult in the $\sln$ case. 
Recall, the crucial fact was the equivalence of categories between $\grvTr(\Foam{2}{})$ and $\hTr(\Foam{2}{})$, 
and we are currently unable to extend this result to the $\sln$ setting.
Nevertheless, we can show that it holds ``up to homotopy,'' \ie we have the following,
which suffices to lift the complex $C_n(\cal{L})$ to a complex $\tilde{C}_n(\cal{L})$ in $\grvTr(\Foam{n}{})$. 

\begin{prop} \label{prop:equivHomCat}
There is an equivalence of categories $\Kom(\hTr(\Foam{n}{})) \cong \Kom(\grvTr(\Foam{n}{}))$, 
where $\Kom(-)$ denotes the homotopy category of bounded complexes over an additive category.
\end{prop}

In order to prove this, we first detail the decategorified version of this result. 
The setting for this is the $\Z[q,q^{-1}]$-module of Cautis, Kamnitzer, and Morrison's $\sln$ webs \cite{CKM} embedded in the annulus. 
We believe this result is of independent interest; see \eg \cite{QS} and \cite{RTub} where this result is used in the study of the 
HOMFLY-PT polynomial via a doubled Schur algebra and a symmetric web formulation of the colored Jones polynomial (respectively).

\begin{lem}[Annular web evaluation algorithm] \label{lem:annEvalAlgo}
The annular closure of any $\sln$ web is equal to a $\Z[q,q^{-1}]$-linear combination of nested, labeled essential circles.
\end{lem}

Note that in order to consider the annular closure of a web, the labelings on the top and bottom web endpoints must agree.

\begin{proof}
By \cite[Theorem 5.3.1]{CKM}, it suffices to consider ``ladder webs,'' and our argument proceeds by increasing the labelings on the 
leftmost ladder upright, until this upright is freed from the remainder of the web closure. The result then follows inductively.

To begin, we can use the following web relations to combine as many consecutive ladder rungs pointing in the 
same direction as possible:
\begin{equation}\label{easywebmoves}
\xy
(0,0)*{
\begin{tikzpicture}[scale=.55]
%% Uprights
\draw [very thick, ->] (0,0) -- (0,3);
\draw [very thick, ->] (1,0) -- (1,3);
%% Ladders
\draw [very thick, directed=.55] (0,.75) -- (1,1.25);
\node at (.5,.625) {\tiny $a$};
\draw [very thick, directed=.55] (0,1.75) -- (1,2.25);
\node at (.5,2.375) {\tiny $b$};
\end{tikzpicture}
};
\endxy
=
{a+b \brack a}
\;\;
\xy
(0,0)*{
\begin{tikzpicture}[scale=.55]
%% Uprights
\draw [very thick, ->] (0,0) -- (0,3);
\draw [very thick, ->] (1,0) -- (1,3);
%% Ladders
\draw [very thick, directed=.55] (0,1.25) -- (1,1.75);
\node at (.5,1.125) {\tiny $a{+}b$};
\end{tikzpicture}
};
\endxy
\quad,\quad
\xy
(0,0)*{
\begin{tikzpicture}[scale=.55]
%% Uprights
\draw [very thick, ->] (0,0) -- (0,3);
\draw [very thick, ->] (1,0) -- (1,3);
%% Ladders
\draw [very thick, directed=.55] (1,.75) -- (0,1.25);
\node at (.5,.62) {\tiny $a$};
\draw [very thick, directed=.55] (1,1.75) -- (0,2.25);
\node at (.5,2.375) {\tiny $b$};
\end{tikzpicture}
};
\endxy
=
{a+b \brack a}
\;\;
\xy
(0,0)*{
\begin{tikzpicture}[scale=.55]
%% Uprights
\draw [very thick, ->] (0,0) -- (0,3);
\draw [very thick, ->] (1,0) -- (1,3);
%% Ladders
\draw [very thick, directed=.55] (1,1.25) -- (0,1.75);
\node at (.5,1.125) {\tiny $a{+}b$};
\end{tikzpicture}
};
\endxy \\
\quad , \quad
\xy
(0,0)*{
\begin{tikzpicture}[scale=.55]
%% Upright
\draw[very thick, ->] (0,0) -- (0,3);
\draw[very thick, ->] (1,0) -- (1,3);
\draw[very thick, ->] (2,0) -- (2,3);
%% Ladders
\draw [very thick, directed=.55] (0,.75) -- (1,1.25);
\draw [very thick, directed=.55] (2,1.75) -- (1,2.25);
\end{tikzpicture}
};
\endxy
=
\xy
(0,0)*{
\begin{tikzpicture}[scale=.55]
%% Upright
\draw[very thick, ->] (0,0) -- (0,3);
\draw[very thick, ->] (1,0) -- (1,3);
\draw[very thick, ->] (2,0) -- (2,3);
%% Ladders
\draw [very thick, directed=.55] (0,1.75) -- (1,2.25);
\draw [very thick, directed=.55] (2,.75) -- (1,1.25);
\end{tikzpicture}
};
\endxy
\quad,\quad
\xy
(0,0)*{
\begin{tikzpicture}[scale=.55]
%% Upright
\draw[very thick, ->] (0,0) -- (0,3);
\draw[very thick, ->] (1,0) -- (1,3);
\draw[very thick, ->] (2,0) -- (2,3);
%% Ladders
\draw [very thick, directed=.55] (1,.75) -- (2,1.25);
\draw [very thick, directed=.55] (1,1.75) -- (0,2.25);
\end{tikzpicture}
};
\endxy
=
\xy
(0,0)*{
\begin{tikzpicture}[scale=.55]
%% Upright
\draw[very thick, ->] (0,0) -- (0,3);
\draw[very thick, ->] (1,0) -- (1,3);
\draw[very thick, ->] (2,0) -- (2,3);
%% Ladders
\draw [very thick, directed=.55] (1,1.75) -- (2,2.25);
\draw [very thick, directed=.55] (1,.75) -- (0,1.25);
\end{tikzpicture}
};
\endxy
\end{equation}

Next, we consider the leftmost upright of the ladder web. If no rungs touch it, then it already yields a nested circle, and we are done (by inducting on the number of uprights).
If not, locate rungs between the leftmost and second-to-leftmost uprights which successively point to the right, then to the left, 
\ie arranged as in the web on the left-hand side of equation \eqref{eq:FEswitch_decat} below, 
but possibly with other rungs between the second-to- and third-to-leftmost upright separating them. 
Such a pair of rungs necessarily exists, since the web is an annular closure (this may require sliding rungs ``around the annulus'').

We now aim to use the following relation:
\begin{equation}\label{eq:FEswitch_decat}
\xy
(0,0)*{
\begin{tikzpicture}[scale=.3]
	\draw [very thick, directed=.55] (-2,-4) to (-2,-2);
	\draw [very thick, directed=1] (-2,-2) to (-2,0.25);
	\draw [very thick, directed=.55] (2,-4) to (2,-2);
	\draw [very thick, directed=1] (2,-2) to (2,0.25);
	\draw [very thick, directed=.55] (-2,-2.5) to (2,-1.5);
	\draw [very thick] (-2,0.25) to (-2,2);
	\draw [very thick, directed=.55] (-2,2) to (-2,4);
	\draw [very thick] (2,0.25) to (2,2);
	\draw [very thick, directed=.55] (2,2) to (2,4);
	\draw [very thick, rdirected=.55] (-2,2.5) to (2,1.5);
	\node at (-2,-4.5) {\tiny $k$};
	\node at (2,-4.5) {\tiny $l$};
	\node at (-2.25,4.5) {\tiny $k{-}j_1{+}j_2$};
	\node at (2.25,4.5) {\tiny $l{+}j_1{-}j_2$};
	\node at (-3.5,0) {\tiny $k{-}j_1$};
	\node at (3.5,0) {\tiny $l{+}j_1$};
	\node at (0,-1.25) {\tiny $j_1$};
	\node at (0,2.75) {\tiny $j_2$};
\end{tikzpicture}
};
\endxy=\sum_{j^{\prime} = 0}^{\min(j_1,j_2)} {k-j_1-l+j_2 \brack j^{\prime}}\xy
(0,0)*{
\begin{tikzpicture}[scale=.3]
	\draw [very thick, directed=.55] (-2,-4) to (-2,-2);
	\draw [very thick, directed=1] (-2,-2) to (-2,0.25);
	\draw [very thick, directed=.55] (2,-4) to (2,-2);
	\draw [very thick, directed=1] (2,-2) to (2,0.25);
	\draw [very thick, rdirected=.55] (-2,-1.5) to (2,-2.5);
	\draw [very thick] (-2,0.25) to (-2,2);
	\draw [very thick, directed=.55] (-2,2) to (-2,4);
	\draw [very thick] (2,0.25) to (2,2);
	\draw [very thick, directed=.55] (2,2) to (2,4);
	\draw [very thick, directed=.55] (-2,1.5) to (2,2.5);
	\node at (-2,-4.5) {\tiny $k$};
	\node at (2,-4.5) {\tiny $l$};
	\node at (-2.25,4.5) {\tiny $k{-}j_1{+}j_2$};
	\node at (2.25,4.5) {\tiny $l{+}j_1{-}j_2$};
	\node at (-4.25,0) {\tiny $k{+}j_2{-}j^{\prime}$};
	\node at (4.25,0) {\tiny $l{-}j_2{+}j^{\prime}$};
	\node at (0,-1.25) {\tiny $j_2{-}j^{\prime}$};
	\node at (0,2.75) {\tiny $j_1{-}j^{\prime}$};
\end{tikzpicture}
};
\endxy
\end{equation}
to express our web in terms of webs where either one label on the leftmost upright is larger than our web (and all others are the same), 
or where the number of rungs between the leftmost and second-to-leftmost uprights decreases.
Iterating this procedure, it eventually terminates, since there is a maximal possible labeling of an edge in our web (a crude bound is $n$ multiplied by the number of uprights), 
and expresses our web as a linear combination of webs where there are no rungs connected to the leftmost upright.

We now show that it is always possible to apply equation \eqref{eq:FEswitch_decat} to our chosen pair of rungs, 
after possibly expressing our web as a linear combination of webs 
using relations which affect neither the labels on the leftmost upright nor the number of rungs connected to the leftmost upright. 
Consider the set of rungs which are ``trapped'' between our chosen rungs, \ie those rungs which cannot be moved using the latter two ``slide'' relations in equation \eqref{easywebmoves} 
to the part of the web outside the region between our two rungs. 
For example, the red rungs in the following web are trapped, while the blue ones are not:
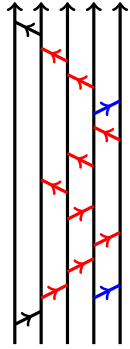
\[
\xy
(0,0)*{
\begin{tikzpicture}[scale=.35]
%% Uprights
\draw [very thick, ->] (0,0) -- (0,13);
\draw [very thick, ->] (1,0) -- (1,13);
\draw [very thick, ->] (2,0) -- (2,13);
\draw [very thick, ->] (3,0) -- (3,13);
\draw [very thick, ->] (4,0) -- (4,13);
%% Rungs
\draw [very thick, directed=.65] (0,.75) -- (1,1.25);
\draw [very thick, directed=.65, red] (1,1.75) -- (2,2.25);
\draw [very thick, directed=.65, blue] (3,1.75) -- (4,2.25);
\draw [very thick, directed=.65, red] (2,2.75) -- (3,3.25);
\draw [very thick, directed=.65, red] (3,3.75) -- (4,4.25);
\draw [very thick, directed=.65, red] (2,4.75) -- (3,5.25);
\draw [very thick, directed=.65, red] (2,5.75) -- (1,6.25);
\draw [very thick, directed=.65, red] (3,6.75) -- (2,7.25);
\draw [very thick, directed=.65, red] (4,7.75) -- (3,8.25);
\draw [very thick, directed=.65, blue] (3,8.75) -- (4,9.25);
\draw [very thick, directed=.65, red] (3,9.75) -- (2,10.25);
\draw [very thick, directed=.65, red] (2,10.75) -- (1,11.25);
\draw [very thick, directed=.65, black] (1,11.75) -- (0,12.25);
\end{tikzpicture}
};
\endxy
\]

We'll now apply relations to express our web in terms of those with fewer trapped rungs, noting that when there are no trapped rungs, we can apply 
equation \eqref{eq:FEswitch_decat} as desired.
Begin by using slide relations to move away all untrapped rungs. 
Next, find a leftward oriented trapped rung which can be slid down so that there are no leftward oriented rungs 
anywhere between this rung and the rightward oriented rung in the chosen pair (\ie no leftward oriented rungs below it).
Sliding this rung as far down as possible, we must arrive at the configuration given on the left-hand side of equation \eqref{eq:FEswitch_decat}. 
Applying this relation, we obtain webs with either fewer trapped rungs, or with a leftward oriented rung which can be slid even further down. 
Repeating this procedure eventually allows us to slide the lowest leftward oriented rung out the bottom of the region (\ie it becomes untrapped). 
\end{proof}

Unfortunately, this process doesn't directly lift to the categorified case of $\hTr(\Foam{n}{})$.
Indeed, we can try to mimic the above argument to express any annular web closure as a graded direct sum of essential circles. 
However, the categorical analog of equation \eqref{eq:FEswitch_decat}:
\begin{equation}\label{eq:FEswitch_cat}
\xy
(0,0)*{
\begin{tikzpicture}[scale=.3]
	\draw [very thick, directed=.55] (-2,-4) to (-2,-2);
	\draw [very thick, directed=1] (-2,-2) to (-2,0.25);
	\draw [very thick, directed=.55] (2,-4) to (2,-2);
	\draw [very thick, directed=1] (2,-2) to (2,0.25);
	\draw [very thick, directed=.55] (-2,-2.5) to (2,-1.5);
	\draw [very thick] (-2,0.25) to (-2,2);
	\draw [very thick, directed=.55] (-2,2) to (-2,4);
	\draw [very thick] (2,0.25) to (2,2);
	\draw [very thick, directed=.55] (2,2) to (2,4);
	\draw [very thick, rdirected=.55] (-2,2.5) to (2,1.5);
	\node at (-2,-4.5) {\tiny $k$};
	\node at (2,-4.5) {\tiny $l$};
	\node at (-2.25,4.5) {\tiny $k{-}j_1{+}j_2$};
	\node at (2.25,4.5) {\tiny $l{+}j_1{-}j_2$};
	\node at (-3.5,0) {\tiny $k{-}j_1$};
	\node at (3.5,0) {\tiny $l{+}j_1$};
	\node at (0,-1.25) {\tiny $j_1$};
	\node at (0,2.75) {\tiny $j_2$};
\end{tikzpicture}
};
\endxy
\cong
\bigoplus_{j' = 0}^{\min(j_1,j_2)}
\bigoplus_{{k{-}j_1{-}l{+}j_2 \brack j^{\prime}}}
\xy
(0,0)*{
\begin{tikzpicture}[scale=.3]
	\draw [very thick, directed=.55] (-2,-4) to (-2,-2);
	\draw [very thick, directed=1] (-2,-2) to (-2,0.25);
	\draw [very thick, directed=.55] (2,-4) to (2,-2);
	\draw [very thick, directed=1] (2,-2) to (2,0.25);
	\draw [very thick, rdirected=.55] (-2,-1.5) to (2,-2.5);
	\draw [very thick] (-2,0.25) to (-2,2);
	\draw [very thick, directed=.55] (-2,2) to (-2,4);
	\draw [very thick] (2,0.25) to (2,2);
	\draw [very thick, directed=.55] (2,2) to (2,4);
	\draw [very thick, directed=.55] (-2,1.5) to (2,2.5);
	\node at (-2,-4.5) {\tiny $k$};
	\node at (2,-4.5) {\tiny $l$};
	\node at (-2.25,4.5) {\tiny $k{-}j_1{+}j_2$};
	\node at (2.25,4.5) {\tiny $l{+}j_1{-}j_2$};
	\node at (-4.25,0) {\tiny $k{+}j_2{-}j^{\prime}$};
	\node at (4.25,0) {\tiny $l{-}j_2{+}j^{\prime}$};
	\node at (0,-1.25) {\tiny $j_2{-}j^{\prime}$};
	\node at (0,2.75) {\tiny $j_1{-}j^{\prime}$};
\end{tikzpicture}
};
\endxy
\end{equation}
holds in $\Foam{n}{}$ only when $k-l \geq j_1-j_2$, since otherwise the right-hand side of equation \eqref{eq:FEswitch_decat} contains negative numbers.
When this inequality doesn't hold, we cannot replace the web on the left-hand side with a direct sum of webs of the form on the 
right\footnote{The reason is essentially that $F^{(b)}E^{(a)}1_\l$ is not a canonical basis element in $\U_q(\slnn{2})$ when $\l < b-a$.}. 
Instead, the web appears as a term on the right-hand side of the isomorphism
\[
\xy
(0,0)*{
\begin{tikzpicture}[scale=.3]
	\draw [very thick, directed=.55] (-2,-4) to (-2,-2);
	\draw [very thick, directed=1] (-2,-2) to (-2,0.25);
	\draw [very thick, directed=.55] (2,-4) to (2,-2);
	\draw [very thick, directed=1] (2,-2) to (2,0.25);
	\draw [very thick, directed=.55] (2,-2.5) to (-2,-1.5);
	\draw [very thick] (-2,0.25) to (-2,2);
	\draw [very thick, directed=.55] (-2,2) to (-2,4);
	\draw [very thick] (2,0.25) to (2,2);
	\draw [very thick, directed=.55] (2,2) to (2,4);
	\draw [very thick, rdirected=.55] (2,2.5) to (-2,1.5);
	\node at (-2,-4.5) {\tiny $k$};
	\node at (2,-4.5) {\tiny $l$};
	\node at (-2.25,4.5) {\tiny $k{+}j_1{-}j_2$};
	\node at (2.25,4.5) {\tiny $l{-}j_1{+}j_2$};
	\node at (-3.5,0) {\tiny $k{+}j_1$};
	\node at (3.5,0) {\tiny $l{-}j_1$};
	\node at (0,-1.25) {\tiny $j_1$};
	\node at (0,2.75) {\tiny $j_2$};
\end{tikzpicture}
};
\endxy
\cong
\bigoplus_{j' = 0}^{\min(j_1,j_2)}
\bigoplus_{{l{-}j_1{-}k{+}j_2 \brack j^{\prime}}}
\xy
(0,0)*{
\begin{tikzpicture}[scale=.3]
	\draw [very thick, directed=.55] (-2,-4) to (-2,-2);
	\draw [very thick, directed=1] (-2,-2) to (-2,0.25);
	\draw [very thick, directed=.55] (2,-4) to (2,-2);
	\draw [very thick, directed=1] (2,-2) to (2,0.25);
	\draw [very thick, rdirected=.55] (2,-1.5) to (-2,-2.5);
	\draw [very thick] (-2,0.25) to (-2,2);
	\draw [very thick, directed=.55] (-2,2) to (-2,4);
	\draw [very thick] (2,0.25) to (2,2);
	\draw [very thick, directed=.55] (2,2) to (2,4);
	\draw [very thick, directed=.55] (2,1.5) to (-2,2.5);
	\node at (-2,-4.5) {\tiny $k$};
	\node at (2,-4.5) {\tiny $l$};
	\node at (-2.25,4.5) {\tiny $k{+}j_1{-}j_2$};
	\node at (2.25,4.5) {\tiny $l{-}j_1{+}j_2$};
	\node at (-4.25,0) {\tiny $k{-}j_2{+}j^{\prime}$};
	\node at (4.25,0) {\tiny $l{+}j_2{-}j^{\prime}$};
	\node at (0,-1.25) {\tiny $j_2{-}j^{\prime}$};
	\node at (0,2.75) {\tiny $j_1{-}j^{\prime}$};
\end{tikzpicture}
};
\endxy .
\]
Nevertheless, we'll see that we can express the web on the left-hand side of equation \eqref{eq:FEswitch_cat} in terms of those on the right-hand side 
in the homotopy category, which suffices to deduce Proposition \ref{prop:equivHomCat}.

\begin{proof}[Proof of Proposition \ref{prop:equivHomCat}]
Given a complex in $\Kom(\hTr(\Foam{n}{}))$, we'd like to apply the above annular evaluation algorithm 
to each term in the complex, utilizing the categorical analogs of equations \eqref{easywebmoves} and \eqref{eq:FEswitch_decat}. 
However, we need a replacement for the isomorphism in equation \eqref{eq:FEswitch_cat} when $k-l < j_1-j_2$.

In this case, we can use the equality 
$
{a \brack b} = (-1)^b {-a+b-1 \brack b}
$
from \cite[1.3]{Lus4} to rearrange equation \eqref{eq:FEswitch_decat} to give the relation
\begin{align*}
\xy
(0,0)*{
\begin{tikzpicture}[scale=.3]
	\draw [very thick, directed=.55] (-2,-4) to (-2,-2);
	\draw [very thick, directed=1] (-2,-2) to (-2,0.25);
	\draw [very thick, directed=.55] (2,-4) to (2,-2);
	\draw [very thick, directed=1] (2,-2) to (2,0.25);
	\draw [very thick, directed=.55] (-2,-2.5) to (2,-1.5);
	\draw [very thick] (-2,0.25) to (-2,2);
	\draw [very thick, directed=.55] (-2,2) to (-2,4);
	\draw [very thick] (2,0.25) to (2,2);
	\draw [very thick, directed=.55] (2,2) to (2,4);
	\draw [very thick, rdirected=.55] (-2,2.5) to (2,1.5);
	\node at (-2,-4.5) {\tiny $k$};
	\node at (2,-4.5) {\tiny $l$};
	\node at (-2.25,4.5) {\tiny $k{-}j_1{+}j_2$};
	\node at (2.25,4.5) {\tiny $l{+}j_1{-}j_2$};
	\node at (-3.5,0) {\tiny $k{-}j_1$};
	\node at (3.5,0) {\tiny $l{+}j_1$};
	\node at (0,-1.25) {\tiny $j_1$};
	\node at (0,2.75) {\tiny $j_2$};
\end{tikzpicture}
};
\endxy
+
\sum_{\substack{j^{\prime}=1 \\ j^{\prime} \text{odd}}}^{\min(j_1,j_2)}
& {l-k+j_1-j_2+j^{\prime} -1 \brack j^{\prime}}
\xy
(0,0)*{
\begin{tikzpicture}[scale=.3]
	\draw [very thick, directed=.55] (-2,-4) to (-2,-2);
	\draw [very thick, directed=1] (-2,-2) to (-2,0.25);
	\draw [very thick, directed=.55] (2,-4) to (2,-2);
	\draw [very thick, directed=1] (2,-2) to (2,0.25);
	\draw [very thick, rdirected=.55] (-2,-1.5) to (2,-2.5);
	\draw [very thick] (-2,0.25) to (-2,2);
	\draw [very thick, directed=.55] (-2,2) to (-2,4);
	\draw [very thick] (2,0.25) to (2,2);
	\draw [very thick, directed=.55] (2,2) to (2,4);
	\draw [very thick, directed=.55] (-2,1.5) to (2,2.5);
	\node at (-2,-4.5) {\tiny $k$};
	\node at (2,-4.5) {\tiny $l$};
	\node at (-2.25,4.5) {\tiny $k{-}j_1{+}j_2$};
	\node at (2.25,4.5) {\tiny $l{+}j_1{-}j_2$};
	\node at (-4.25,0) {\tiny $k{+}j_2{-}j^{\prime}$};
	\node at (4.25,0) {\tiny $l{-}j_2{+}j^{\prime}$};
	\node at (0,-1.25) {\tiny $j_2{-}j^{\prime}$};
	\node at (0,2.75) {\tiny $j_1{-}j^{\prime}$};
\end{tikzpicture}
};
\endxy \\
&=
\sum_{\substack{j^{\prime}=0 \\ j^{\prime} \text{even}}}^{\min(j_1,j_2)}
{l-k+j_1-j_2+j^{\prime} -1 \brack j^{\prime}}
\xy
(0,0)*{
\begin{tikzpicture}[scale=.3]
	\draw [very thick, directed=.55] (-2,-4) to (-2,-2);
	\draw [very thick, directed=1] (-2,-2) to (-2,0.25);
	\draw [very thick, directed=.55] (2,-4) to (2,-2);
	\draw [very thick, directed=1] (2,-2) to (2,0.25);
	\draw [very thick, rdirected=.55] (-2,-1.5) to (2,-2.5);
	\draw [very thick] (-2,0.25) to (-2,2);
	\draw [very thick, directed=.55] (-2,2) to (-2,4);
	\draw [very thick] (2,0.25) to (2,2);
	\draw [very thick, directed=.55] (2,2) to (2,4);
	\draw [very thick, directed=.55] (-2,1.5) to (2,2.5);
	\node at (-2,-4.5) {\tiny $k$};
	\node at (2,-4.5) {\tiny $l$};
	\node at (-2.25,4.5) {\tiny $k{-}j_1{+}j_2$};
	\node at (2.25,4.5) {\tiny $l{+}j_1{-}j_2$};
	\node at (-4.25,0) {\tiny $k{+}j_2{-}j^{\prime}$};
	\node at (4.25,0) {\tiny $l{-}j_2{+}j^{\prime}$};
	\node at (0,-1.25) {\tiny $j_2{-}j^{\prime}$};
	\node at (0,2.75) {\tiny $j_1{-}j^{\prime}$};
\end{tikzpicture}
};
\endxy
\end{align*}
where all the coefficients appearing are positive. 
It follows that the categorical analog of this equation holds as well, 
so we have an isomorphism
\begin{equation}\label{eq:FEswitch_Ho}
\begin{aligned}
\xy
(0,0)*{
\begin{tikzpicture}[scale=.3]
	\draw [very thick, directed=.55] (-2,-4) to (-2,-2);
	\draw [very thick, directed=1] (-2,-2) to (-2,0.25);
	\draw [very thick, directed=.55] (2,-4) to (2,-2);
	\draw [very thick, directed=1] (2,-2) to (2,0.25);
	\draw [very thick, directed=.55] (-2,-2.5) to (2,-1.5);
	\draw [very thick] (-2,0.25) to (-2,2);
	\draw [very thick, directed=.55] (-2,2) to (-2,4);
	\draw [very thick] (2,0.25) to (2,2);
	\draw [very thick, directed=.55] (2,2) to (2,4);
	\draw [very thick, rdirected=.55] (-2,2.5) to (2,1.5);
	\node at (-2,-4.5) {\tiny $k$};
	\node at (2,-4.5) {\tiny $l$};
	\node at (-2.25,4.5) {\tiny $k{-}j_1{+}j_2$};
	\node at (2.25,4.5) {\tiny $l{+}j_1{-}j_2$};
	\node at (-3.5,0) {\tiny $k{-}j_1$};
	\node at (3.5,0) {\tiny $l{+}j_1$};
	\node at (0,-1.25) {\tiny $j_1$};
	\node at (0,2.75) {\tiny $j_2$};
\end{tikzpicture}
};
\endxy
\oplus
& \left(
\bigoplus_{\substack{j^{\prime}=1 \\ j^{\prime} \text{odd}}}^{\min(j_1,j_2)}
\bigoplus_{{l-k+j_1-j_2+j^{\prime} -1 \brack j^{\prime}}} 
\xy
(0,0)*{
\begin{tikzpicture}[scale=.3]
	\draw [very thick, directed=.55] (-2,-4) to (-2,-2);
	\draw [very thick, directed=1] (-2,-2) to (-2,0.25);
	\draw [very thick, directed=.55] (2,-4) to (2,-2);
	\draw [very thick, directed=1] (2,-2) to (2,0.25);
	\draw [very thick, rdirected=.55] (-2,-1.5) to (2,-2.5);
	\draw [very thick] (-2,0.25) to (-2,2);
	\draw [very thick, directed=.55] (-2,2) to (-2,4);
	\draw [very thick] (2,0.25) to (2,2);
	\draw [very thick, directed=.55] (2,2) to (2,4);
	\draw [very thick, directed=.55] (-2,1.5) to (2,2.5);
	\node at (-2,-4.5) {\tiny $k$};
	\node at (2,-4.5) {\tiny $l$};
	\node at (-2.25,4.5) {\tiny $k{-}j_1{+}j_2$};
	\node at (2.25,4.5) {\tiny $l{+}j_1{-}j_2$};
	\node at (-4.25,0) {\tiny $k{+}j_2{-}j^{\prime}$};
	\node at (4.25,0) {\tiny $l{-}j_2{+}j^{\prime}$};
	\node at (0,-1.25) {\tiny $j_2{-}j^{\prime}$};
	\node at (0,2.75) {\tiny $j_1{-}j^{\prime}$};
\end{tikzpicture}
};
\endxy
\right) \\
& \cong
 \bigoplus_{\substack{j^{\prime}=0 \\ j^{\prime} \text{even}}}^{\min(j_1,j_2)}
\bigoplus_{{l-k+j_1-j_2+j^{\prime} -1 \brack j^{\prime}}}
\xy
(0,0)*{
\begin{tikzpicture}[scale=.3]
	\draw [very thick, directed=.55] (-2,-4) to (-2,-2);
	\draw [very thick, directed=1] (-2,-2) to (-2,0.25);
	\draw [very thick, directed=.55] (2,-4) to (2,-2);
	\draw [very thick, directed=1] (2,-2) to (2,0.25);
	\draw [very thick, rdirected=.55] (-2,-1.5) to (2,-2.5);
	\draw [very thick] (-2,0.25) to (-2,2);
	\draw [very thick, directed=.55] (-2,2) to (-2,4);
	\draw [very thick] (2,0.25) to (2,2);
	\draw [very thick, directed=.55] (2,2) to (2,4);
	\draw [very thick, directed=.55] (-2,1.5) to (2,2.5);
	\node at (-2,-4.5) {\tiny $k$};
	\node at (2,-4.5) {\tiny $l$};
	\node at (-2.25,4.5) {\tiny $k{-}j_1{+}j_2$};
	\node at (2.25,4.5) {\tiny $l{+}j_1{-}j_2$};
	\node at (-4.25,0) {\tiny $k{+}j_2{-}j^{\prime}$};
	\node at (4.25,0) {\tiny $l{-}j_2{+}j^{\prime}$};
	\node at (0,-1.25) {\tiny $j_2{-}j^{\prime}$};
	\node at (0,2.75) {\tiny $j_1{-}j^{\prime}$};
\end{tikzpicture}
};
\endxy
\end{aligned}
\end{equation}
Next note that if we have an isomorphism $A \oplus B \cong C$ between objects in an additive category, 
then any complex of the form
\[
\cdots \to A^{i-1} \to A^i \oplus A \to A^{i+1} \to \cdots
\]
is homotopy equivalent to a complex of the form
\[
\cdots \to A^{i-1} \oplus B \to A^i \oplus C \to A^{i+1} \to \cdots
\]
Given this, we can perform the annular evaluation algorithm as in the decategorified case, 
at each step using either equation \eqref{eq:FEswitch_cat} to replace a term in the complex with a direct sum of terms in the same 
homological degree, or equation \eqref{eq:FEswitch_Ho} to replace a term with two direct sums of terms differing by one in homological degree.
\end{proof}

Again, all steps in the above proof actually lift to the quotient $\Ucat_Q(\glm)^{0\leq}$ of categorified quantum $\glm$, so we have in fact proven the following.

\begin{cor} \label{cor:equivHomCat}
There is an equivalence of categories $\Kom(\hTr(\Ucat_Q(\glm)^{0\leq})) \cong \Kom(\grvTr(\Ucat_Q(\glm)^{0\leq}))$.
\end{cor}

In fact, we conjecture a stronger result. Although our categorical annular evaluation algorithm given in Proposition \ref{prop:equivHomCat} shows that every annular web 
is homotopy equivalent in $\Kom(\hTr(\Foam{n}{}))$ to a complex\footnote{Using a variation on the last step of the proof of Proposition \ref{prop:equivHomCat}, we can moreover show that 
this complex is concentrated in two adjacent homological degrees.}
in $\Kom(\grvTr(\Foam{n}{}))$, in practice we find that in all examples we've computed the complex is concentrated in homological degree zero.
Equivalently, in examples we find that any annular web closure can be expressed at the decategorified level as a $\N[q,q^{-1}]$-linear sum of nested circles, 
which suggests the following.

\begin{conj}
There is an equivalence of categories $\hTr(\Ucat_Q(\glm)^{0\leq}) \cong \grvTr(\Ucat_Q(\glm)^{0\leq})$, and in particular, an equivalence of 
categories $\hTr(\Foam{n}{}) \cong \grvTr(\Foam{n}{})$.
\end{conj}

We do not need this stronger result at the moment. Indeed, Proposition \ref{prop:equivHomCat} allows us to accomplish step \eqref{step2}, 
replacing the complex $C_n(\cal{L})$ in $\hTr(\Foam{n}{})$ assigned to a colored annular link with a homotopy equivalent complex 
$\tilde{C}_n(\cal{L})$ in $\grvTr(\Foam{n}{})$.

\begin{defn}
Let $\cal{L}$ be a colored, framed, annular link. The \emph{sutured annular Khovanov-Rozansky homology} of $\cal{L}$, denoted $\saKhR(\cal{L})$, 
is the homology of the complex $\SH(\tilde{C}_n(\cal{L}))$.
\end{defn}

That $\saKhR(\cal{L})$ is an invariant of colored, framed, annular links follows easily from the fact that applying a 
Reidemeister move\footnote{All links are considered framed, hence we work with the version of the Reidemeister I move for framed links.}
to $\cal{L}$ only changes $C_n(\cal{L})$ up to homotopy equivalence and tensoring with $n$-labeled essential circles, and the choice of $\tilde{C}_n(\cal{L})$ 
is unique up to homotopy equivalence. Neither of these affect the homology of $\SH(\tilde{C}_n(\cal{L}))$, since in particular $n$-labeled circles are mapped to the 
trivial $\sln$ representation.

\begin{rem}
In recent work, Cherednik and Elliot define composite DAHA-superpolynomials of torus knots, and suggest that these invariants should be 
related to HOMFLY-PT homology in the thickened annulus \cite{CE}. It is an interesting problem to relate their invariants to our annular $\sln$ link homology.
\end{rem}

Since $\SH(\tilde{C}_n(\cal{L}))$ is a complex whose terms are ($q$-graded) $\sln$ representations, and the differentials are degree-zero intertwiners, 
the following generalization of Grigsby-Licata-Wehrli's result is again immediate.

\begin{cor}
For any colored, framed, annular link $\cal{L}$, $\saKhR(\cal{L})$ carries an action of $\sln$.
\end{cor}

\subsection{Decategorification}%####################################################################################################

In \cite{APS}, Asaeda, Przytycki, and Sikora show that $\saKh(\cal{L})$ categorifies the Kauffman skein module of $\cal{A}$, 
although their proof is somewhat subtle as there are \emph{a priori} too many elements in their decategorification. 
It is known that the Kauffman bracket skein module of $\cal{A}$ is isomorphic to $\Z[q,q^{-1}] \otimes_{\Z} \cal{R}_2$, 
where $\cal{R}_n \cong \mathrm{K}_0(\Rep(\sln))$ is the {\it representation ring} of $\sln$. 
This in turn suggests that we should utilize the $\sln$ action on $\saKhR$ in our decategorification.

To begin, we first analyze the $\sln$ analog of the Kauffman bracket skein module of the annulus. 
Let $\AWeb$ denote the $\Z[q,q^{-1}]$-module of Cautis-Kamnitzer-Morrison's $\sln$ webs embedded in the annulus. 
The identification of the category of $\sln$ webs with the subcategory of $\Rep(U_q(\sln))$ generated by fundamental representations in \cite{CKM}
shows that $\AWeb \cong \vTr(\Rep(U_q(\sln)))$. Note also that $\AWeb$ is in fact an algebra, with multiplication induced by the 
tensor product of $\Rep(U_q(\sln))$, \ie given by annular gluing.

\begin{lem}\label{lem:AWeb}
$\AWeb \cong \Z[q,q^{-1}] \otimes_\Z \cal{R}_n$
\end{lem}

\begin{proof}
The set of finite-dimensional irreducible quantum $\sln$ representations gives an upper-triangular basis for $\Rep(U_q(\sln))$, in the sense of \cite[Section 4.3]{BHLZ}. 
It then follows that $\AWeb \cong \vTr(\Rep(U_q(\sln)))$ is isomorphic to the free $\Z[q,q^{-1}]$-module spanned by the finite-dimensional 
irreducible (quantum) $\sln$ representations. The latter is precisely $\Z[q,q^{-1}] \otimes_\Z \cal{R}_n$. Moreover, the algebra structures on $\AWeb$ 
and $\Z[q,q^{-1}] \otimes_\Z \cal{R}_n$ coincide since both are induced by tensor product in $\Rep(U_q(\sln))$.
\end{proof}

Given a colored, framed, annular link, we can consider both $\saKhR(\cal{L})$ and the invariant $[\cal{L}]_n$ in $\AWeb$ determined by 
the shifts of \cite[Corollary 6.2.3]{CKM} giving the decategorification of the complexes from \cite{QR}. 
We can decategorify $\saKhR(\cal{L})$ by taking the formal alternating sum of homology groups, and taking into account the $q$-grading on $\grRep(\sln)$, 
this gives a formal $\Z[q,q^{-1}] $-linear combination of $\sln$ representations.

\begin{prop}
Viewing both as elements of $\Z[q,q^{-1}] \otimes_\Z \cal{R}_n$, we have $\mathrm{K}_0(\saKhR(\cal{L})) = [\cal{L}]_n$.
\end{prop}

\begin{proof}
We have the diagram
\begin{equation}\label{eq:decatsquare}
\begin{gathered}
\xymatrix{
\Kom(\hTr(\Foam{n}{})) \ar[rr] \ar[d] & & \Kom(\grRep(\sln)) \ar[d] \\
\mathrm{K}_0(\hTr(\Foam{n}{})) \ar[rr] & & \Z[q,q^{-1}] \otimes_{\Z} \cal{R}_n
} 
\end{gathered}
\end{equation}
where the vertical maps are given by taking alternating sums of terms in complexes, 
and we use \eg \cite{Rose2} to identify $\mathrm{K}_0(\mathrm{Kom}(\hTr(\Foam{n}{}))) \cong \mathrm{K}_0(\hTr(\Foam{n}{}))$. 

Under inspection, we find that $\mathrm{K}_0(\hTr(\Foam{n}{}))$ is precisely the $\Z[q,q^{-1}]$-module of clockwise oriented 
webs, in which we've retained the $n$-labeled edges\footnote{These are (annular) quantum $\gln$ webs.}. 
The bottom map is given using Lemma \ref{lem:annEvalAlgo} and then identifying $k$-labeled essential circles with the class of $\wedge^k \C^n$.
Equivalently, this is given as the composition of the map $\mathrm{K}_0(\hTr(\Foam{n}{})) \to \AWeb$, 
given by cutting $n$-labeled edges to tags as on page $6$ of \cite{CKM}, 
with the isomorphism from Lemma \ref{lem:AWeb}.

Furthermore, the image of $C_n(\cal{L})$ in $\AWeb$ under the composition
\[
\mathrm{Kom}(\hTr(\Foam{n}{})) \rightarrow  \mathrm{K}_0(\hTr(\Foam{n}{})) \rightarrow \AWeb
\]
precisely gives $[\cal{L}_n]$, using the identification from \cite[Theorem 5.3.1]{CKM}, and the related process from \cite{QR} which presents any tangle diagram in ladder form. 
The result then follows using the commutativity of the square in \eqref{eq:decatsquare}.
\end{proof}

%-------------------------------------------------------------------------------------------------------------------------------------------------------------
%
\subsection{A computation}%####################################################################################################
%
%-------------------------------------------------------------------------------------------------------------------------------------------------------------

We'll now perform a sample computation of $\saKhR(\cal{L})$ where 
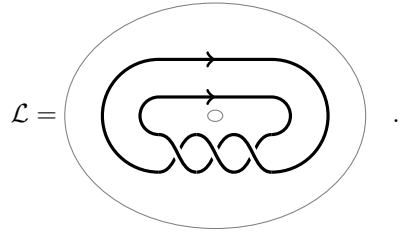
\[
\cal{L} = 
\xy
(0,0)*{
\begin{tikzpicture}[scale=.5, rotate=90]
%%Annulus
\draw[gray] (1,0) circle (3 and 4);
\draw[gray] (1,0) circle (.15 and .2);
%% Braid
% Lower crossing
\draw [very thick] (.5,-1.5) .. controls (.5,-1) and (-.5,-1) .. (-.5,-.5);
\draw [draw =white, double=black, very thick, double distance=1.25pt] (-.5,-1.5) .. controls (-.5,-1) and (.5,-1) .. (.5,-.5);
% Middle crossing
\draw [very thick] (.5,-.5) .. controls (.5,0) and (-.5,0) .. (-.5,.5);
\draw [draw =white, double=black, very thick, double distance=1.25pt] (-.5,-.5) .. controls (-.5,0) and (.5,0) .. (.5,.5);
% Top crossing
\draw [very thick] (.5,.5) .. controls (.5,1) and (-.5,1) .. (-.5,1.5);
\draw [draw =white, double=black, very thick, double distance=1.25pt] (-.5,.5) .. controls (-.5,1) and (.5,1) .. (.5,1.5);
%% Closure
\draw[very thick, directed=.5] (-.5,1.5) to [out=90,in=180] (1,3) to [out=0,in=90] (2.5,1.5) to (2.5,-1.5) to [out=270,in=0] (1,-3) to [out=180,in=270] (-.5,-1.5);
\draw[very thick, directed=.5] (.5,1.5) to [out=90,in=180] (1,2) to [out=0,in=90] (1.5,1.5) to (1.5,-1.5) to [out=270,in=0] (1,-2) to [out=180,in=270] (.5,-1.5);
\end{tikzpicture}
};
\endxy \quad .
\]
This annular knot is the closure of the braid $\cal{B} = \xy
(0,0)*{
\begin{tikzpicture}[scale=.5, rotate=90]
%% Braid
% Lower crossing
\draw [very thick] (.5,-1.5) .. controls (.5,-1) and (-.5,-1) .. (-.5,-.5);
\draw [draw =white, double=black, very thick, double distance=1.25pt] (-.5,-1.5) .. controls (-.5,-1) and (.5,-1) .. (.5,-.5);
% Middle crossing
\draw [very thick] (.5,-.5) .. controls (.5,0) and (-.5,0) .. (-.5,.5);
\draw [draw =white, double=black, very thick, double distance=1.25pt] (-.5,-.5) .. controls (-.5,0) and (.5,0) .. (.5,.5);
% Top crossing
\draw [very thick] (.5,.5) .. controls (.5,1) and (-.5,1) .. (-.5,1.5);
\draw[very thick, ->] (-.5,1.5) to (-.5,1.75);
\draw [draw =white, double=black, very thick, double distance=1.25pt] (-.5,.5) .. controls (-.5,1) and (.5,1) .. (.5,1.5);
\draw[very thick, ->] (.5,1.5) to (.5,1.75);
\end{tikzpicture}
};
\endxy$
and the complex $\llbracket \cal{B} \rrbracket_n$ assigned to $\cal{B}$ in $\Foam{n}{}$ lifts to the complex 
\[
\xymatrix{
q^{-3} \onell{[1,1]} \ar[r]^-{\begin{tikzpicture}[anchorbase, scale=.7] \draw [thick, <-] (0,0) .. controls (0,-.5) and (-.5,-.5) .. (-.5,0); \end{tikzpicture}}
& q^{-2}  \cal{F}\cal{E}\onell{[1,1]} \ar[rr]^-{\begin{tikzpicture}[anchorbase, scale=.7] \draw [thick,<-] (0,-.6) -- (0,0); \draw [thick, <-] (.5,0) -- (.5,-.6); \node at (0,-.3) {\tiny$\bullet$}; \end{tikzpicture}
\; - \; \begin{tikzpicture}[anchorbase,scale=.7] \draw [thick,<-] (0,-.6) -- (0,0); \draw [semithick, <-] (.5,0) -- (.5,-.6); \node at (.5,-.3) {\tiny$\bullet$}; \end{tikzpicture}}
& & \cal{F}\cal{E}\onell{[1,1]} \ar[r]^-{\begin{tikzpicture}[anchorbase, scale=.7] \draw [thick, <-] (0,0) .. controls (0,-.5) and (-.5,-.5) .. (-.5,0); \draw [thick, <-] (-.5,-1) .. controls (-.5,-.5) and (0,-.5) .. (0,-1); \end{tikzpicture}}
& \uwave{q^2 \cal{F}\cal{E}\onell{[1,1]}}
}
\]
in $\Ucat_Q(\glnn{2})^{0 \leq n}$, see \eg the (dual) computation in \cite[Section 10.2]{Cautis}. 

We can hence compute $\saKhR(\cal{L})$ by mapping this complex to $\mathrm{Kom}(\hTr(\Ucat_Q(\glnn{2})^{0 \leq n}))$, 
passing to an isomorphic complex in $\mathrm{Kom}(\grvTr(\Ucat_Q(\glnn{2})^{0 \leq n}))$, 
using $t=0$ skew Howe duality to map to a complex of (graded) $\sln$ modules, and finally computing homology.
Equations \eqref{eq:FEannularIso1} and \eqref{eq:FEannularIso2} again give an isomorphism 
$\cal{F}\cal{E}\onell{[1,1]} \cong q^{-1} \onenn{[2,0]} \oplus q \onenn{[2,0]}$ in $\hTr(\Ucat_Q(\glnn{2})^{0 \leq n})$,
so the above complex is isomorphic to the complex
\[
\xymatrix{
q^{-3} \onell{[1,1]} \ar[r]^-{A}
& q^{-3}  \onenn{[2,0]} \oplus q^{-1} \onenn{[2,0]} \ar[r]^-{B}
& q^{-1} \onenn{[2,0]} \oplus q \onenn{[2,0]} \ar[r]^-{C}
& \uwave{q \onenn{[2,0]} \oplus q^3 \onenn{[2,0]} }
}
\]
which lies in $\mathrm{Kom}(\grvTr(\Ucat_Q(\glnn{2})^{0 \leq n}))$. 
Using the identification with the current algebra of $\slnn{2}$, we compute that
\[
A = \begin{pmatrix} E1_{[2,0]} \\ E_1 1_{[2,0]} \end{pmatrix}
\;\; , \;\;
B = \begin{pmatrix} 0 & 0 \\ 0 & 0 \end{pmatrix}
\;\; , \;\;
C = \begin{pmatrix} E F_1 1_{[2,0]} - E F H_1 1_{[2,0]} & E F 1_{[2,0]} \\ E_1 F_1 1_{[2,0]} - E_1 FH_1 1_{[2,0]} & E_1 F 1_{[2,0]} \end{pmatrix}
\]
(here, we've dropped the first subscript on current algebra elements, since we work with $\slnn{2}$).
Taking the degree-zero part of the differential and applying skew Howe duality gives the complex
\[
\xymatrix{
q^{-3} \C^n \otimes \C^n \ar[r]^-{\left(\begin{smallmatrix} \pi \\ 0 \end{smallmatrix}\right)}
&  q^{-3}\wedge^2\C^n \oplus q^{-1} \wedge^2\C^n \ar[r]^-{\left(\begin{smallmatrix} 0 & 0 \\ 0 & 0 \end{smallmatrix}\right)}
& q^{-1}\wedge^2\C^n \oplus q^{1} \wedge^2\C^n \ar[r]^-{\left(\begin{smallmatrix} 0 & 2 \\ 0 & 0 \end{smallmatrix}\right)}
& \uwave{q^{1}\wedge^2\C^n \oplus q^{3} \wedge^2\C^n}
}
\]
so computing homology we find that
\[
\saKhR(\cal{L})^i = 
\begin{cases}
q^{-3} \mathrm{Sym}^2 \C^n & \text{if }  i=-3 \\
q^{-1} \wedge^2 \C^n & \text{if } i=-2 \\
q^{-1} \wedge^2 \C^n & \text{if } i=-1 \\
q^{3} \wedge^2 \C^n & \text{if } i=0 \\
0 & \text{else}. \\
\end{cases}
\]

\section{A spectral sequence to Khovanov-Rozansky homology}\label{KhR}%##############################################################################################

We now construct a spectral sequence from the sutured annular Khovanov-Rozansky homology of a colored, framed, annular link to 
the (traditional) Khovanov-Rozansky homology of the corresponding link in $S^3$, extending Robert's results in the $\slnn{2}$ case \cite{Rob}.
To do so, we'll describe a new procedure to extract the Khovanov-Rozansky homology of a 
balanced, colored, framed tangle closure from the 2-category $\Foam{n}{}^\bullet$
(or more precisely, its categorified quantum group analog, given in equation \eqref{eq:ndots} below).
Surprisingly, this approach relates current algebras to non-annular Khovanov-Rozansky homology. 

\begin{rem}
In the following, we must be especially careful with our gradings. 
Recall that the degree of a string diagram is computed locally in terms of generating morphisms, \eg
\[
\mathrm{deg}\left(
\xy
(0,0)*{
\begin{tikzpicture}[scale=.5]
\draw[thick,->] (0,0) to (0,1);
\node at (0,.5) {$\bullet$};
\end{tikzpicture}
};
\endxy
\right) = 2
\;\; , \;\;
\mathrm{deg} \left(
\;
\xy
(0,0)*{
\begin{tikzpicture}[scale=.5]
\draw[thick, <-] (0,-1) to [out=90,in=180] (.5,0) to [out=0,in=90] (1,-1);
\node at (1.25,-1) {\tiny$_i$};
\node at (1.25,.25) {\tiny$\mathbf{a}$};
\end{tikzpicture}
};
\endxy
\right) = 1+ a_i - a_{i+1}
\]
and the $q$-degree of a 2-morphism $q^{d_1} \cal{G} \xrightarrow{m_\mathsf{D}} q^{d_2} \cal{H}$ which is given 
by a string diagram $\mathsf{D}$ mapping between shifts of $1$-morphisms $\cal{G}$ and $\cal{H}$ is
\[
q\text{-}\mathrm{deg}(m_\mathsf{D}) = \mathrm{deg}(\mathsf{D}) - d_2 + d_1.
\]
We'll use $\HOM$ and $\END$ to denote the vector spaces of not-necessarily degree-zero morphisms.
\end{rem}

Let $\tau$ be a colored, framed tangle; we say that $\tau$ is \emph{balanced} if the colors and orientations of the left and right boundary points agree. 
Given such a tangle, we can consider the colored, framed link $\cal{L}_\tau \subset S^3$ obtained by taking its closure. 
In \cite{QR}, we showed that the Khovanov-Rozansky homology $\KhR(\cal{L})$ of a colored, framed link $\cal{L} \subset S^3$ can be computed from 
the complex $\llbracket \cal{L} \rrbracket_n$ assigned to $\cal{L}$ in $\Foam{n}{}^\bullet$ by applying the evaluation functor 
\[\HOM
\left(
\xy
(0,0)*{
\begin{tikzpicture}[scale=.4]
\draw[double] (2,1) to (0,1);
\draw[double] (2,0) to (0,0);
\node[scale=.75] at (1,.75) {$\vdots$};
\node at (2.375,1) {\small$_n$};
\node at (2.375,0) {\small$_n$};
\end{tikzpicture}
};
\endxy
, -
\right)
\]
from the identity tangle of a highest weight object to obtain a complex of graded $\C$-vector spaces, and computing homology. 
Equivalently, we can use equation \eqref{QRthm} to lift $\llbracket \cal{L} \rrbracket_n$ to a complex 
$\langle \cal{L} \rangle_n$ in the $2$-category\footnote{The additional quotient which kills $n$ dots on all strands specifies that we work 
with Khovanov-Rozansky homology, not its equivariant or deformed versions.} 
\begin{equation}\label{eq:ndots}
\Ucatc_Q(\glm)^{\bullet^n} := \Ucatc_Q(\glm)^{0 \leq n} \bigg/ 
\left(
\xy
(0,0)*{
\begin{tikzpicture}[scale=.5]
\draw[thick,->] (0,0) to (0,1);
\node at (0,.5) {$\bullet$};
\node at (.375,.625) {\tiny $n$};
\end{tikzpicture}
};
\endxy = 0
\right)
\end{equation}
for $m$ sufficiently large, 
apply $\HOM(\onenn{[n,\dots,n,0,\dots,0]},-)$ to obtain a complex of vector spaces, and take homology.

However, in the case of a tangle closure $\cal{L}_\tau$, we can actually compute $\KhR(\cal{L}_\tau)$ from the complex assigned to $\tau$ directly, avoiding the unnecessary closure step.
Indeed, we can again lift the complex $\llbracket \tau \rrbracket_n$ in $\Foam{n}{}$ to a complex $\langle \tau \rangle_n$ 
in\footnote{Again, for $m$ sufficiently large relative to the boundary (labelings) of our tangle. It suffices to take $m$ larger than $n$ multiplied by the number of boundary points, 
and we'll assume this condition for the duration.} 
$\Ucatc_Q(\glm)^{\bullet^n}$, 
and apply the functor $\HOM(\onenn{\mathbf{a}},-)$, with $\boldsymbol{a}=[a_1,\dots,a_k,0,\ldots,0]$ where $a_1,\ldots,a_k$ 
give the boundary labelings of the webs in  $\llbracket \tau \rrbracket_n$.
It is easy to see that the complexes $\HOM(\onenn{[n,\dots,n,0,\dots,0]},\langle \cal{L}_\tau \rangle_n)$ and $\HOM(\onenn{\mathbf{a}}, \langle \tau \rangle_n)$ are isomorphic,
up to an overall shift by $\sum_i a_i(n-a_i)$ in quantum degree.
We'll refer to both as the \emph{Khovanov-Rozansky complex}.

This setting is now actually very close to the annular one: the isomorphisms used to simplify the complex $C(\cal{L}_\tau)$ in $\hTr(\Foam{n}{})$ which slide ladder rungs around the annulus 
are analogous to the isomorphisms of vector spaces $\HOM(\onenn{\mathbf{b}},\cal{G} \onenn{\mathbf{c}} \cal{H} )\cong \HOM(\onenn{\boldsymbol{c}}, q^d \cal{H}\onenn{\boldsymbol{b}} \cal{G})$, 
for $1$-morphisms $\cal{G}, \cal{H} \in \Ucatc_Q(\glm)^{\bullet^n}$ given by
\begin{equation}\label{eq:slideunder}
\begin{tikzpicture}[anchorbase,decoration={brace}]
\draw[very thick] (0,1) -- (2.5,1) arc (0:-180:1.25);
\node at (1.25,.5) {$m$};
\node at (1.25,-.5) {$\mathbf{b}$};
\draw[thick] (.5,1) -- (.5,2);
\node at (.75,1.75) {\tiny \dots};
\draw[thick] (1,1) -- (1,2);
\draw[thick] (1.5,1) -- (1.5,2);
\node at (1.75,1.75) {\tiny \dots};
\draw[thick] (2,1) -- (2,2);
\node at (1.25,1.5) {$\mathbf{c}$};
\draw [decorate] (.4,2.2) -- (1.1,2.2);
\node at (.75,2.5) {\small $\cal{G}$};
\draw [decorate] (1.4,2.2) -- (2.2,2.2);
\node at (1.75,2.5) {\small $\cal{H}$};
\node at (1.25,-2) {};
\end{tikzpicture}
\qquad \mapsto \qquad
\begin{tikzpicture}[anchorbase,decoration={brace}]
\draw[very thick] (0,1) -- (2.5,1) arc (0:-180:1.25);
\node at (1.25,.5) {$m$};
\node at (1.25,-.5) {$\mathbf{b}$};
\draw[thick] (.5,1) .. controls (.5,1.25) and (-1,1.25) .. (-1,.5) .. controls (-1,-.5) and (0,-1)  .. (1.25,-1) .. controls (2.25,-1) and (3,-.5) .. (3,.5) .. controls (3,1.5) and (1.5,1.5) .. (1.5,2);
\node at (.75,1.1) {\tiny \dots};
\draw[thick] (1,1) .. controls (1,1.75) and (-1.5,1.75) .. (-1.5,.5) .. controls (-1.5,-.75) and (-.25,-1.5) .. (1.25,-1.5) .. controls (2.5,-1.5) and (3.5,-.75) .. (3.5,.5) .. controls (3.5,1.5) and (2,1.5) .. (2,2);
\draw[thick] (1.5,1) .. controls (1.5,1.25) and (.5,1.5) ..  (.5,2);
\node at (1.75,1.1) {\tiny \dots};
\draw[thick] (2,1) .. controls (2,1.25) and (1,1.5) .. (1,2);
\node at (1.25,-2) {$\mathbf{c}$};
\draw [decorate] (.4,2.2) -- (1.1,2.2);
\node at (.75,2.5) {\small $\cal{H}$};
\draw [decorate] (1.4,2.2) -- (2.1,2.2);
\node at (1.75,2.5) {\small $\cal{G}$};
\end{tikzpicture} \quad .
\end{equation}
Using these isomorphisms, together with isomorphisms between $1$-morphisms in $\Ucatc_Q(\glm)^{\bullet^n}$, 
we can mimic the proof of Proposition \ref{prop:equivHomCat} to obtain the following result.

\begin{thm}\label{thm:saKhRandKhR}
The Khovanov-Rozansky complex of a tangle closure $\cal{L}_\tau$ is homotopy equivalent to a complex in which each term takes the form 
$\bigoplus_{\mathbf{b},d} \HOM(\onenn{\boldsymbol{b}}, q^d \onenn{\boldsymbol{b}})$, 
and the differential is given by the action of the current algebra $\dcurm$ on the center of objects in $\Ucatc_Q(\glm)^{\bullet^n}$. 
This complex can be endowed with a filtration, and the associated graded complex is isomorphic to $\SH(\tilde{C}_n(\cal{L}_\tau))$.
\end{thm}

For details concerning the action of the vertical trace of a $2$-category $\cal{C}$ on the center of objects in a $2$-representation of $\cal{C}$, 
and in particular the action of $\U(\slm[t])$ on the center of objects in a $2$-representation of $\Ucat_Q(\slm)$, see \cite[Section 9]{BHLW}. 
Note that this action requires that the $2$-category $\cal{C}$ be cyclic.
Fortunately, the $2$-categories $\Ucat_Q(\slm)$ and $\Ucat_Q(\glm)$ are known to be isomorphic to $2$-categories which are 
cyclic (see \cite{BHLW2} and \cite{MSV2}), and we hence pass to such 
cyclic\footnote{In practice, this allows us to avoid keeping track of the signs which appear in the non-cyclic case when sliding a 2-morphism through caps and cups 
(and which don't appear when sliding such a 2-morphism around a cylinder in the trace).}
versions for the duration.
If the 2-representation of $\Ucat_Q(\slm)$ is presented diagrammatically, the action of the current algebra is then 
given by wrapping dotted circles around an element in the center of objects, \eg
\begin{equation}\label{eq:centeraction}
E_{i,r}: \;
\xy
(0,0)*{
\begin{tikzpicture}[scale=.5]
\draw[very thick] (0,0) circle (.5);
\node at (0,0) {$m$};
\end{tikzpicture}
};
\endxy
\; \mapsto \;
\xy
(0,0)*{
\begin{tikzpicture}[scale=.5]
\draw[very thick] (0,0) circle (.5);
\node at (0,0) {$m$};
\draw[blue, thick, rdirected=.5] (0,0) circle (1);
\node at (1,0) {$\bullet$};
\node at (1.25,.25) {\tiny$r$};
\end{tikzpicture}
};
\endxy
\;\; , \;\;
F_{i,r}: \;
\xy
(0,0)*{
\begin{tikzpicture}[scale=.5]
\draw[very thick] (0,0) circle (.5);
\node at (0,0) {$m$};
\end{tikzpicture}
};
\endxy
\; \mapsto \;
\xy
(0,0)*{
\begin{tikzpicture}[scale=.5]
\draw[very thick] (0,0) circle (.5);
\node at (0,0) {$m$};
\draw[blue, thick, directed=.5] (0,0) circle (1);
\node at (1,0) {$\bullet$};
\node at (1.25,.25) {\tiny$r$};
\end{tikzpicture}
};
\endxy \quad .
\end{equation}

\begin{proof}[Proof (of Theorem \ref{thm:saKhRandKhR})]
Consider the lift of the complex $\llbracket \tau \rrbracket_n$ in $\Foam{n}{}$ to the $n$-bounded quotient $\Ucatc_Q(\glm)^{0 \leq n}$, 
which we'll again (by slight abuse of notation) denote by $\langle \tau \rangle_n$. 
We'll compare the procedures from which we obtain $\KhR(\cal{L}_\tau)$ and $\saKhR(\cal{L}_\tau)$. 

To compute $\saKhR(\cal{L}_\tau)$ we consider $\langle \tau \rangle_n$ as a complex in $\hTr(\Ucatc_Q(\glm)^{0 \leq n})$,  
and we apply isomorphisms of the form
\begin{equation}\label{eq:slidearound}
\cal{E}^{(k)}_i \onell{\mathbf{c}} \cal{G} \onell{\mathbf{b}} \cong  \cal{G} \onell{\mathbf{b}} \cal{E}^{(k)}_i \onell{\mathbf{c}}
\;\; \text{and} \;\;
\cal{F}^{(k)}_i \onell{\mathbf{c}} \cal{G} \onell{\mathbf{b}} \cong  \cal{G} \onell{\mathbf{b}} \cal{F}^{(k)}_i \onell{\mathbf{c}}
\end{equation}
in $\hTr(\Ucatc_Q(\glm)^{0 \leq n})$ and isomorphisms coming from $\Ucatc_Q(\glm)^{0 \leq n}$ to use the (categorified quantum group version of the)
annular evaluation algorithm to obtain a complex $\tilde{\langle \tau \rangle}_n$ which is homotopy equivalent to $\langle \tau \rangle_n$, and whose 
terms take the form $\bigoplus_{\mathbf{b},l} q^{l} \onenn{\mathbf{b}}$. 
Such a complex necessarily lies in $\grvTr(\Ucat_Q(\glm)^{0 \leq n})$, 
so we can use $t=0$ skew Howe duality to map each $\onell{\mathbf{b}}$ to the $\sln$ representation 
$\wedge^{b_1}\C^n \otimes \cdots \otimes \wedge^{b_m} \C^n$ and to send the differentials to the maps of $\sln$ representations given by the 
skew Howe dual action of $\glm$. This recovers the complex $\SH(\tilde{C}_n(\cal{L}_\tau))$, whose homology is $\saKhR(\cal{L}_\tau)$.

Alternatively, we can first apply $\HOM(\onell{\mathbf{a}}, - )$ to $\langle \tau \rangle_n$ to obtain a complex which gives the Khovanov-Rozansky complex
after passing to the quotient in equation \eqref{eq:ndots}.
We can then use equation \eqref{eq:slideunder} to apply the isomorphisms
\[
\HOM(\onell{\mathbf{b}}, \cal{E}^{(k)}_i \onell{\mathbf{c}} \cal{G} \onell{\mathbf{b}}) \cong  \HOM(\onell{\mathbf{c}} , q^{2k(k+c_i-c_{i+1})} \cal{G} \onell{\mathbf{b}} \cal{E}^{(k)}_i \onell{\mathbf{c}})
\]
and
\[
\HOM(\onell{\mathbf{b}}, \cal{F}^{(k)}_i \onell{\mathbf{c}} \cal{G} \onell{\mathbf{b}}) \cong \HOM(\onell{\mathbf{c}} , q^{2k(k-c_i+c_{i+1})} \cal{G} \onell{\mathbf{b}} \cal{F}^{(k)}_i \onell{\mathbf{c}})
\]
(which are the analogs of those in equation \eqref{eq:slidearound}) 
and isomorphisms in $\Ucatc_Q(\glm)^{0 \leq n}$, mimicking the annular evaluation algorithm. 
In the end, we arrive at a complex homotopy equivalent to $\HOM(\onell{\mathbf{a}}, \langle \tau \rangle_n)$ and whose terms take the form 
$\bigoplus_{\mathbf{b},d} \HOM(\onenn{\mathbf{b}}, q^{d} \onenn{\mathbf{b}})$, with the number of summands, and the $\glm$ weights appearing in the terms, 
the same as those in $\tilde{\langle \tau \rangle}_n$. 
Denote the image of this complex in $\Ucatc_Q(\glm)^{\bullet^n}$ by $\widehat{\langle \tau \rangle}_n$.
Under careful inspection, we find that the differentials in $\widehat{\langle \tau \rangle}_n$ are given by precisely the same elements 
in the current algebra $\U(\glm[t])$ as in $\tilde{\langle \tau \rangle}_n$, acting here on the center of objects in $\Ucatc_Q(\glm)^{\bullet^n}$
via equation \eqref{eq:centeraction}.
This confirms the first statement of the result.

For the second statement, we first note that the chain groups in the Khovanov-Rozansky complex are isomorphic 
(as vector spaces) to those in the sutured annular Khovanov-Rozansky complex. 
Indeed, this follows since the algebra $\END(\onenn{\mathbf{b}})$ in $\Ucatc_Q(\glm)^{\bullet^n}$ is identified under 
equation \eqref{QRthm} with the algebra of endomorphisms of parallel strands in $\Foam{n}{}^\bullet$, labeled according to $\mathbf{b}$. 
It follows from \cite[Remark 4.1]{QR} that the latter is isomorphic to $H^*(Gr_{b_1}(\C^n)) \otimes \cdots \otimes H^*(Gr_{b_m}(\C^n))$, 
which in turn is isomorphic to $\wedge^{b_1}\C^n \otimes \cdots \otimes \wedge^{b_m} \C^n$
(using \eg the geometric Satake isomorphism, see \cite{CK02}), as desired.

It remains to define the filtration degree on the chain groups in $\widehat{\langle \tau \rangle}_n$, 
and to identify the degree-preserving differential with the differential in $\SH(\tilde{C}_n(\cal{L}_\tau))$.
Define the filtration degree on $\HOM(\onenn{\mathbf{b}}, q^{d} \onenn{\mathbf{b}})$ by
\begin{equation}\label{eq:tdeg}
t\text{-}\mathrm{deg}(m) = q\text{-}\mathrm{deg}(m) + d - \sum_{j=1}^m b_j (n-b_j)
\end{equation}
Note that the first two terms on the right-hand side of equation \eqref{eq:tdeg} combine to simply compute the degree of a string diagram. 
Using this, it follows easily that the action of the current algebra elements $E_{i,r}$ and $F_{i,r}$ (and hence $H_{i,r}$) via equation \eqref{eq:centeraction} raises $t$-degree by $2r$. 
For example, in the case of $1_\mathbf{b} E_{i,r} 1_\mathbf{c}$ this follows since the $r$ dots contribute $2r$ to the $t$-degree, 
and the clockwise $i$-labeled cap and cup contribute $2+2(b_{i+1}-b_i)$ which is exactly cancelled 
since $\sum_{j=1}^m c_j (n-c_j) - \sum_{j=1}^m b_j (n-b_j) = -2+2(b_{i}-b_{i+1})$.
Observe also that, up to a global shift, the $q$-degree of a summand of a chain group in $\SH(\tilde{C}_n(\cal{L}_\tau))$,
which differs from the $q$-degree in $\widehat{\langle \tau \rangle}_n$ since there is no shift on the right-hand side of equation \eqref{eq:slidearound},
can be computed by taking $t\text{-}\mathrm{deg}(m) - q\text{-}\mathrm{deg}(m)$ for any $m$ in the corresponding term in $\widehat{\langle \tau \rangle}_n$.

Defining $F^p=\{x \; | \; t\text{-}\mathrm{deg}(x)\geq p \}$, we have a decreasing filtration $\cdots \supset F^p \supset F^{p+1} \supset \cdots$ on $\widehat{\langle \tau \rangle}_n$
compatible with the differential. The $t$-degree preserving part of the differential is exactly the part of the differential in current algebra degree-zero, 
so it suffices to show that the action of these elements coming from equation \eqref{eq:centeraction} agrees with the one coming from skew Howe duality, 
after identifying $\END(\onenn{\mathbf{b}})$ with $\wedge^{b_1}\C^n \otimes \cdots \otimes \wedge^{b_m} \C^n$.
Indeed, these must agree since both cases give representations of $\slm$ (\ie the degree-zero part of the current algebra) 
with weight spaces of the same dimension.
\end{proof}

The filtration defined in the proof of Theorem \ref{thm:saKhRandKhR} is bounded, so \eg \cite[Theorem 2.6]{McCleary} shows that the corresponding 
spectral sequence converges to the associated graded of the homology of the complex. This in turn is isomorphic to the homology of the complex itself 
(working over a field). We hence obtain the following.

\begin{cor}\label{cor:SS}
Given a colored, framed, balanced tangle $\tau$, there exists a spectral sequence whose first page is the sutured annular Khovanov-Rozansky 
homology of the annular closure of $\tau$ and which converges to the Khovanov-Rozansky homology of $\cal{L}_\tau$.
\end{cor}

% ==============================================================================
% REFERENCES

\bibliographystyle{plain}
%\bibliography{bib-skew}

\begin{thebibliography}{10}

\bibitem{APS}
M.M. Asaeda, J.H. Przytycki, and A.S. Sikora.
\newblock Categorification of the {K}auffman bracket skein module of
  {I}-bundles over surfaces.
\newblock {\em Algebr. Geom. Topol}, 4(52):1177--1210, 2004.
\newblock \href{http://arxiv.org/abs/math/0409414}{arXiv:math/0409414}.

\bibitem{BN2}
D.~Bar-Natan.
\newblock Khovanov's homology for tangles and cobordisms.
\newblock {\em Geom. Topol.}, 9:1443--1499, 2005.
\newblock \href{http://arxiv.org/abs/math/0410495}{arXiv:math/0410495}.

\bibitem{BGHL}
A.~{Beliakova}, Z.~{Guliyev}, K.~{Habiro}, and A.~D. {Lauda}.
\newblock Trace as an alternative decategorification functor.
\newblock {\em Acta Mathematica Vietnamica}, 39(4):425--480, December 2014.
\newblock \href{http://arxiv.org/abs/1409.1198}{arXiv:1409.1198}.

\bibitem{BHLZ}
A.~{Beliakova}, K.~{Habiro}, A.~D. {Lauda}, and M.~{{\v Z}ivkovi{\'c}}.
\newblock {Trace decategorification of categorified quantum sl(2)}.
\newblock April 2014.
\newblock \href{http://arxiv.org/abs/1404.1806}{arXiv:1404.1806}.

\bibitem{BHLW}
A.~{Beliakova}, K.~{Habiro}, A.~D. {Lauda}, and B.~{Webster}.
\newblock {Current algebras and categorified quantum groups}.
\newblock 2014.
\newblock \href{http://arxiv.org/abs/1412.1417}{arXiv:1412.1417}.

\bibitem{BHLW2}
A.~{Beliakova}, K.~{Habiro}, A.~D. {Lauda}, and B.~{Webster}.
\newblock Cyclicity for categorified quantum groups.
\newblock 2015.
\newblock \href{http://arxiv.org/abs/1506.04671}{arXiv:1506.04671}.


\bibitem{Blan}
C.~Blanchet.
\newblock An oriented model for {K}hovanov homology.
\newblock {\em J. Knot Theory Ramifications}, 19(2):291--312, 2010.
\newblock \href{http://arxiv.org/abs/1405.7246}{arXiv:1405.7246}.

\bibitem{Cautis}
S.~Cautis.
\newblock Clasp technology to knot homology via the affine {G}rassmannian.
\newblock {\em Math. Ann.}, March 2015.
\newblock \href{http://arxiv.org/abs/1207.2074}{arXiv:1207.2074}.

\bibitem{CK02}
S.~Cautis and J.~Kamnitzer.
\newblock Knot homology via derived categories of coherent sheaves. {II}.
  {$\mf{sl}_m$} case.
\newblock {\em Invent. Math.}, 174(1):165--232, 2008.
\newblock \href{http://arxiv.org/abs/0710.3216}{arXiv:0710.3216}.

\bibitem{CKL}
S.~Cautis, J.~Kamnitzer, and A.~Licata.
\newblock Categorical geometric skew {H}owe duality.
\newblock {\em Invent. Math.}, 180(1):111--159, 2010.
\newblock \href{http://arxiv.org/abs/0902.1795}{arXiv:0902.1795}.

\bibitem{CKM}
S.~Cautis, J.~Kamnitzer, and S.~Morrison.
\newblock Webs and quantum skew howe duality.
\newblock {\em Math. Ann.}, 360(1-2):351--390, 2014.
\newblock \href{http://arxiv.org/abs/1210.6437}{arXiv:1210.6437}.

\bibitem{CLau}
S.~Cautis and A.~D. Lauda.
\newblock Implicit structure in 2-representations of quantum groups.
\newblock {\em Selecta Mathematica}, 21(1):201--244, 2015.
\newblock \href{http://arxiv.org/abs/1111.1431}{arXiv:1111.1431}.

\bibitem{CE}
I.~{Cherednik} and R.~{Elliot}.
\newblock Refined composite invariants of torus knots via {DAHA}.
\newblock 2015.
\newblock \href{http://arxiv.org/abs/1503.01441}{arXiv:1503.01441}.

\bibitem{CR}
J.~Chuang and R.~Rouquier.
\newblock Derived equivalences for symmetric groups and sl\_2-categorification.
\newblock {\em Ann. of Math.}, 167:245--298, 2008.
\newblock \href{http://arxiv.org/abs/math/0407205}{arXiv:math/0407205}.

\bibitem{CMW}
D.~Clark, S.~Morrison, and K.~Walker.
\newblock Fixing the functoriality of {K}hovanov homology.
\newblock {\em Geom. Topol.}, 13(3):1499--1582, 2009.
\newblock \href{http://arxiv.org/abs/math/0701339}{arXiv:math/0701339}.

\bibitem{GLW}
J.~E. {Grigsby}, A.~{Licata}, and S.~M. {Wehrli}.
\newblock Annular {K}hovanov homology and knotted {S}chur-{W}eyl
  representations.
\newblock 2015.
\newblock \href{http://arxiv.org/abs/1505.04386}{arXiv:1505.04386}.

\bibitem{GW}
J.~E. Grigsby and Stephan~M. Wehrli.
\newblock {K}hovanov homology, sutured {F}loer homology, and annular links.
\newblock {\em Algebraic \& Geometric Topology}, 10(4):2009--2039, 2010.
\newblock \href{http://arxiv.org/abs/0907.4375}{arXiv:0907.4375}.

\bibitem{Kh1}
M.~Khovanov.
\newblock A categorification of the {J}ones polynomial.
\newblock {\em Duke Math. J.}, 101(3):359--426, 2000.
\newblock \href{http://arxiv.org/abs/math/9908171}{arXiv:math/9908171}.

\bibitem{Kh2}
M.~Khovanov.
\newblock A functor-valued invariant of tangles.
\newblock {\em Algebr. Geom. Topol.}, 2:665--741 (electronic), 2002.
\newblock \href{http://arxiv.org/abs/math/0103190}{arXiv:math/0103190}.

\bibitem{Kh5}
M.~Khovanov.
\newblock sl(3) link homology.
\newblock {\em Algebr. Geom. Topol.}, 4:1045--1081, 2004.
\newblock \href{http://arxiv.org/abs/math/0304375}{arXiv:math/0304375}.

\bibitem{Kh6}
M.~Khovanov.
\newblock An invariant of tangle cobordisms.
\newblock {\em Trans. Am. Math. Soc.}, 358(1):315--328, 2006.
\newblock \href{http://arxiv.org/abs/math/0207264}{arXiv:math/0207264}.

\bibitem{KL}
M.~Khovanov and A.~Lauda.
\newblock A diagrammatic approach to categorification of quantum groups {I}.
\newblock {\em Represent. Theory}, 13:309--347, 2009.
\newblock \href{http://arxiv.org/abs/0803.4121}{arXiv:0803.4121}.

\bibitem{KL3}
M.~Khovanov and A.~Lauda.
\newblock A diagrammatic approach to categorification of quantum groups {III}.
\newblock {\em Quantum Topol.}, 1:1--92, 2010.
\newblock \href{http://arxiv.org/abs/0807.3250}{arXiv:0807.3250}.

\bibitem{KL2}
M.~Khovanov and A.~Lauda.
\newblock A diagrammatic approach to categorification of quantum groups {II}.
\newblock {\em Trans. Amer. Math. Soc.}, 363:2685--2700, 2011.
\newblock \href{http://arxiv.org/abs/0804.2080}{arXiv:0804.2080}.

\bibitem{KLMS}
M.~Khovanov, A.~Lauda, M.~Mackaay, and M.~{S}to{\v{s}i\'c}.
\newblock Extended graphical calculus for categorified quantum sl(2).
\newblock {\em Mem. Am. Math. Soc.}, 219, 2012.
\newblock \href{http://arxiv.org/abs/1006.2866}{arXiv:1006.2866}.

\bibitem{KhR}
M.~Khovanov and L.~Rozansky.
\newblock Matrix factorizations and link homology.
\newblock {\em Fund. Math.}, 199(1):1--91, 2008.
\newblock \href{http://arxiv.org/abs/math/0401268}{arXiv:math/0401268}.

\bibitem{Lau1}
A.~D. Lauda.
\newblock A categorification of quantum sl(2).
\newblock {\em Adv. Math.}, 225:3327--3424, 2008.
\newblock \href{http://arxiv.org/abs/0803.3652}{arXiv:0803.3652}.

\bibitem{LQR1}
A.~D. Lauda, H.~Queffelec, and D.~E.~V. Rose.
\newblock Khovanov homology is a skew howe 2-representation of categorified
  quantum $\mathfrak{sl}_m$.
\newblock {\em Algebr. Geom. Topol. (to appear)}, 2012.
\newblock \href{http://arxiv.org/abs/1212.6076}{arXiv:1212.6076}.

\bibitem{Lus4}
G.~Lusztig.
\newblock {\em Introduction to quantum groups}, volume 110 of {\em Progress in
  Mathematics}.
\newblock Birkh\"auser Boston Inc., Boston, MA, 1993.

\bibitem{MSV}
M.~Mackaay, M.~Sto{\v{s}}i{\'c}, and P.~Vaz.
\newblock {$\mf{sl}(N)$}-link homology {$(N\geq 4)$} using foams and the
  {K}apustin-{L}i formula.
\newblock {\em Geom. Topol.}, 13(2):1075--1128, 2009.
\newblock \href{http://arxiv.org/abs/0708.2228}{arXiv:0708.2228}.

\bibitem{MSV2}
M.~Mackaay, M.~Sto{\v{s}}i{\'c}, and P.~Vaz.
\newblock A diagrammatic categorification of the q-{S}chur algebra.
\newblock {\em Quantum Topol.}, 4(1):1--75, 2013.
\newblock \href{http://arxiv.org/abs/1008.1348}{arXiv:1008.1348}.

\bibitem{McCleary}
J.S McCleary.
\newblock {\em A user's guide to spectral sequences}.
\newblock Number~58. Cambridge University Press, 2001.

\bibitem{ORS}
P.~Ozsv\'{a}th, J.~Rasmussen, and Z.~Szab\'{o}.
\newblock Odd {K}hovanov homology.
\newblock {\em Algebraic \& Geometric Topology}, 13(3):1465--1488, 2013.
\newblock \href{http://arxiv.org/abs/0710.4300}{arXiv:0710.4300}.

\bibitem{QR}
H.~{Queffelec} and D.~E.~V. {Rose}.
\newblock {The $\mathfrak{sl}_n$ foam 2-category: a combinatorial formulation of
  Khovanov-Rozansky homology via categorical skew Howe duality}.
\newblock May 2014.
\newblock \href{http://arxiv.org/abs/1405.5920}{arXiv:1405.5920}.

\bibitem{QS}
H.~Queffelec and A.~Sartori.
\newblock Homfly-{P}t and {A}lexander polynomials from a doubled {S}chur
  algebra.
\newblock 2014.
\newblock \href{http://arxiv.org/abs/1412.3824}{arXiv:1412.3824}.

\bibitem{Rob}
L.~P. Roberts.
\newblock On knot {F}loer homology in double branched covers.
\newblock {\em Geometry \& Topology}, 17(1):413--467, 2013.
\newblock \href{http://arxiv.org/abs/0706.0741}{arXiv:0706.0741}.

\bibitem{Rose}
D.~E.~V. Rose.
\newblock A categorification of quantum $\mf{sl}_3$ projectors and the
  $\mf{sl}_3$ {R}eshetikhin-{T}uraev invariant of tangles.
\newblock {\em Quantum Topol.}, 5(1):1--59, 2014.
\newblock \href{http://arxiv.org/abs/1109.1745}{arXiv:1109.1745}.

\bibitem{Rose2}
D.~E.~V. {Rose}.
\newblock A note on the {G}rothendieck group of an additive category.
\newblock {\em Bulletin of {C}helyabinsk {S}tate {U}niversity}, 17(3):135--139,
  2015.
\newblock \href{http://arxiv.org/abs/1109.2040}{arXiv:1109.2040}.

\bibitem{RTub}
D.~E.~V. {Rose} and D.~{Tubbenhauer}.
\newblock Symmetric webs, {J}ones-{W}enzl recursions and $q$-{H}owe duality.
\newblock 2015.
\newblock \href{http://arxiv.org/abs/1501.00915}{arXiv:1501.00915}.

\bibitem{Rou2}
R.~Rouquier.
\newblock 2-{K}ac-{M}oody algebras.
\newblock 2008.
\newblock \href{http://arxiv.org/abs/0812.5023}{arXiv:0812.5023}.

\bibitem{Roz}
L.~Rozansky.
\newblock An infinite torus braid yields a categorified {J}ones-{W}enzl
  projector.
\newblock {\em Fund. Math.}, pages 305--326, 2014.
\newblock \href{http://arxiv.org/abs/1005.3266}{arXiv:1005.3266}.

\end{thebibliography}

%
% ==============================================================================

% ==============================================================================
%
\end{document}